\documentclass{amsart}
\usepackage{amsmath,amsthm,amssymb}
\usepackage{hyperref,graphicx,color,float}
\usepackage{a4wide}
\usepackage{mathtools}

\allowdisplaybreaks %

\title{Hook length property of $d$-complete posets via $q$-integrals}

\author{Jang Soo Kim}
\address{
Department of Mathematics, Sungkyunkwan University, Suwon 16420,
South Korea}
\email{jangsookim@skku.edu}

\author{Meesue Yoo}
\address{
Applied Algebra and Optimization Research Center, Sungkyunkwan University, Suwon 16420,
South Korea}
\email{meesue.yoo@skku.edu}

\date{\today}

\thanks{The first author was supported by NRF grants \#2016R1D1A1A09917506 and \#2016R1A5A1008055.
The second author was supported by NRF grants \#2016R1A5A1008055 and \#2017R1C1B2005653.}



\date{\today}

\newtheorem{thm}{Theorem}[section]
\newtheorem{lem}[thm]{Lemma}

\theoremstyle{definition}

\newtheorem{defn}[thm]{Definition}

\newtheorem{remark}[thm]{Remark}
\numberwithin{equation}{section}
\newcommand\PP{\mathcal{P}}
\newcommand\NN{\mathbb{N}}
\newcommand\GF{\operatorname{GF}_q}

\newcommand\dqx{d_q x_1\cdots d_q x_n}
\newcommand\rdiag{\operatorname{rdiag}}
\newcommand\SSYT{\operatorname{SSYT}}
\newcommand\Par{\mathrm{Par}}
\newcommand\RPP{\operatorname{RPP}}
\newcommand\slant[2]{{}^{#1}\backslash_{#2}}

\begin{document}

\begin{abstract}
The hook length formula for $d$-complete posets states that the $P$-partition generating function for them is given by a product in terms of hook lengths. We give a new proof of the hook length formula using $q$-integrals. The proof is done by a case-by-case analysis consisting of two steps. First, we express the $P$-partition generating function for each case as a $q$-integral and then we evaluate the $q$-integrals. Several $q$-integrals are evaluated using partial fraction expansion identities and the others are verified by computer. 
\end{abstract}

\maketitle

\tableofcontents

\section{Introduction}

The classical hook length formula due to Frame, Robinson and Thrall \cite{Frame1954} states that for a partition $\lambda$ of $n$, the number $f^{\lambda}$ of standard Young tableaux of shape $\lambda$ is given by
\begin{equation}
  \label{eq:12}
f^\lambda = \frac{n!}{\prod_{x\in \lambda} h(x)},  
\end{equation}
where $h(x)$ is the hook length of the cell $x$ in $\lambda$. One can naturally consider the shape $\lambda$ as a poset $P$ on the cells in $\lambda$. Then the $P$-partition generating function for the poset also has the following hook length formula:
\begin{equation}
  \label{eq:13}
\sum_{\sigma:P\to\NN}q^{|\sigma|}=\prod_{x\in P}\frac1{1-q^{h(x)}},  
\end{equation}
where the sum is over all $P$-partitions $\sigma$. See Section~\ref{sec:preliminaries} for the definition of $P$-partitions. Using the theory of $P$-partitions, one can see that \eqref{eq:13} implies \eqref{eq:12}. 
There are also hook length formulas for the $P$-partition generating functions for the posets coming from shifted shapes and trees. 

In his study of minuscule representations for Kac--Moody algebras, Proctor \cite{Proctor1999a} introduced $d$-complete posets, which include the posets of shapes, shifted shapes and forests. Proctor \cite{Proctor1999} proved that every $d$-complete poset can be decomposed into irreducible $d$-complete posets. He then classified all irreducible $d$-complete posets into 15 classes. 

To every element $x$ in a $d$-complete poset a positive integer $h(x)$ called the \emph{hook length} is assigned. 
The $d$-complete posets have the following hook length property. 

\begin{thm}[Hook Length Formula for $d$-complete posets] \label{thm:HLP}
For any $d$-complete poset $P$, we have
\[
\sum_{\sigma:P\to\NN}q^{|\sigma|}=\prod_{x\in P}\frac1{1-q^{h(x)}},
\]
where the sum is over all $P$-partitions $\sigma$.
\end{thm}

With Peterson's help, Proctor \cite{Proctor2014} showed Theorem~\ref{thm:HLP}. 
We note that Theorem~\ref{thm:HLP} was also proved by Nakada \cite{Nakada2009, Nakada} and generalized by Ishikawa and Tagawa \cite{IshikawaTagawa07, IshikawaTagawa} to ``leaf posets''. However, their proofs are only sketched in conference proceedings, and so a completely detailed proof of the hook length formula (Theorem~\ref{thm:HLP}) has not been available in the literature. The methods in \cite{IshikawaTagawa,Nakada,Proctor2014} actually produce more general multivariate ``colored'' hook length formulas for $d$-complete and leaf posets. See \cite{Proctor2017} for a survey of work involving $d$-complete posets. Recently, Naruse and Okada
\cite{NaruseOkada} proved a skew hook formula for $d$-complete posets using the equivariant $K$-theory. Their result generalizes the multivariate version of the hook length formula for $d$-complete posets, a skew hook length formula for shapes due to Naruse \cite{Naruse} and its $q$-analog due to
Morales, Pak and Panova \cite{MPP}. 

In this paper we give a new and complete proof of Theorem~\ref{thm:HLP} using $q$-integrals. Our proof is based on case-by-case analysis. Several cases are proved by using partial fraction expansion identities and the remaining cases are verified by computer. We outline our proof of Theorem~\ref{thm:HLP} as follows. 

We introduce semi-irreducible $d$-complete posets, which contain all irreducible $d$-complete posets. First, we show that in order to prove Theorem~\ref{thm:HLP}, it is sufficient to consider the semi-irreducible $d$-complete posets. Then we show that the $P$-partition generating function for each semi-irreducible $d$-complete poset can be written as a $q$-integral of a product of alternants in $n$ variables. Here $n$ is the number of elements on its ``diagonal''. Therefore, in order to compute the generating functions for 15 classes of semi-irreducible $d$-complete posets, we need to evaluate the corresponding 15 $q$-integrals. We note that a $q$-integral is by definition a certain formal power series in $q$. 
The key ingredient to the expression of a $P$-partition generating function as a $q$-integral  is the result \cite[Theorem 8.7]{KimStanton17} of the first author and Stanton on a formula for a generating function for reverse plane partitions with fixed diagonal entries.

The first two of the 15 classes are shapes and shifted shapes which are well known to have the hook length formula. Hence we focus on considering the remaining 13 classes, but we also provide a proof for the class of shifted shapes using the $q$-integral technique. Among the $q$-integrals corresponding to the 13 classes, 2 of them have an arbitrary number of integration variables and other 11 $q$-integrals have a fixed number of integration variables. In this paper we verify these 11 $q$-integrals by computer and compute the remaining $2$ $q$-integrals by hand. The computation of the 11 $q$-integrals consists of running a routine SAGE program on a personal computer on a time scale of hours. 

In the process of proving the $q$-integrals by hand, we utilize a known $q$-integral formula related shapes and shifted shapes, and several partial fraction expansion identities (cf. \cite{Milne1988}, \cite{WW}).

The rest of this paper is organized as follows.  In Section~\ref{sec:preliminaries}, we give necessary definitions. In Section~\ref{sec:some-properties-p} we prove some properties of $P$-partition generating functions which are used later. In Section~\ref{sec:semi-irreducible-d} we introduce semi-irreducible $d$-complete posets and prove that it suffices to consider them for showing the hook length formula for the $d$-complete posets. In Section~\ref{sec:q-integrals} we consider a certain class of posets that includes all semi-irreducible $d$-complete posets. We express the $P$-partition generating function for an arbitrary poset in this class as a $q$-integral. In Section~\ref{sec:p-part-gener} using the result in the previous section we express the $P$-partition generating function for each semi-irreducible $d$-complete poset as a $q$-integral. Then we find a $q$-integral formula which is equivalent to the hook length property of the poset. In Section~\ref{sec:eval-q-integr} we prove the $q$-integral formulas obtained in the Section~\ref{sec:p-part-gener}.

%==================================================

\section{Preliminaries}
\label{sec:preliminaries}

In this section, we introduce basic definitions and notation that are used throughout this paper.
We also recall some properties of $d$-complete posets. For the details, we refer the readers to \cite{Proctor1999, Proctor2017}.\\

\subsection{Basic definitions and notation}

We will use the following notation for $q$-series:
\[
(a;q)_n = (1-a)(1-aq)\cdots(1-aq^{n-1}),
\qquad (a_1,a_2,\dots,a_k;q)_n=(a_1;q)_n \cdots (a_k;q)_n.
\]

A \emph{partition} is a weakly decreasing sequence $\mu=(\mu_1,\dots,\mu_k)$ of nonnegative integers. Each $\mu_i$ is called a \emph{part} of $\mu$. The \emph{size} $|\mu|$ of $\mu$ is the sum of its parts. If the size of $\mu$ is $n$, we say that $\mu$ is a partition of $n$ and write $\mu\vdash n$. The \emph{length} of $\mu$ is the number of nonzero parts in $\mu$. 
We denote the set of partitions of length at most $n$ by $\Par_n$. 

The staircase partition $(n-1,n-2,\dots,1,0)$ is denoted by $\delta_n$.
For a partition $\lambda=(\lambda_1,\dots,\lambda_n)$, the \emph{alternant} $a_\lambda(x_1,\dots,x_n)$ is defined by
\[
a_\lambda(x_1,\dots,x_n) = \det(x_i^{\lambda_j})_{i,j=1}^n.
\]

Two partitions are considered to be the same if they have the same nonzero parts. For example,
$(4,3,1)=(4,3,1,0)=(4,3,1,0,0)$. Thus, for a partition $\mu=(\mu_1,\dots,\mu_k)$,
we use the convention $\mu_i=0$ for $i>k$ if it is necessary.
For two partitions $\lambda=(\lambda_1,\dots,\lambda_r)$ and $\mu=(\mu_1,\dots,\mu_s)$, we define $\lambda+\mu$ to be the partition $(\lambda_1+\mu_1,\dots,\lambda_k+\mu_k)$, where $k=\max(r,s)$. For example, $(3,1,1)+(4,3,2,1,0)=(7,4,3,1,0)$.

Let $\lambda=(\lambda_1,\dots,\lambda_n)$ be a partition. 
The \emph{Young diagram} of $\lambda$ is the left-justified array of squares in which there are $\lambda_i$ squares in row $i$. The \emph{Young poset} of $\lambda$ is the poset whose elements are the squares in the Young diagram of $\lambda$ with relation $x\le y$ if $x$ is weakly below and weakly to the right of $y$. See Figure~\ref{fig:young}. The \emph{transpose} of $\lambda$ is the partition $\lambda'$ whose $i$th part $\lambda'_i$ is the number of parts of $\lambda$ at least $i$.

\begin{figure}[h]
  \centering
\includegraphics{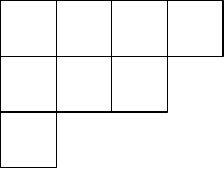}  \qquad \qquad
\includegraphics{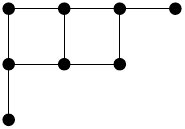}  
  \caption{The Young diagram of $\lambda=(4,3,1)$ on the left and its Young poset on the right.}
  \label{fig:young}
\end{figure}

Throughout this paper, Hasse diagrams are rotated $45^\circ$ counterclockwise unless otherwise stated. Therefore, smaller elements are located weakly to the right and weakly below larger elements.

If $\lambda$ has no equal nonzero parts, $\lambda$ is called \emph{strict}.
For a strict partition $\lambda$, the \emph{shifted Young diagram} of $\lambda$ is 
the diagram obtained from the Young diagram of $\lambda$ by shifting the $i$th row to the right by $i-1$ units. The \emph{shifted Young poset} of $\lambda$ is defined similarly. See Figure~\ref{fig:shifted_young}.

\begin{figure}[h]
  \centering
\includegraphics{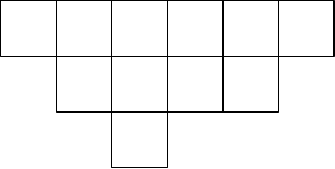}  \qquad \qquad
\includegraphics{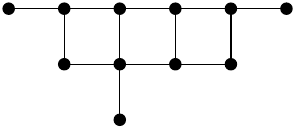}  
  \caption{The shifted Young diagram of $\lambda=(6,4,1)$ on the left and its shifted Young poset on the right.}
  \label{fig:shifted_young}
\end{figure}

By abuse of notation, we identify a partition $\lambda$ with its Young diagram and also with its Young poset if there is no possible confusion. For a strict partition $\lambda$, the shifted Young diagram of $\lambda$ is denoted by $\lambda^*$. Similarly, the shifted Young poset of $\lambda$ will also be written as $\lambda^*$.

For a Young diagram or a shifted Young diagram $\lambda$, a \emph{semistandard Young tableau} of shape $\lambda$ is a filling of $\lambda$ with nonnegative integers such that the integers are weakly increasing in each row and strictly increasing in each column.  A \emph{reverse plane partition} of shape $\lambda$ is a filling of $\lambda$ with nonnegative integers such that the integers are weakly increasing in each row and each column.
See Figures~\ref{fig:ssyt} and \ref{fig:shifted_ssyt}.

\begin{figure}[h]
  \centering
\includegraphics{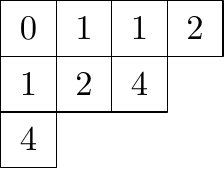}  \qquad \qquad
\includegraphics{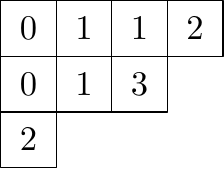}  
  \caption{A semistandard Young tableau of shape $(4,3,1)$ on the left and a reverse plane partition of shape $(4,3,1)$ on the right.}
  \label{fig:ssyt}
\end{figure}

\begin{figure}[h]
  \centering
\includegraphics{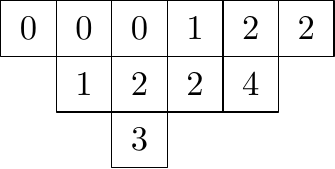}  \qquad \qquad
\includegraphics{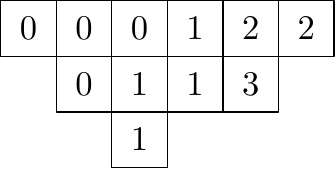}  
  \caption{A semistandard Young tableau of shifted shape $(6,4,1)$ on the left and a reverse plane partition of shifted shape $(6,4,1)$ on the right.}
  \label{fig:shifted_ssyt}
\end{figure}

 We denote by $\SSYT(\lambda)$ and $\RPP(\lambda)$ the set of semistandard Young tableaux of shape $\lambda$
and the set of reverse plane partitions of shape $\lambda$, respectively. 
Note that there is a natural bijection from $\RPP(\lambda)$ to $\SSYT(\lambda)$ which sends a reverse plane partition $T$ to the semistandard Young tableau obtained from $T$ by adding $i-1$ to each entry in the $i$th row. For example, the reverse plane partitions and the semistandard Young tableaux in Figures~\ref{fig:ssyt} and \ref{fig:shifted_ssyt} are mapped to each other by this bijection.

Let $\lambda$ be a strict partition. For $T\in \SSYT(\lambda^*)$ or $T\in\RPP(\lambda^*)$,
the leftmost entry in each row is called a \emph{diagonal entry}. 
We define the \emph{reverse diagonal sequence} $\rdiag(T)$ to be the sequence of diagonal entries in the non-increasing order. For example, if $T_1$ and $T_2$ are the semistandard Young tableau and the reverse plane partition in Figure~\ref{fig:shifted_ssyt}, respectively, then $\rdiag(T_1)=(3,1,0)$ and $\rdiag(T_2)=(1,0,0)$.

Now we recall basic properties of $P$-partitions. See \cite[Chapter 3]{EC1} for more details.

Let $P$ be a poset with $n$ elements.  A \emph{$P$-partition} is a map $\sigma:P\to\NN$ such that $x\le_P y$ implies $\sigma(x)\ge \sigma(y)$. In other words, a $P$-partition is just an order-reversing map from $P$ to $\NN$. 

For an integer $m\ge0$, we denote by $\PP_{\ge m}(P)$ the set of all $P$-partitions $\sigma$ with $\min(\sigma)\ge m$.  We also define $\PP(P)=\PP_{\ge0}(P)$.  For a $P$-partition $\sigma$, the \emph{size} $|\sigma|$ of $\sigma$ is defined by 
\[
 |\sigma| = \sum_{x\in P}\sigma(x).  
\]

For a poset $P$, we define $\GF(P)$ to be the $P$-partition generating function:
\[
\GF(P) = \sum_{\sigma\in \PP(P)}q^{|\sigma|}.
\]
Note that if $P$ is a disjoint union of posets $P_1,P_2,\dots,P_k$, then
\[
\GF(P) = \GF(P_1)\cdots \GF(P_k).
\]

The following definitions allow us to build $d$-complete posets starting from a chain.

\begin{defn}
Let $P$ be a poset containing a chain $C=\{x_1<x_2<\dots<x_n\}$. 
For $\lambda\in\Par_n$, we denote by $D(P,C,\lambda)$ the poset obtained by taking the disjoint union
of $P$ and $(\lambda+\delta_{n+1})^*$ and identifying $x_n,x_{n-1},\dots,x_1$ with the diagonal elements of 
$(\lambda+\delta_{n+1})^*$. 
\end{defn}

\begin{defn}\label{defn:P(X)}
Let $n$ and $k$ be positive integers. Let
\[
X = \{(\lambda^{(i)}, n_i, s_i): 1\le i\le k\},
\]
where $n_i$ and $s_i$ are positive integers with $s_i+n_i-1\le n$, $\lambda^{(i)}\in\Par_{n_i}$. 
We define $P_{n}(X)$ to be the poset constructed as follows. 
Let $P_0$ be a chain $x_1<x_2<\dots <x_n$ with $n$ elements, called \emph{diagonal entries}. For $1\le i\le k$, we define
$P_i=D(P_{i-1},C_i,\lambda^{(i)})$ where $C_i=\{x_{s_i}<x_{s_i+1}<\dots<x_{s_i+n_i-1}\}$.
Then we define $P_{n}(X)=P_k$. We also define $P_n^m(X)$ to be the poset obtained from $P_n(X)$ by attaching a chain with $m$ elements above $x_n$. See Figure~\ref{fig:5} for an example. We say that an element $y\in P_n(X)$ is \emph{of level $i$} if $y\le x_i$ and $y\not\le x_{i-1}$, see Figure~\ref{fig:level}. 
\end{defn}

\begin{figure}[h]
  \centering
\includegraphics{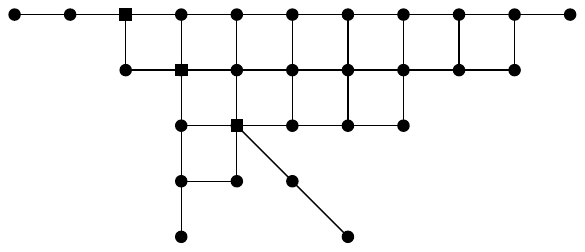}  
  \caption{The poset $P_3^2(X)$ for $X=\{((2),1,1), ((2,1),2,1), ((6,5,3),3,1), (\emptyset, 2,2)\}$,
where $\lambda=(4,2)$ and $\mu=(5,4,2)$.}
  \label{fig:5}
\end{figure}

\begin{figure}[h]
  \centering
\includegraphics{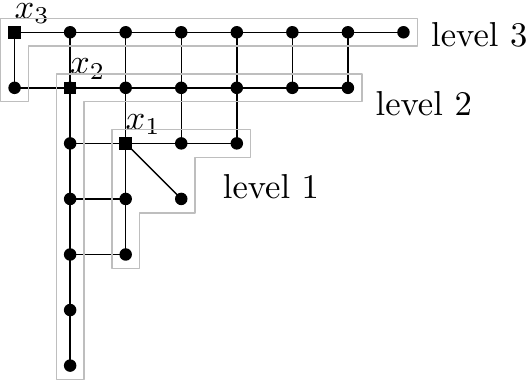}  
  \caption{The poset $P_3(X)$ for $X=\{((1),1,1), ((4,2),2,1), ((5,4,2),3,1), (\emptyset,2,2)\}$. This poset has $6$ elements of level $1$, $11$ elements of level $2$ and $9$ elements of level $3$.}
  \label{fig:level}
\end{figure}

In this paper the diagonal entries of  $P_n^m(X)$ will be represented by black squares in the Hasse diagram as shown in Figures~\ref{fig:5} and \ref{fig:level}.

\subsection{$d$-complete posets}

In this subsection we define $d$-complete posets following \cite{Proctor1999}.

Let $P$ be a poset. For two elements $x\le y$ in $P$,  $[x,y]$ denote the interval
$\{z: x\le z\le y\}$. A \emph{convex set} is a subset $A$ of $P$ such that  $x,y\in A$
and $x\le z\le y$ imply $z\in A$. A \emph{diamond} in $P$ consists of four elements $w, x,y,z$ where $z$ covers $x$ and $y$, and both $x$ and $y$ cover $w$. In this case, $z$ is called the top, $w$ is called the bottom, and $x$ and $y$ are the sides. 

\begin{figure}[ht]
 \centering
\includegraphics{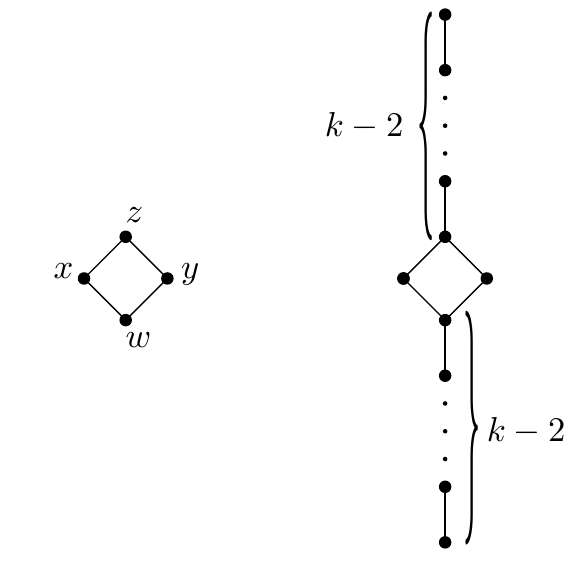}  
\caption{A diamond $d_3 (1)$ (left) and a doubled-tailed diamond poset $d_k(1)$ (right). Here, the Hasse diagrams are drawn in the usual way.}\label{fig:d_k}
\end{figure}
Let $k\ge 3$. A \emph{double-tailed diamond} poset $d_k(1)$ can be constructed from a diamond,
by attaching a chain of length $k-3$ above the top $z$, and below the bottom $w$. Thus $d_k(1)$ has $2k-2$ elements. 
The $k-2$ elements above the two incomparable elements $x$ and $y$ are called \emph{neck} elements. 
An interval $[w,z]$ is called a \emph{$d_k$-interval} if it is isomorphic to $d_k (1)$. A convex set is called a \emph{$d_k ^{-}$-convex set} if it is isomorphic to $d_k (1)\setminus \{ t\}$, where $t$ is the maximal element of $d_k (1)$. 
Note that $d_3^-$-convex set is a set of three elements $x,y$ and $w$ such that $w$ is covered by $x$ and $y$, and therefore not an interval. However, for $k\ge4$,
a \emph{$d_k^{-}$-convex set} is an interval.

Now we define $d$-complete posets and the hook lengths of their elements.

\begin{defn}
A poset $P$ is called \emph{$d$-complete} if it satisfies the following conditions for any $k\ge 3$:
\begin{enumerate}
\item if $I$ is a $d_k ^-$-convex set, then there exists an element $z$ such that $z$ covers the maximal element(s) 
of $I$ and $I\cup \{z\}$ is a $d_k$-interval,
\item if $[w,z]$ is a $d_k$-interval then $z$ does not cover any elements of $P$ outside $[w,z]$, and 
\item there are no $d_k ^-$ intervals which differ only in the minimal elements. 
\end{enumerate}
\end{defn}

\begin{defn}
Let $P$ be a $d$-complete poset. For any element $z\in P$, we define its \emph{hook length}, denoted by $h(z)$, 
recursively as follows:
\begin{enumerate}
\item if $z$ is not in the neck of any $d_k$-interval, then $h(z)$ is the number of elements in $P$ which are less than or 
equal to $z$, i.e., the number of elements of the principal order ideal generated by $z$.
\item If $z$ is in the neck of some $d_k$-interval, then we can take the unique element $w\in P$ such that $[w,z]$ is 
a $d_\ell$-interval, for some $\ell\le k$. If we let $x$ and $y$ be the two incomparable elements in this $d_\ell$-interval, 
then $h(z)=h(x)+h(y)-h(w)$.
\end{enumerate}
\end{defn}

We say that a $d$-complete poset $P$ has the hook length property if $P$ satisfies the hook length formula in Theorem~\ref{thm:HLP}. 

It is not hard to see from the definition that, if $P$ is $d$-complete and connected, then $P$ has a unique maximal
element $\widehat{1}$ and for each $w\in P$, every saturated chain from $w$ to $\widehat{1}$ has the same length
(see \cite[Proposition in \S 3]{Proctor1999}). A \emph{top tree element} $x\in P$ is an element such that 
every element $y\ge x$ is covered by at most one other element, and the \emph{top tree} of $P$ consists of all top 
tree elements. 

Given $P$ a connected $d$-complete poset with top tree $T$, an element $y\in P$ is called \emph{acyclic} 
if $y\in T$ and it is not in the neck of any $d_k$-interval for any $k\ge 3$.

\begin{defn}
  For $d$-complete posets $P_1$ and $P_2$, suppose that $P_1$ has an acyclic element $x$ and $P_2$ has the unique maximal element $y$. The \emph{slant sum} $P_1 \slant xy P_2$ of $P_1$ and $P_2$ at $x$ and $y$ is the poset obtained by taking the disjoint union of $P_1$ and $P_2$ with additional covering relation $x<y$. 
\end{defn}

A $d$-complete poset $P$ is called \emph{irreducible} if it is connected and it cannot be written as a slant sum of two non-empty $d$-complete posets. In \cite{Proctor1999}, Proctor showed that any connected $d$-complete poset $P$ can be uniquely decomposed into a slant sum of irreducible $d$-complete posets. Furthermore, he classified 
all irreducible $d$-complete posets which are connected into $15$ disjoint classes.
For the list of $15$ classes of irreducible $d$-complete posets, see \cite[Table 1]{Proctor1999}.

\section{Some properties of $P$-partitions}
\label{sec:some-properties-p}

In this section we prove two lemmas that will be used later in this paper.

\begin{defn}
  For a poset $P$,  let $P^+$ be the poset obtained from $P$ by adding a new element which is greater than all elements in $P$.
If $P$ has a unique maximal element, we define $P^-$ to be the poset obtained from $P$ by removing the maximal element.
\end{defn}

Note that  $(P^+)^-=P$ for any poset $P$. If $P$ has a unique maximal element, $(P^-)^+=P$.
There is a simple relation between $\GF(P^+)$ and $\GF(P)$.

\begin{lem}\label{lem:P+}
For a poset $P$ with $p$ elements, we have
\[
\GF(P^+) = \frac{1}{1-q^{p+1}}\GF(P).
\]
\end{lem}
\begin{proof}
Let $z$ be the unique maximal element of $P^+$. 
For $\sigma\in \PP(P^+)$, if $\sigma(z)=k$ and $\sigma'$ is the restriction of $\sigma$ to $P$, 
we have $\sigma'\in\PP_{\ge k}(P)$. Thus
\[
\GF(P^+) = \sum_{\sigma\in\PP(P^+)}q^{|\sigma|} = 
\sum_{k=0}^\infty \sum_{\sigma'\in\PP_{\ge k}(P)}q^{|\sigma'|+k} = 
\sum_{k=0}^\infty \sum_{\tau\in\PP(P)}q^{|\tau|+kp+k} = \frac{1}{1-q^{p+1}}\GF(P). \qedhere
\]
\end{proof}

\begin{defn}\label{def:dmkp}
Let $P$ be a poset in which there is a unique maximal element $y_1$ 
and a specified element $y_2$ covered by $y_1$. For integers $m,k\ge 1$, we define $D_{m,k}(P)$
to be the poset obtained from $P$ by adding a disjoint chain $z_m>\dots>z_1>z_0>z_{-1}>\dots>z_{-k}$
and a new element $y_0$ with additional covering relations 
$z_1>y_0, z_0>y_1, z_{-1}>y_2$ and $y_0>y_1$. See Figure~\ref{fig:dmkp}.
We also define $D_{k}(P)$ to be the poset obtained from $D_{m,k}(P)$ by removing the elements
$z_m,\dots,z_1$ and $y_0$.
\end{defn}

\begin{figure}[h]
  \centering
\includegraphics{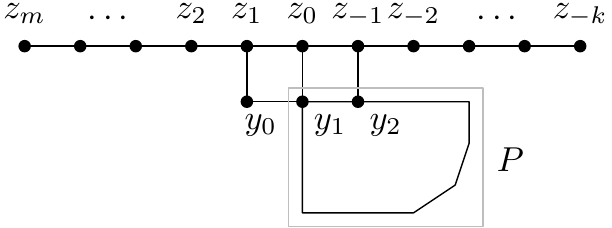}  \quad \quad \quad \includegraphics{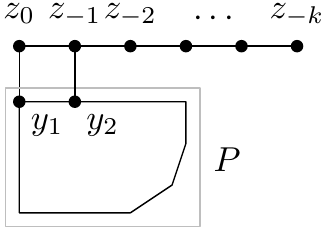}  
  \caption{The posets $D_{m,k}(P)$ on the left and $D_{k}(P)$ on the right.}
  \label{fig:dmkp}
\end{figure}

The following lemma will be used later to decompose the $P$-partition generating function of $d$-complete posets.

\begin{lem}\label{lem:dmkp}
Let $P=\{y_1,y_2,\dots,y_p\}$ be a poset in which $y_1$ is the unique maximal element and $y_2$ is covered by $y_1$.
Then
\[
\GF(D_{m,k}(P)) = \frac{1}{(q^{p+k+1};q)_{m+2}}
\left(
\frac{q^{p+1}}{(q;q)_{k-1}} \GF(P^+) + (1-q^{2p+2k+2})\GF(D_{k}(P))
\right).
\]
\end{lem}
\begin{proof}
By applying Lemma~\ref{lem:P+} repeatedly, we have
\[
\GF(D_{m,k}(P)) = \frac{1}{(q^{p+k+3};q)_{m}} \GF(D_{m,k}(P)-\{z_1,\dots,z_m\}).
\]  
Note that $D_{m,k}(P)-\{z_1,\dots,z_m\}$ is the poset obtained from $D_k(P)$ by adding a new element $y_0$ which covers $y_1$. Let $A_1$ (resp.~$A_2$) be the set of $P$-partitions $\sigma$ of $D_{m,k}(P)-\{z_1,\dots,z_m\}$ such that $\sigma(y_0)<\sigma(z_0)$ (resp.~$\sigma(y_0)\ge \sigma(z_0)$). Then
\[
\GF(D_{m,k}(P)-\{z_1,\dots,z_m\}) = \sum_{\sigma\in A_1} q^{|\sigma|}+ \sum_{\sigma\in A_2} q^{|\sigma|}.
\]
By considering the value of $\sigma(y_0)$ separately, one can see that
\begin{align*}
\sum_{\sigma\in A_1} q^{|\sigma|} &= \sum_{t\ge0} q^t \sum_{\sigma\in\PP_{\ge t+1}(D_k(P))} q^{|\sigma|}
=\sum_{t\ge0} q^t q^{(t+1)(p+k+1)}\sum_{\sigma\in\PP(D_k(P))} q^{|\sigma|}\\
&=\frac{q^{p+k+1}}{1-q^{p+k+2}}\GF(D_k(P)).
\end{align*}
Let $D_k'(P)$ be the poset obtained from $D_k(P)$ by adding a new element $y_0$ and the covering relations $y_1<y_0<z_0$. By Lemma~\ref{lem:P+}, 
\[
\sum_{\sigma\in A_2} q^{|\sigma|} = \sum_{\sigma\in\PP(D_k'(P))} q^{|\sigma|}
=\frac{1}{1-q^{p+k+2}}\sum_{\sigma\in\PP((D_k'(P))^-)} q^{|\sigma|}
=\frac{1}{1-q^{p+k+2}}\left(\sum_{\sigma\in B_1} q^{|\sigma|}
+\sum_{\sigma\in B_2} q^{|\sigma|}\right),
\]
where $B_1$ (resp.~$B_2$) is the set of $\sigma\in \PP((D_k'(P))^-)$ such that 
$\sigma(y_0)\ge \sigma(z_{-1})$ (resp.~$\sigma(y_0)< \sigma(z_{-1})$). 
First, observe that
\[
\sum_{\sigma\in B_1} q^{|\sigma|} = \GF(D_k(P)).
\]
For the sum over $\sigma\in B_2$, by considering the value of $\sigma(z_{-1})$ separately, we have
\begin{align*}
\sum_{\sigma\in B_2} q^{|\sigma|}  &=\sum_{t\ge0} q^t
\sum_{\sigma\in \PP_{\ge t+1}(P^+)} q^{|\sigma|} \sum_{\sigma'\in\PP_{\ge t}(C_{k-1})}q^{|\sigma'|}\\
&=\sum_{t\ge0} q^{t+(t+1)(p+1)+t(k-1)}
\sum_{\sigma\in \PP(P^+)} q^{|\sigma|} \sum_{\sigma'\in\PP(C_{k-1})}q^{|\sigma'|}\\
&=\frac{q^{p+1}}{1-q^{p+k+1}} \cdot \frac{1}{(q;q)_{k-1}} \GF(P^+),
\end{align*}
where $C_{k-1}$ is a chain with $k-1$ elements. 
By combining the above identities, we obtain 
the stated formula for $\GF(D_{m,k}(P))$.
\end{proof}

%==================================================

\section{Semi-irreducible $d$-complete posets}
\label{sec:semi-irreducible-d}

In this section we introduce semi-irreducible $d$-complete posets.
We show that in order to prove the hook length property of the $d$-complete posets it suffices to consider the semi-irreducible $d$-complete posets. 

\begin{defn}
A $d$-complete poset $P$ is \emph{semi-irreducible} if it is obtained from an irreducible $d$-complete poset by attaching a chain with arbitrary number of elements (possibly 0) below each acyclic element.
\end{defn}

We give the complete description of $15$ semi-irreducible $d$-complete poset classes in Section \ref{sec:p-part-gener} upon the computation of the $q$-integrals. These include all irreducible $d$-complete posets. See Appendix~\ref{apdx:figs}  for the figures of the semi-irreducible $d$-complete posets.

We will prove the following theorem by a case-by-case analysis in Sections~\ref{sec:p-part-gener} and \ref{sec:eval-q-integr}.

\begin{thm}\label{thm:semi-irreducible}
Every semi-irreducible $d$-complete poset has the hook length property.  
\end{thm}

Now we show that the above theorem implies the hook length property of the $d$-complete posets.
First we need the following lemma. 

\begin{lem}\label{lem:attach}
  Let $P_0$ be an irreducible $d$-complete poset with $k$ acyclic elements $y_1,\dots,y_k$.  Suppose that $P_{1},\dots, P_{k}$ are (possibly empty) connected $d$-complete posets having the hook length property. Let $P$ be the poset obtained from $P_0$ by attaching $P_i$ below $y_i$ 
for each $1\le i\le k$, i.e.,
\[
P = (\cdots(P_0 \slant{y_1}{v_1} P_1) \slant{y_2}{v_2} P_2)\cdots \slant{y_k}{v_k} P_k),
\]
where $v_i$ is the unique maximal element of $P_i$. Then $P$ also has the hook length property.
\end{lem}

\begin{proof}
First, observe that
\begin{equation}
  \label{eq:5}
\GF(P) = \sum_{\sigma\in \PP(P)} q^{|\sigma|} 
= \sum_{t_1,\dots,t_k\ge0} \sum_{\substack{\sigma_0\in \PP(P_0)\\ \sigma_0(y_i)=t_i}} q^{|\sigma_0|}
\prod_{i=1}^k \sum_{\sigma_i\in \PP_{\ge t_i}(P_i)} q^{|\sigma_i|}.
\end{equation}
Let $p_i=|P_i|$ for $1\le i\le k$. Since each $P_i$ has the hook length property, we have
\begin{equation}
  \label{eq:6}
\sum_{\sigma_i\in \PP_{\ge t_i}(P_i)} q^{|\sigma_i|}
=q^{t_i p_i}\sum_{\sigma_i\in \PP(P_i)} q^{|\sigma_i|}
=q^{t_i p_i}\prod_{u\in P_i} \frac{1}{1-q^{h(u)}}.
\end{equation}
Then by \eqref{eq:5} and \eqref{eq:6} we have
\begin{equation}
  \label{eq:4}
\GF(P) = \sum_{t_1,\dots,t_k\ge0} \sum_{\substack{\sigma_0\in \PP(P_0)\\ \sigma_0(y_i)=t_i}} q^{|\sigma_0|}
\prod_{i=1}^k \left(q^{t_i p_i} \prod_{u\in P_i} \frac{1}{1-q^{h(u)}}\right).
\end{equation}
For an integer $n\ge0$, denote by $C_n$ a chain with $n$ elements. Using the relation
\[
\sum_{\sigma_i\in \PP_{\ge t_i}(C_{p_i})} q^{|\sigma_i|} = \frac{q^{t_i p_i}}{(q;q)_{p_i}},
\]
we can rewrite \eqref{eq:4} as
\begin{equation}
  \label{eq:7}
\GF(P) = \sum_{t_1,\dots,t_k\ge0} \sum_{\substack{\sigma_0\in \PP(P_0)\\ \sigma_0(y_i)=t_i}} q^{|\sigma_0|}
\prod_{i=1}^k \left((q;q)_{p_i} \sum_{\sigma_i\in \PP_{\ge t_i}(C_{p_i})} q^{|\sigma_i|} 
\prod_{u\in P_i} \frac{1}{1-q^{h(u)}} \right).
\end{equation}

Let $P'$ be the semi-irreducible $d$-complete poset obtained from $P_0$ by
attaching $C_{p_i}$ below $y_i$ for each $i$. Then
\begin{equation}
  \label{eq:8}
\GF(P') = \sum_{t_1,\dots,t_k\ge0} \sum_{\substack{\sigma_0\in \PP(P_0)\\ \sigma_0(y_i)=t_i}} q^{|\sigma_0|}
\prod_{i=1}^k \sum_{\sigma_i\in \PP_{\ge t_i}(C_{p_i})} q^{|\sigma_i|}.
\end{equation}
By \eqref{eq:7} and \eqref{eq:8}, we obtain
\begin{equation}
  \label{eq:9}
\GF(P) = \GF(P') \prod_{i=1}^k \left((q;q)_{p_i} \prod_{u\in P_i} \frac{1}{1-q^{h(u)}}\right).
\end{equation}
On the other hand, by Theorem~\ref{thm:semi-irreducible}, 
\begin{equation}
  \label{eq:10}
\GF(P') = \prod_{u\in P'} \frac{1}{1-q^{h(u)}}
= \prod_{u\in P_0} \frac{1}{1-q^{h_{P'}(u)}}
\prod_{i=1}^k \prod_{u\in C_{p_i}}\frac{1}{1-q^{h(u)}}
= \prod_{u\in P_0} \frac{1}{1-q^{h_{P'}(u)}}
\prod_{i=1}^k \frac{1}{(q;q)_{p_i}}.
\end{equation}
By \eqref{eq:9} and \eqref{eq:10}, we obtain
\begin{equation}
  \label{eq:11}
\GF(P) = \prod_{u\in P_0} \frac{1}{1-q^{h_{P'}(u)}} \prod_{i=1}^k \prod_{u\in P_i} \frac{1}{1-q^{h(u)}}.  
\end{equation}
Since $y_i$'s are acyclic elements, we have $h_P(u)=h_{P'}(u)$ for all $u\in P_0$.
Thus the right hand side of \eqref{eq:11} is equal to $\prod_{u\in P}1/(1-q^{h(u)})$ and  $P$ has the hook length property. 
\end{proof}

Assuming Theorem~\ref{thm:semi-irreducible} we can prove the hook length formula of the $d$-complete posets. 

\begin{thm}
  Every $d$-complete poset has the hook length property.
\end{thm}
\begin{proof}
  Let $P$ be a $d$-complete poset. We prove the theorem by induction on the number of irreducible components of $P$. If $P$ is disconnected, by induction hypothesis each connected component has the hook length property. Now we assume that $P$ is connected. If $P$ is irreducible, the theorem is true by Theorem~\ref{thm:semi-irreducible}.  Suppose that $P$ is not irreducible. Let $P_0$ be the irreducible component of $P$ containing the unique maximum element of $P$. Then $P$ is obtained from $P_0$ by attaching several (possibly empty) $d$-complete posets below each acyclic element of $P_0$. By induction hypothesis, the attached $d$-complete posets have the hook length property. By Lemma~\ref{lem:attach}, $P$ also has the hook length property, which completes the proof.  \end{proof}

The rest of this paper is devoted to proving Theorem~\ref{thm:semi-irreducible}.

\section{$q$-integrals}
\label{sec:q-integrals}

In this section we prove that the $P$-partition generating function for $P_n(X)$ can be expressed as a $q$-integral, where $P_n(X)$ is the poset in Definition~\ref{defn:P(X)}.

The \emph{$q$-integral} of a function $f(x)$ over $[a,b]$ is defined by
\[
\int_a^b f(x) d_qx = (1-q)\sum_{i=0}^\infty 
\left( f(bq^i) bq^i - f(aq^i) aq^i \right),
\]
where it is assumed that $0<q<1$ and the sum absolutely converges.
We refer the reader to \cite{AAR} for the history and properties of $q$-integrals. 
In this paper, we only consider $q$-integrals when $f(x)$ is a polynomial. 
Note that in this case we can always evaluate any $q$-integral using the fact
\begin{equation}
  \label{eq:20}
\int_a^b x^k d_qx = \frac{(1-q)(b^{k+1}-a^{k+1})}{1-q^{k+1}}.  
\end{equation}

We define the multivariate $q$-integral over the simplex $\{(x_1,\dots,x_n): 0\le x_1\le\cdots\le x_n\le1\}$ by
\[
\int_{0\le x_1\le\cdots\le x_n\le1} f(x_1,\dots,x_n)\dqx
=\int_{0}^{1}\int_{0}^{x_n} \int_{0}^{x_{n-1}} \dots \int_0^{x_2}f(x_1,\dots,x_n)\dqx.
\]
For a multivariable function $f(x_1,\dots,x_n)$ and a partition $\lambda=(\lambda_1,\dots,\lambda_n)$, we denote
\[
f(q^\lambda) = f(q^{\lambda_1},\dots,q^{\lambda_n}).
\]
Then the following two lemmas are immediate from the definition and \eqref{eq:20}.

\begin{lem}\label{lem:qint}
We have
\[
\int_{0\le x_1\le \dots\le x_n\le 1} x_1^{a_1}\dots x_n^{a_n}d_q x_1\dots d_q x_n 
=\prod_{i=1}^n\frac{1-q}{1-q^{a_1+\dots+a_i+i}}.
\]
\end{lem}

\begin{lem}\label{lem:par}
We have
\[
\int_{0\le x_1\le\cdots\le x_n\le1}
f(x_1,\dots,x_n)\dqx=
(1-q)^n \sum_{\mu\in\Par_n} q^{|\mu|} f(q^{\mu}).
\]
\end{lem}

\begin{remark}
The definition of a multiple $q$-integral naturally extends to a multiple $q$-integral over a convex polytope (see \cite{KimStanton17}). The $q$-integral in Lemma~\ref{lem:par} is over the order polytope of a chain. 
Kim and Stanton \cite[Theorem~4.1]{KimStanton17} showed that for a poset $P$, 
the $q$-integral over the order polytope of $P$ is the generating function for the $(P,\omega)$-partitions, where
$\omega$ is the labeling of $P$ determined by the order of integration.
\end{remark}

The following lemma will be used in Section~\ref{sec:eval-q-integr}. 

\begin{lem}\label{lem:change_of_variables}
Let $f(x_1,\dots,x_{n-1})$ be a homogeneous function of degree $d$ in variables $x_1,\dots,x_{n-1}$, i.e.,
$f(tx_1,\dots,tx_{n-1})=t^d f(x_1,\dots,x_{n-1})$. Then
\begin{multline*}
\int_{0\le x_1\le \dots\le x_n\le 1} f(x_1,\dots,x_{n-1}) x_n^k d_qx_1\dots d_qx_n\\
=\frac{1-q}{1-q^{n+k+d}}\int_{0\le x_1\le \dots\le x_{n-1}\le 1} f(x_1,\dots,x_{n-1}) d_qx_1\dots d_qx_{n-1}.
\end{multline*}
\end{lem}
\begin{proof}
By Lemma~\ref{lem:par}, the left hand side is equal to 
\[
(1-q)^n\sum_{\mu_1\ge \dots\ge \mu_n\ge0} q^{|\mu|+k\mu_n}f(q^{\mu_1},\dots,q^{\mu_{n-1}}).
\]
Considering $\lambda=(\lambda_1,\dots,\lambda_{n-1})$  given by $\lambda_i=\mu_i-\mu_n$, the above sum can be written as
\[
(1-q)^n\sum_{\mu_n=0}^\infty q^{\mu_n(n+k+d)}\sum_{\lambda\in\Par_{n-1}} 
q^{|\lambda|}f(q^\lambda).
\]
By Lemma~\ref{lem:par} again, the above sum is equal to the right hand side of the formula in this lemma.
\end{proof}

In Appendix~\ref{apdx:figs},  we show that  every semi-irreducible $d$-complete poset can be written as $P_n^m(X)$.
By Lemma~\ref{lem:P+}, we have
\begin{equation}
  \label{eq:P(X;m)}
\GF(P_n^m(X))= \frac{1}{(q^{|P_n(X)|+1};q)_m} \GF(P_n(X)).  
\end{equation}

We will show that $\GF(P_n(X))$ can be written as a $q$-integral.
First, we need some lemmas.

\begin{lem}\label{lem:ssyt1}
Let $n$ and $k$ be positive integers and
\[
X = \{(\lambda^{(i)}, n_i, s_i): 1\le i\le k\},
\]
where $n_i$ and $s_i$ are positive integers with $s_i+n_i-1\le n$ and $\lambda^{(i)}\in\Par_{n_i}$. 
For $\mu=(\mu_1,\dots,\mu_n)\in\Par_n$, let
\[
\mu^{[i]} = (\mu_{s_i},\mu_{s_{i}+1},\dots,\mu_{s_i+n_i-1}).
\]
Then we have
\[
\GF(P_n(X)) = q^{-\sum_{i=1}^n(n-i)\ell_i}\sum_{\substack{\mu\in \Par_n\\ \mu: \rm{strict}}} q^{|\mu|}
\prod_{i=1}^k
\sum_{\substack{T\in \SSYT ( (\delta_{n_i+1}+\lambda^{(i)})^* ) \\
\rdiag(T)=\mu^{[i]}}}  q^{|T|-|\mu^{[i]}|},
\]
where $\ell_i$ is the number of elements of level $i$ in $P_n(X)$.
\end{lem}
\begin{proof}
Let $\PP_0$ be the set of $P_n(X)$-partitions. 
Let $\PP_1$ be the set of   $P_n(X)$-partitions $\tau$ with the condition that if an element $x$ of level $i$ covers an element $y$ of level $i-1$, then $\tau(x)<\tau(y)$. 
For $\sigma\in \PP_0$, define $f(\sigma)$ be the map $\tau:P_n(X)\to \NN$ given by
$\tau(x) = \sigma(x)+n-\ell(x)$, where $\ell(x)$ is the level of $x$. It is not hard to verify that
the map $f$ is a bijection from $\PP_0$ to $\PP_1$ with $|f(\sigma)|=|\sigma|+\sum_{i=1}^n(n-i)\ell_i$. Thus, we have
\[
\GF(P_n(X))= q^{-\sum_{i=1}^n(n-i)\ell_i} \sum_{\tau\in \PP_1} q^{|\tau|}.
\]
For every $\tau\in\PP_1$, we have $\tau(x_1)>\tau(x_2)>\dots>\tau(x_n)$. By the construction of $P_n(X)$, one can easily see that
\[
\sum_{\tau\in \PP_1} q^{|\tau|} = \sum_{\substack{\mu\in \Par_n\\ \mu: \rm{strict}}} q^{|\mu|}
\prod_{i=1}^k
\sum_{\substack{T\in \SSYT ( (\delta_{n_i+1}+\lambda^{(i)})^* ) \\
\rdiag(T)=\mu^{[i]}}}  q^{|T|-|\mu^{[i]}|},
\]
which completes the proof.
\end{proof}

\begin{lem}
\label{lem:ssyt2}
For $\lambda,\mu\in\Par_n$ we have
\[
\sum_{\substack{T\in \SSYT ( (\delta_{n+1}+\lambda)^* ) \\
\rdiag(T)=\mu}}  q^{|T|-|\mu|}
=  \frac{(-1)^{\binom n2} a_{\lambda+\delta_n}(q^{\mu}) }
{\prod_{j=1}^{n}(q;q)_{\lambda_j+n-j}}.
\]
\end{lem}
\begin{proof}
If $\mu$ has two equal parts, then both sides are zero. 
Thus it suffices to consider the case $\mu=\nu+\delta_n$ for some $\nu\in\Par_n$. 
Subtracting $i-1$ from every entry in the $i$th row gives a bijection from 
$\SSYT ( (\delta_{n+1}+\lambda)^* )$ to $\RPP( (\delta_{n+1}+\lambda)^* )$. Thus
\[
\sum_{\substack{T\in \SSYT ( (\delta_{n+1}+\lambda)^* ) \\
\rdiag(T)=\nu+\delta_n}}  q^{|T|}
=\sum_{\substack{T\in \RPP ( (\delta_{n+1}+\lambda)^* ) \\
\rdiag(T)=\nu}}  q^{|T|+\sum_{i=1}^n (i-1)(\lambda_i+n+1-i)}.
\]  
In \cite[Theorem~8.7]{KimStanton17}, it is shown that
\[
\sum_{\substack{T\in \RPP ( (\delta_{n+1}+\lambda)^* ) \\
\rdiag(T)=\nu}}  q^{|T|}
=  \frac{q^{|\nu+\delta_n|-\sum_{i=1}^n (i-1)(\lambda_i+n+1-i)} (-1)^{\binom n2} a _{\lambda+\delta_n}(q^{\nu+\delta_n})}
{\prod_{j=1}^{n}(q;q)_{\lambda_j+n-j}}.
\]
By the above two identities, we obtain the lemma.
\end{proof}

The following result implies that $\GF(P_n(X))$ can be expressed as a $q$-integral.

\begin{thm}\label{thm:attach}
Let $n$ and $k$ be positive integers and
\[
X = \{(\lambda^{(i)}, n_i, s_i): 1\le i\le k\},
\]
where $n_i$ and $s_i$ are positive integers with $s_i+n_i-1\le n$, $\lambda^{(i)}$ is a partition with $n_i$ parts. 
Suppose that for every $1\le j\le n-1$, there is $1\le i\le k$ such that
$s_i\le j<j+1\le s_i+n_i-1$. Then
\begin{multline*}
\GF(P_n(X)) \\= \frac{q^{-\sum_{i=1}^n(n-i)\ell_i}}{(1-q)^{n}}
\int_{0\le x_1\le \cdots\le x_n\le 1}\prod_{i=1}^k
\frac{(-1)^{\binom{n_i}2} a_{\lambda^{(i)}+\delta_{n_i}}(x_{s_i},x_{s_i+1},\dots,x_{s_i+n_i-1})}
{\prod_{j=1}^{n_i}(q;q)_{\lambda^{(i)}_j+n_i-j}} \dqx,
\end{multline*}
where $\ell_i$ is the number of elements of level $i$ in $P_n(X)$.
\end{thm}
\begin{proof}
Let 
\[
f(x_1,\dots,x_n) = \prod_{i=1}^k \frac{(-1)^{\binom{n_i}2} a_{\lambda^{(i)}+\delta_{n_i}}(x_{s_i},x_{s_i+1},\dots,x_{s_i+n_i-1})}
{\prod_{j=1}^{n_i}(q;q)_{\lambda^{(i)}_j+n_i-j}}. 
\]
By Lemmas~\ref{lem:ssyt1} and \ref{lem:ssyt2}, the left hand side is 
\[
\GF(P(X)) = 
q^{-\sum_{i=1}^n(n-i)\ell_i}\sum_{\substack{\mu\in \Par_n\\ \mu: \rm{strict}}} q^{|\mu|} f({q^{\mu}}).
\]
On the other hand, the right hand side is
\[
\frac{q^{-\sum_{i=1}^n(n-i)\ell_i}}{(1-q)^n}\int_{0\le x_1\le \cdots\le x_n\le 1} f(x_1,\dots,x_n) \dqx
=q^{-\sum_{i=1}^n(n-i)\ell_i} \sum_{\mu\in\Par_n} q^{|\mu|} f(q^\mu).
\]
If $\mu_i=\mu_j$ for some $i<j$, then we have $\mu_i=\mu_{i+1}$.
In this case, by the given condition, we have $f(q^\mu)=0$. Therefore, we only need to consider the partitions $\mu\in\Par_n$ which are strict. Hence, we obtain the desired identity.
\end{proof}

\section{$P$-partition generating functions to $q$-integrals}
\label{sec:p-part-gener}

In this section we express the $P$-partition generating function for each semi-irreducible $d$-complete poset as a $q$-integral using Theorem~\ref{thm:attach}. We then find a $q$-integral formula which is equivalent to the hook length property of the poset.  See Appendix~\ref{apdx:figs} for the list of all semi-irreducible $d$-complete posets and their figures. 

Note that for a partition $\lambda\in\Par_n$, we have
\begin{equation}
  \label{eq:hook}
\prod_{u\in \lambda} \frac1{1-q^{h(u)}} = 
\frac{\prod_{1\le i<j\le n} (1-q^{\lambda_i -\lambda_j +j-i})}{\prod_{i=1}^n (q;q)_{\lambda_i +n-i}},
\end{equation}
\begin{equation}
  \label{eq:shifted_hook}
\prod_{u\in (\lambda+\delta_{n+1})^*} \frac1{1-q^{h(u)}} = 
\frac{\prod_{1\le i<j\le n} (1-q^{\lambda_i -\lambda_j +j-i})}
{\prod_{i=1}^n (q;q)_{\lambda_i +n-i}\prod_{1\le i\le j\le n}(1-q^{2n+1-i-j+\lambda_i +\lambda_{j+1}})}.
\end{equation}

In this section, when a $d$-complete poset is given,  $\ell_i$ denotes the number of elements of level $i$ in the poset.

%----------------------------------------------------------------------------------

\subsection{Class 1: Shapes}\label{subsec:class1}

A semi-irreducible $d$-complete poset of class $1$ is $P_n(X_1)$, where $n\ge1$ and 
\[
X_1 = \{(\lambda, n,1), (\mu,n,1)\},
\]
with $\lambda,\mu\in\Par_n$, see Figure~\ref{fig:class1}. 
For $1\le i\le n$, we have $\ell_i = \lambda_{n+1-i}+\mu_{n+1-i}+2i-1$.
By Theorem~\ref{thm:attach}, 
\begin{multline*}
\GF(P_n(X_1)) = \\ \frac{q^{-\binom n2 -2\binom n3- \sum_{i=1}^{n-1}i (\lambda_{i+1}+\mu_{i+1})}}{(1-q)^n}
\int_{0\le x_1\le \cdots \le x_n\le 1} 
\frac{a_{\lambda+\delta_n}(x_1,\dots,x_n)a_{\mu+\delta_n}(x_1,\dots,x_n)}
{\prod_{i=1}^n(q;q)_{\lambda_i+n-i}(q;q)_{\mu_i+n-i}}
d_qx_1 \cdots d_qx_n.
\end{multline*}
Since
\[
\prod_{u \in P_n(X_1)} \frac{1}{1-q^{h(u)}} =
\frac{\prod_{1\le i<j\le n} (1-q^{\lambda_i -\lambda_j +j-i})(1-q^{\mu_i -\mu_j +j-i})}
{\prod_{i=1}^n (q;q)_{\lambda_i +n-i} (q;q)_{\mu_i +n-i}}
 \prod_{i,j=1}^n \frac{1}{1-q^{\lambda_i+\mu_j+2n-i-j+1}},
\]
the hook length property for class 1 is equivalent to
\begin{multline}\label{eq:class1}
\int_{0\le x_1\le \cdots \le x_n\le 1} 
a_{\lambda+\delta_n}(x_1,\dots,x_n)a_{\mu+\delta_n}(x_1,\dots,x_n)
d_qx_1 \cdots d_qx_n \\
= q^{\binom n2 +2\binom n3+\sum_{i=1}^{n-1}i (\lambda_{i+1}+\mu_{i+1})}(1-q)^n
\frac{\prod_{1\le i<j\le n} (1-q^{\lambda_i -\lambda_j +j-i})(1-q^{\mu_i -\mu_j +j-i})}
{\prod_{i,j=1}^n(1-q^{\lambda_i+\mu_j+2n-i-j+1})}.
\end{multline}
This is the special case $k = 1$ of Warnaar's integral \cite[Theorem 1.4]{Warnaar}.

%----------------------------------------------------------------------------------

\subsection{Class 2: Shifted shapes}\label{subsec:class2}
A semi-irreducible $d$-complete poset of class $2$ is $P_n(X_2)$, where $n\ge4$ and 
\[
X_2 = \{(\mu,n,1)\},
\]
with $\mu\in\Par_n$, see Figure~\ref{fig:class2}. 
For $1\le i\le n$, we have $\ell_i = \mu_{n+1-i}+i$.
By Theorem~\ref{thm:attach}, 
\[
\GF(P_n(X_2)) = \frac{q^{-\binom {n+1}3 -\sum_{i=1}^{n-1}i \mu_{i+1}}}{(1-q)^n}
\int_{0\le x_1\le \cdots \le x_n\le 1} 
\frac{a_{\mu+\delta_n}(x_1,\dots,x_n)}{\prod_{i=1}^n(q;q)_{\mu_i+n-i}}
d_qx_1 \cdots d_qx_n.
\]
By \eqref{eq:shifted_hook}, 
the hook length property for class 2 is equivalent to
\begin{multline}\label{eq:class2}
\int_{0\le x_1\le \cdots \le x_n\le 1} a_{\mu+\delta_n}(x_1,\dots,x_n) d_qx_1 \cdots d_qx_n \\
= q^{\binom {n+1}3 +\sum_{i=1}^{n-1}i \mu_{i+1}}(1-q)^n
\frac{\prod_{1\le i<j\le n} (1-q^{\mu_i -\mu_j +j-i})}
{\prod_{1\le i\le j\le n}(1-q^{2n+1-i-j+\mu_i +\mu_{j+1}})},
\end{multline}
which is proved in \cite[Theorem 8.16]{KimStanton17} using the connection between reverse plane partitions and $q$-integrals. In Section~\ref{sec:eval-q-integr} we give a new prove of \eqref{eq:class2} by explicitly computing the $q$-integral.

%----------------------------------------------------------------------------------

\subsection{Class 3: Birds}\label{subsec:class3}

A semi-irreducible $d$-complete poset of class $3$ is $P_2 ^m(X_3)$, where 
\[
X_3 = \{(\lambda, 2,1), (\mu,2,1), ((m), 1, 1)\},
\]
with $\lambda=(\lambda_1,\lambda_2)$ and $\mu=(\mu_1, \mu_2)$, see Figure~\ref{fig:class3}. 

We have
$\ell_1=\lambda_2 + \mu_2+m+1$, $\ell_2=\lambda_1+\mu_1+3$. Thus
\[
\GF(P_2(X_3)) = \frac{q^{-(\lambda_2 + \mu_2+m+1)}}{(1-q)^2}
\int_{0\le x_1\le x_2\le 1} 
\frac{a_{\mu+\delta_2}(x_1,x_2)}{(q;q)_{\mu_1+1}(q;q)_{\mu_2}}
\cdot\frac{a_{\lambda+\delta_2}(x_1,x_2)}{(q;q)_{\lambda_1+1}(q;q)_{\lambda_2}}
\cdot\frac{a_{(m)}(x_1)}{(q;q)_{m}} d_qx_1d_qx_2.
\]

The hook length property for Class 3 is equivalent to
\begin{align*}
& \GF(P_2 ^m(X_3))\\
&= \frac{q^{-(\lambda_2 + \mu_2+m+1)}}{(1-q)^2 (q^{|\lambda|+|\mu|+m+5};q)_m}
\int_{0\le x_1\le x_2\le 1} 
\frac{a_{\mu+\delta_2}(x_1,x_2)}{(q;q)_{\mu_1+1}(q;q)_{\mu_2}}
\cdot\frac{a_{\lambda+\delta_2}(x_1,x_2)}{(q;q)_{\lambda_1+1}(q;q)_{\lambda_2}}
\cdot\frac{a_{(m)}(x_1)}{(q;q)_{m}} d_qx_1d_qx_2\\
&= \frac{1}{\prod_{c\in P_2^m(X_3;m)}(1-q^{h(c)})}\\
&= \frac{1-q^{\lambda_1-\lambda_2+1}}{(q;q)_{\lambda_1 +1}(q;q)_{\lambda_2}}
\frac{1-q^{\mu_1-\mu_2+1}}{(q;q)_{\mu_1 +1}(q;q)_{\mu_2}}\cdot \frac{1}{(q;q)_m}
\cdot \frac{1}{(q^{|\lambda|+|\mu|+m+4};q)_m\prod_{i=1}^2 \prod_{j=1}^2 (1-q^{\lambda_i+\mu_j+m+5-i-j})},
\end{align*}
or, 
\begin{multline}
\int_{0\le x_1\le x_2\le 1} x_1 ^m a_{\lambda+\delta_2}(x_1,x_2)
a_{\mu+\delta_2}(x_1,x_2) d_qx_1d_qx_2\\
=q^{\lambda_2 + \mu_2+m+1}
\frac{(1-q)^2(1-q^{|\lambda|+|\mu|+2m+4})(1-q^{\lambda_1-\lambda_2+1})(1-q^{\mu_1-\mu_2+1})}{(1-q^{|\lambda|+|\mu|+m+4})\prod_{i=1}^2 \prod_{j=1}^2 (1-q^{\lambda_i+\mu_j+m+5-i-j})}.\label{eqn:class3id}
\end{multline}
This formula has been verified by computer.

%----------------------------------------------------------------------------------

\subsection{Class 4: Insets}\label{subsec:class4}
A semi-irreducible $d$-complete poset of class $4$ is $P_{n+1} ^m(X_4)$, where $n\ge2$, $k\ge0$ and
$$X_4= \{(\lambda, n-1, 1), (\mu, n+1,1), ((k),1,n)\},
$$
with  $\lambda\in \Par_{n-1}$ and $\mu\in\Par_{n+1}$, see Figure~\ref{fig:class4}.
In this poset, 
$\ell_j = \lambda_{n-j}+\mu_{n-j+2}+2j-1$ for $1\le j\le n-1$, $\ell_n = \mu_2 +n+k$ and $\ell_{n+1}=\mu_1 + n+1$. 
By applying Lemma~\ref{lem:P+} and Theorem~\ref{thm:attach}, we obtain
$$\GF(P_{n+1} ^{\lambda_1 +n-2}(X_4))=\frac{1}{(q^{|\lambda|+|\mu|+n^2 +k+3};q)_{\lambda_1 +n-2}}\GF (P_{n+1}(X_4)),
$$
where
\begin{multline*}
 \GF (P_{n+1}(X_4)) = \frac{q^{-(\sum_{i=1}^n ((i+1)\lambda_{i}+i\mu_{i+1})+\frac{1}{6}n(n-1)(2n+5)+1+k)}}{(1-q)^{n+1}}
 \int_{0\le x_1\le \cdots \le x_{n+1}\le 1}\frac{a_{(k)}(x_n)}{(q;q)_k} \\
\times \frac{(-1)^{\binom{n-1}{2}}a_{\lambda+\delta_{n-1}}(x_1,\dots, x_{n-1})}{\prod_{i=1}^{n-1}(q;q)_{\lambda_i +n-1-i}}
\cdot \frac{(-1)^{\binom{n+1}{2}}a_{\mu +\delta_{n+1}}(x_1,\dots, x_{n+1})}{\prod_{i=1}^{n+1}(q;q)_{\mu_i +n+1-i}}
 d_q x_1\cdots d_q x_{n+1}.
\end{multline*}

An explicit computation of the hook lengths in the poset $P_{n+1}^{\lambda_1 +n-2}(X_4)$ gives
\begin{multline*}
\prod_{u\in P_{n+1}^{\lambda_1 +n-2}(X_4)}\frac{1}{1-q^{h(u)}}= 
 \frac{1}{(q;q)_k}\cdot \frac{1}{\prod_{i=1}^{n+1}(1-q^{|\lambda|+|\mu|-\mu_i+n(n-1)+k+i})}\\
\times \frac{\prod_{j=1}^{n-1}(1-q^{|\lambda|+|\mu|+\lambda_j +n^2 +n-j+k+1})}{(q^{|\lambda|+|\mu|+n^2+k+2};q)_{\lambda_1+n-1}}
\cdot \prod_{\substack{1\le i\le n+1\\ 1\le j \le n-1}}\frac{1}{1-q^{\mu_i+\lambda_j +2n-i-j+1}}\\
\times \frac{\prod_{1\le i<j \le n+1}(1-q^{\mu_i -\mu_j+j-i})}{\prod_{i=1}^{n+1} (q;q)_{\mu_i+n+1-i}}
\cdot\frac{\prod_{1\le i<j\le n-1}(1-q^{\lambda_i -\lambda_j+j-i})}{\prod_{i=1}^{n-1} (q;q)_{\lambda_i+n-1-i}}.
\end{multline*}

Thus the hook length property for class $4$ is equivalent to
\begin{multline}\label{eq:class4}
 \int_{0\le x_1 \le \cdots \le x_{n+1}\le 1} x_n ^k a_{\lambda +\delta_{n-1}}(x_1,\dots, x_{n-1}) a_{\mu +\delta_{n+1}}(x_1,\dots, x_{n+1}) d_q x_1 \cdots d_q x_{n+1}\\
= \frac{(-1)q^{\sum_{i=1}^n ((i+1)\lambda_{i}+i\mu_{i+1})+\frac{1}{6}n(n-1)(2n+5)+1+k}(1-q)^{n+1}}{\prod_{i=1}^{n+1}(1-q^{|\lambda|+|\mu|-\mu_i+n(n-1)+k+i})}\cdot \frac{\prod_{j=1}^{n-1}(1-q^{|\lambda|+|\mu|+\lambda_j +n^2 +n-j+k+1})}{1-q^{|\lambda|+|\mu|+n^2+k+2}}\\
\quad \times \frac{\prod_{1\le i<j\le n+1}(1-q^{\mu_i-\mu_j+j-i})\prod_{1\le i<j\le n-1}(1-q^{\lambda_i -\lambda_j +j-i})}{\prod_{\substack{1\le i\le n+1\\ 1\le j \le n-1}}(1-q^{\mu_i+\lambda_j +2n-i-j+1})}.
\end{multline}
We prove this identity in Section~\ref{thm:class4}.

%----------------------------------------------------------------------------------

\subsection{Class 5: Tailed insets}

A semi-irreducible $d$-complete poset of class $5$ is $P_3 ^{\lambda_1 +1}(X_5)$ with $\lambda\in\Par_2$, $\mu\in\Par_3$ and
$$X_5 =\{ (\lambda,2,1), (\mu,3,1),(\emptyset,2,2),((1),1,1)\}.$$
See Figure~\ref{fig:class5}.
In this poset, $\ell_1=\lambda_2 +\mu_3 +2$, $\ell_2 = \lambda_1 +\mu_2+ 3$ and $\ell_3 =\mu_1+4$. 
In this case,
$$\GF(P_3 ^{\lambda_1 +1}(X_5))=\frac{1}{(q^{|\lambda|+|\mu|+10};q)_{\lambda_1 +1}}\GF (P_3 (X_5)),
$$
where
\begin{multline*}
\GF (P_3 (X_5))= \frac{q^{-(\sum_{i=1}^2 i(\lambda_i +\mu_{i+1})+7)}}{(1-q)^3}
\int_{0\le x_1\le x_2\le x_3\le 1 } \frac{-a_{\lambda+\delta_2} (x_1,x_2)}{(q;q)_{\lambda_1 +1}(q;q)_{\lambda_2}}\\
\times \frac{-a_{\mu+\delta_3} (x_1,x_2,x_3)}{\prod_{j=1}^3 (q;q)_{\mu_j +3-j}}
\cdot \frac{-a_{\delta_2}(x_2,x_3)}{1-q}\cdot \frac{a_{(1)+\delta_1}(x_1)}{1-q} d_q x_1 d_q x_2 d_q x_3.
\end{multline*}
Then the hook length property for class $5$ is equivalent to the following identity
\begin{multline}\label{eq:class5}
\int_{0\le x_1\le x_2\le x_3\le 1 } x_1 (x_2-x_3)
a_{\lambda+\delta_2} (x_1,x_2) a_{\mu+\delta_3} (x_1,x_2,x_3)
d_q x_1 d_q x_2 d_q x_3\\
=\frac{(-1) q^{\sum_{i=1}^2 i(\lambda_i +\mu_{i+1})+7}(1-q)^4 (1-q^{\lambda_1-\lambda_2 +1})(1-q^{|\lambda|+|\mu|+\lambda_1 +10})(1-q^{|\lambda|+|\mu|+\lambda_2+9})}{1-q^{|\lambda|+|\mu|+9}}.\\
\times\frac{\prod_{1\le i<j \le 3}  (1-q^{\mu_i -\mu_j +j-i})}{\prod_{i=1}^2 \prod_{j=1}^3 (1-q^{\lambda_i +\mu_j +7-i-j})\prod_{i=1}^3 (1-q^{|\lambda|+|\mu|-\mu_i+4+i})}.
\end{multline}
This formula has been verified by computer.

%----------------------------------------------------------------------------------

\subsection{Class 6: Banners}

A semi-irreducible $d$-complete poset of class $6$ is $P_4 ^m(X_6)$ with $m\ge0$ and
$$X_6 =\{ (\mu,4,1),((m),1,2)\}$$
for $\mu\in\Par_4$.
See Figure~\ref{fig:class6}. 
In this poset, we have $\ell_1 = \mu_4 +1$, $\ell_2 = \mu_3 +m+2$, $\ell_3 = \mu_2 +3$ and $\ell_4 = \mu_1 +4$.
Then 
$$\GF (P_4 ^m(X_6))=\frac{1}{(q^{|\mu|+m+11};q)_m}\GF (P_4(X_6)),$$
where 
\begin{multline*}
\GF (P_4 (X_6)) = \frac{q^{-(\sum_{i=1}^3 i\mu_{i+1}+2m+10)}}{(1-q)^4}
\int_{0\le x_1\le x_2\le x_3\le x_4\le 1}\frac{a_{\mu +\delta_4}(x_1,x_2,x_3,x_4)}{\prod_{j=1}^4 (q;q)_{\mu_j+4-j}}\\ \times 
\frac{a_{(m)}(x_2)}{(q;q)_m}d_q x_1 d_q x_2 d_q x_3 d_q x_4.\qquad\qquad
\end{multline*}
Hence the hook length property for class $6$ is equivalent to the following identity
\begin{multline}\label{eq:class6}
\int_{0\le x_1\le x_2 \le x_3\le x_4\le 1} x_2 ^m a_{\mu+\delta_4}(x_1,x_2,x_3,x_4)  d_q x_1 d_q x_2 d_q x_3 d_q x_4\\
=\frac{q^{\sum_{i=1}^3 i\mu_{i+1}+2m+10}(1-q)^4}{\prod_{i=1}^4 (1-q^{\mu_i +5-i})} \cdot \frac{1-q^{|\mu|+2m+10}}{1-q^{|\mu|+m+10}}
 \prod_{1\le i<j\le 4}\frac{1-q^{\mu_i-\mu_j+j-i}}{1-q^{\mu_i +\mu_j+m+10-i-j}}.
\end{multline}
This formula has been verified by computer.

%----------------------------------------------------------------------------------

\subsection{Class 7: Nooks}

In this case, a semi-irreducible $d$-complete poset of class $7$ is $P_4 ^{\lambda_1 +2}(X_7)$ with
$$X_7 = \{ (\lambda,3,1), (\emptyset, 2,1), (\mu,3,2), (\emptyset,2,3)\},$$
where $\lambda= (\lambda_1,\lambda_2,\lambda_3)\in\Par_3$ and $\mu=(\mu_1,\mu_2,0)\in\Par_2$. 
See Figure~\ref{fig:class7}.
Also, we have 
$\ell_1=\lambda_3 +1$, $\ell_2 = \lambda_2 +3$, $\ell_3 = \lambda_1 +\mu_2 +4$ and $\ell_4 = \mu_1 +4$. 
Thus, 
$$\GF (P_4 ^{\lambda_1 +2}(X_7))=\frac{1}{(q^{|\lambda|+|\mu|+13};q)_{\lambda_1 +2}}\GF (P_4(X_7)),$$
where 
\begin{multline*}
\GF (P_4(X_7))=
\frac{q^{-(\sum_{i=1}^3 i\lambda_i +\mu_2 +13)}}{(1-q)^4}
\int_{0\le x_1\le x_2\le x_3 \le x_4 \le 1} \frac{a_{\lambda +\delta_3}(x_1,x_2,x_3)}{\prod_{i=1}^3(q;q)_{\lambda_i+3-i}}\\ \times
\frac{a_{\delta_2}(x_1,x_2)}{1-q}\cdot \frac{a_{\mu+\delta_3}(x_2,x_3,x_4)}{(q;q)_{\mu_1+2}(q;q)_{\mu_2+1}}
\cdot \frac{a_{\delta_2}(x_3,x_4)}{1-q}d_q x_1 d_q x_2 d_q x_3 d_q x_4.
\end{multline*}
The hook length property for class $7$ means the following identity

\begin{multline}\label{eq:class7}
\int_{0\le x_1\le x_2\le x_3 \le x_4 \le 1} (x_1-x_2) (x_3-x_4) a_{\lambda +\delta_3}(x_1,x_2,x_3)
a_{\mu+\delta_3}(x_2,x_3,x_4) d_q x_1 d_q x_2 d_q x_3 d_q x_4\\
 = \frac{q^{\sum_{i=1}^3 i\lambda_i +\mu_2 +13}(1-q)^6 (1-q^{\mu_1+2})(1-q^{\mu_2 +1})(1-q^{\mu_1-\mu_2+1})}{(q^{|\lambda| +|\mu|+10};q)_3}\cdot \prod_{i=1}^3 \frac{1-q^{|\lambda|+|\mu|+\lambda_i +15-i}}{(q^{\lambda_i +4-i};q)_2}\\
\quad\times \frac{\prod_{1\le i<j\le 3}(1-q^{\lambda_i -\lambda_j +j-i})}{\prod_{i=1}^3 \prod_{j=1}^2 (1-q^{|\lambda|+|\mu|-\lambda_i -\mu_j +3+i+j})}.\qquad\qquad\qquad\qquad
\end{multline}
This formula has been verified by computer.

%----------------------------------------------------------------------------------

\subsection{Class 8: Swivels}\label{subsec:class8}

For a semi-irreducible $d$-complete poset of class $8$, there are $4$ subclasses:
\begin{itemize}
\item Class 8-(1): $P_4 ^{\lambda_1 +2}(X_8^{(1)})$ for $\lambda\in\Par_3$ and
$$ X_8^{(1)} = \{ (\lambda, 3,1), ((2), 1,1), (\emptyset, 2,1), (\emptyset, 3,2), (\emptyset, 2,3)\},$$ 
\item Class 8-(2): $P_5 ^{\lambda_1 +3}(X_8^{(2)})$ for $\lambda\in\Par_4$ and
$$ X_8^{(2)} = \{ (\lambda, 4,1), ((1,0), 2,1), (\emptyset, 2,2), (\emptyset, 3,3), (\emptyset, 2,4)\},$$ 
\item Class 8-(3): $P_5 ^{\lambda_1 +3}(X_8^{(3)})$ for $\lambda\in\Par_4$ and
$$ X_8^{(3)} = \{ (\lambda, 4,1), ((1,1), 2,1), (\emptyset, 2,2), (\emptyset, 3,3), (\emptyset, 2,4)\},$$ 
\item Class 8-(4): $P_6 ^{\lambda_1 +4}(X_8^{(4)})$ for $\lambda\in\Par_4$ and
$$ X_8^{(4)} = \{ (\lambda, 5,1), (\emptyset, 3,1), (\emptyset, 2,3), (\emptyset, 3,4), (\emptyset, 2,5)\}.$$ 
\end{itemize}
See Figures~\ref{fig:class8-1}, \ref{fig:class8-2}, \ref{fig:class8-3} and \ref{fig:class8-4}.
We consider the 4 subclasses separately.

\textbf{Class 8-(1):} In this case,
$\ell_1=\lambda_3+3$, $\ell_2=\lambda_2+3$, $\ell_3=\lambda_1+4$ and $\ell_4=4$.
Then,
$$\GF (P_4^{\lambda_1 +2}(X_8^{(1)}))=\frac{1}{(q^{|\lambda|+15};q)_{\lambda_1 +2}}\GF (P_4 (X_8^{(1)})),
$$
where 
\begin{multline*}
\GF (P_4 (X_8^{(1)}))= \frac{q^{-(\sum_{i=1}^3 i\lambda_i +19)}}{(1-q)^4}
\int_{0\le x_1\le \cdots \le x_4\le 1} \frac{-a_{\lambda +\delta_3}(x_1,x_2, x_3)}{\prod_{i=1}^3 (q;q)_{\lambda_i +3-i}}
\cdot \frac{a_{(2)+\delta_1}(x_1)}{(1-q)(1-q^2)}\\
\qquad\times \frac{-a_{\delta_2}(x_1,x_2)}{1-q}
\cdot \frac{-a_{\delta_3}(x_2,x_3,x_4)}{(1-q)^2(1-q^2)}\cdot 
\frac{-a_{\delta_2}(x_3,x_4)}{1-q}d_q x_1 \cdots d_q x_4.
\end{multline*}
Hence, the hook length property for class $8$-(1) is equivalent to 
\begin{multline}
  \label{eq:class8-1}
\int_{0\le x_1\le \cdots \le x_4\le 1} x_1^2(x_1-x_2)(x_3-x_4) a_{\delta_3}(x_2,x_3,x_4)
a_{\lambda +\delta_3}(x_1,x_2, x_3) d_qx_1 \cdots d_q x_4 \\
=\frac{q^{\sum_{i=1}^3 i\lambda_i +19} (1-q)^{9}(1+q)}{(q^{|\lambda|+10};q)_2 (1-q^{|\lambda|+14})}
\cdot \prod_{i=1}^3\frac{1-q^{|\lambda|+\lambda_i+17-i}}
{(q^{\lambda_i+6-i};q)_2(q^{|\lambda|-\lambda_i+5+i};q)_2}
\prod_{1\le i<j\le 3} (1-q^{\lambda_i-\lambda_j+j-i}).
\end{multline}
This formula has been verified by computer.

\textbf{Class 8-(2):} In this case,
$\ell_1=\lambda_4+1$, $\ell_2=\lambda_3+4$, $\ell_3=\lambda_2+4$, $\ell_4=\lambda_1+5$ and $\ell_5=4$.
Then,
$$\GF (P_5^{\lambda_1 +3}(X_8^{(2)}))=\frac{1}{(q^{|\lambda|+19};q)_{\lambda_1 +3}}\GF (P_5 (X_8^{(2)})),
$$
where 
\begin{multline*}
\GF (P_5 (X_8^{(2)}))= \frac{q^{-(\sum_{i=1}^4 i\lambda_i +29)}}{(1-q)^5}
\int_{0\le x_1\le \cdots \le x_5\le 1} \frac{a_{\lambda +\delta_4}(x_1,x_2, x_3,x_4)}{\prod_{i=1}^4 (q;q)_{\lambda_i +4-i}}
\cdot \frac{-a_{(1,0)+\delta_2}(x_1,x_2)}{(1-q)(1-q^2)}\\
\qquad\times \frac{-a_{\delta_2}(x_2,x_3)}{1-q}
\cdot \frac{-a_{\delta_3}(x_3,x_4,x_5)}{(1-q)^2(1-q^2)}\cdot 
\frac{-a_{\delta_2}(x_4,x_5)}{1-q}d_q x_1 \cdots d_q x_5.
\end{multline*}
Hence, the hook length property for class $8$-(2) is equivalent to 
\begin{multline}
  \label{eq:class8-2}
\int_{0\le x_1\le \cdots \le x_5\le 1} (x_1^2-x_2^2)(x_2-x_3)(x_4-x_5) a_{\delta_3}(x_3,x_4,x_5)
a_{\lambda +\delta_4}(x_1,x_2, x_3,x_4) d_qx_1 \cdots d_q x_5 \\
=\frac{q^{\sum_{i=1}^4 i\lambda_i +29} (1-q)^{11}(1+q)^2}{(q^{|\lambda|+15};q^2)_2 (1-q^{|\lambda|+18})}
\cdot \prod_{i=1}^4\frac{1-q^{|\lambda|+\lambda_i+22-i}}
{(q^{\lambda_i+5-i};q^2)_2(1-q^{|\lambda|-\lambda_i+9+i})}
\prod_{1\le i<j\le 4} \frac{1-q^{\lambda_i-\lambda_j+j-i}}{1-q^{\lambda_i+\lambda_j+13-i-j}}.
\end{multline}
This formula has been verified by computer.

\textbf{Class 8-(3):} In this case,
$\ell_1=\lambda_4+2$, $\ell_2=\lambda_3+4$, $\ell_3=\lambda_2+4$, $\ell_4=\lambda_1+5$ and $\ell_5=4$.
Then,
$$\GF (P_5^{\lambda_1 +3}(X_8^{(3)}))=\frac{1}{(q^{|\lambda|+20};q)_{\lambda_1 +3}}\GF (P_5 (X_8^{(3)})),
$$
where 
\begin{multline*}
\GF (P_5 (X_8^{(3)}))= \frac{q^{-(\sum_{i=1}^4 i\lambda_i +33)}}{(1-q)^5}
\int_{0\le x_1\le \cdots \le x_5\le 1} \frac{a_{\lambda +\delta_4}(x_1,x_2, x_3,x_4)}{\prod_{i=1}^4 (q;q)_{\lambda_i +4-i}}
\cdot \frac{-a_{(1,1)+\delta_2}(x_1,x_2)}{(1-q)(1-q^2)}\\
\qquad\times \frac{-a_{\delta_2}(x_2,x_3)}{1-q}
\cdot \frac{-a_{\delta_3}(x_3,x_4,x_5)}{(1-q)^2(1-q^2)}\cdot 
\frac{-a_{\delta_2}(x_4,x_5)}{1-q}d_q x_1 \cdots d_q x_5.
\end{multline*}
Hence, the hook length property for class $8$-(3) is equivalent to 
\begin{multline}
  \label{eq:class8-3}
\int_{0\le x_1\le \cdots \le x_5\le 1} x_1x_2(x_1-x_2)(x_2-x_3)(x_4-x_5) a_{\delta_3}(x_3,x_4,x_5)
a_{\lambda +\delta_4}(x_1,x_2, x_3,x_4) d_qx_1 \cdots d_q x_5 \\
=\frac{q^{\sum_{i=1}^4 i\lambda_i +33} (1-q)^{11}(1+q)}{(q^{|\lambda|+16};q)_2 (1-q^{|\lambda|+19})}
\cdot \prod_{i=1}^4\frac{1-q^{|\lambda|+\lambda_i+23-i}}
{(q^{\lambda_i+6-i};q)_2(1-q^{|\lambda|-\lambda_i+9+i})}
\prod_{1\le i<j\le 4} \frac{1-q^{\lambda_i-\lambda_j+j-i}}{1-q^{\lambda_i+\lambda_j+14-i-j}}.
\end{multline}
This formula has been verified by computer.

\textbf{Class 8-(4):} In this case, $\ell_1 = \lambda_5 +1$, $\ell_2=\lambda_4 +3$, $\ell_3 = \lambda_3 +5$, $\ell_4 = \lambda_2 +5$, 
$\ell_5 = \lambda_1 +6$ and $\ell_6 =4$. Then,
$$\GF (P_6 ^{\lambda_1 +4}(X_8^{(4)}))=\frac{1}{(q^{|\lambda|+25};q)_{\lambda_1 +4}}\GF (P_6 (X_8^{(4)})),
$$
where 
\begin{multline*}
\GF (P_6 (X_8^{(4)}))= \frac{q^{-(\sum_{i=1}^5 i\lambda_i +48)}}{(1-q)^6}
\int_{0\le x_1\le \cdots \le x_6\le 1} \frac{a_{\lambda +\delta_5}(x_1,\dots, x_6)}{\prod_{i=1}^5 (q;q)_{\lambda_i +5-i}}
\cdot \frac{-a_{\delta_3 }(x_1,x_2,x_3)}{(1-q)^2(1-q^2)}\\
\qquad\times \frac{-a_{\delta_2}(x_3,x_4)}{1-q}
\cdot \frac{-a_{\delta_3}(x_4,x_5,x_6)}{(1-q)^2(1-q^2)}\cdot 
\frac{-a_{\delta_2}(x_5,x_6)}{1-q}d_q x_1 \cdots d_q x_6.
\end{multline*}
Hence, the hook length property for class $8$-(4) is equivalent to 
\begin{multline}
  \label{eq:class8-4}
\int_{0\le x_1\le \cdots \le x_6\le 1} (x_3-x_4)(x_5-x_6)
a_{\delta_3 }(x_1,x_2,x_3) a_{\delta_3}(x_4,x_5,x_6) a_{\lambda +\delta_5}(x_1,\dots, x_6)
d_q x_1 \cdots d_q x_6\\  
=\frac{q^{\sum_{i=1}^5 i\lambda_i +48}(1-q)^{14}(1+q)^2}{(q^{|\lambda|+21};q)_4}
\prod_{i=1}^5 \frac{1-q^{|\lambda|+\lambda_i+29-i}}{(q^{\lambda_i+6-i};q)_{3}}
\prod_{1\le i<j \le 5}\frac{1-q^{\lambda_i-\lambda_j+j-i}}{1-q^{|\lambda|-\lambda_i-\lambda_j+i+j+7}}.
\end{multline}
We provide a proof of the hook length property for class $8$-(4) in Theorem \ref{thm:class8}
by decomposing the $P$-partition generating function for $P_6 ^{\lambda_1 +4}(X_8^{(4)})$ using Lemma \ref{lem:dmkp}.

%----------------------------------------------------------------------------------

\subsection{Class 9: Tailed swivels}

For a semi-irreducible $d$-complete poset of class $9$, there are $2$ subclasses:
\begin{itemize}
\item Class 9-(1): $P_5 ^{\lambda_1 +3}(X_9^{(1)})$, with $\lambda\in\Par_5$ and
$$X_9^{(1)} = \{(\lambda, 4,1), ((1),1,1), ((1,0),2,1), (\emptyset, 2,2), (\emptyset,3,3),(\emptyset, 2,4) \}.$$
\item Class 9-(2): $P_5 ^{\lambda_1 +3}(X_9^{(2)})$, with $\lambda\in\Par_5$ and
$$X_9^{(2)} = \{(\lambda, 4,1), ((1),1,1), ((1,1),2,1), (\emptyset, 2,2), (\emptyset,3,3),(\emptyset, 2,4) \}.$$
\end{itemize}
See Figures~\ref{fig:class9-1} and \ref{fig:class9-2}.

\textbf{Class 9-(1):} We have $\ell_1=\lambda_4 +2$, $\ell_2 = \lambda_3 +4$, $\ell_3 =\lambda_2 +4$, $\ell_4 =\lambda_1 +5$ and $\ell_5 =4$. 
Then, 
$$\GF (P_5 ^{\lambda_1 +3}(X_9^{(1)}))=\frac{1}{(q^{|\lambda| +20};q)_{\lambda_1 +3}}\GF (P_5 (X_9^{(1)})),$$
where 
\begin{multline*}
\GF (P_5 (X_9^{(1)})) = \frac{q^{-(\sum_{i=1}^4 i\lambda_i +34)}}{(1-q)^5}
\int_{0\le x_1\le \cdots \le x_5 \le 1} \frac{a_{\lambda +\delta_4}(x_1,x_2,x_3,x_4)}{\prod_{i=1}^4 (q;q)_{\lambda_i +4-i}}
\cdot \frac{a_{(1)+\delta_1 }(x_1)}{1-q}\\
\times \frac{-a_{(1,0)+\delta_2}(x_1,x_2)}{(1-q)^2(1-q^2)}
\cdot \frac{-a_{\delta_2 }(x_2,x_3)}{1-q}\cdot \frac{-a_{\delta_3}(x_3,x_4,x_5)}{(1-q)^2(1-q^2)}
\cdot \frac{-a_{\delta_2}(x_4,x_5)}{1-q}d_q x_1 \cdots d_q x_5.
\end{multline*}
Thus the hook length property for class $9$-(1) is equivalent to the identity
\begin{multline}\label{eq:class9-1}
\int_{0\le x_1\le \cdots \le x_5 \le 1}  x_1 (x_1^2-x_2^2)(x_2-x_3)(x_4-x_5) a_{\delta_3}(x_3,x_4,x_5)  
 a_{\lambda+\delta_4 }(x_1,x_2,x_3,x_4) d_q x_1 \cdots d_q x_5\\
=\frac{q^{\sum_{i=1}^4 i\lambda_i +33}(1-q)^{11}(1+q)^2}{(q^{|\lambda|+15};q^2)_2 (1-q^{|\lambda|+19})} \cdot
\frac{\prod_{j=1}^4 (1-q^{|\lambda|+\lambda_j-j+23})}{\prod_{i=1}^4 (q^{\lambda_i +6-i};q^2)_2(1-q^{|\lambda| -\lambda_i +10+i})}
\prod_{1\le i<j \le 4}\frac{1-q^{\lambda_i -\lambda_j +j-i}}{1-q^{\lambda_i +\lambda_j+13-i-j}}.
\end{multline}
This formula has been verified by computer.

\textbf{Class 9-(2):} We have $\ell_1=\lambda_4 +3$, $\ell_2 = \lambda_3 +4$, $\ell_3 =\lambda_2 +4$, $\ell_4 =\lambda_1 +5$ and $\ell_5 =4$. 
Then, 
$$\GF (P_5 ^{\lambda_1 +3}(X_9^{(2)}))=\frac{1}{(q^{|\lambda| +21};q)_{\lambda_1 +3}}\GF (P_5 (X_9^{(2)})),$$
where 
\begin{multline*}
\GF (P_5 (X_9^{(2)})) = \frac{q^{-(\sum_{i=1}^4 i\lambda_i +37)}}{(1-q)^5}
\int_{0\le x_1\le \cdots \le x_5 \le 1} \frac{a_{\lambda +\delta_4}(x_1,x_2,x_3,x_4)}{\prod_{i=1}^4 (q;q)_{\lambda_i +4-i}}
\cdot \frac{a_{(1)+\delta_1 }(x_1)}{1-q}\\
\times \frac{-a_{(1,1)+\delta_2}(x_1,x_2)}{(1-q)^2(1-q^2)}
\cdot \frac{-a_{\delta_2 }(x_2,x_3)}{1-q}\cdot \frac{-a_{\delta_3}(x_3,x_4,x_5)}{(1-q)^2(1-q^2)}
\cdot \frac{-a_{\delta_2}(x_4,x_5)}{1-q}d_q x_1 \cdots d_q x_5.
\end{multline*}
Thus the hook length property for class $9$-(2) is equivalent to the identity
\begin{multline}\label{eq:class9-2}
\int_{0\le x_1\le \cdots \le x_5 \le 1}  x_1^2x_2(x_1-x_2) (x_2-x_3)(x_4-x_5) a_{\delta_3}(x_3,x_4,x_5)  
 a_{\lambda+\delta_4 }(x_1,x_2,x_3,x_4) d_q x_1 \cdots d_q x_5\\
=\frac{q^{\sum_{i=1}^4 i\lambda_i +37}(1-q)^{11}(1+q)}{(q^{|\lambda|+16};q)_2 (1-q^{|\lambda|+20})} \cdot
\frac{\prod_{j=1}^4 (1-q^{|\lambda|+\lambda_j-j+24})}{\prod_{i=1}^4 (q^{\lambda_i +7-i};q)_2(1-q^{|\lambda| -\lambda_i +10+i})}
\prod_{1\le i<j \le 4}\frac{1-q^{\lambda_i -\lambda_j +j-i}}{1-q^{\lambda_i +\lambda_j+14-i-j}}.
\end{multline}
This formula has been verified by computer.

%----------------------------------------------------------------------------------

\subsection{Class 10: Tagged swivels}

A semi-irreducible $d$-complete poset of class $10$ is $P_6 ^{\lambda_1 +4}(X_{10})$ with
$$X_{10}=\{ (\lambda,5,1),(\emptyset ,2,1),((1),2,2),(\emptyset, 2,3),(\emptyset, 3,4),(\emptyset, 2,5) \},$$
where $\lambda\in\Par_5$ and $\ell_1=\lambda_5 +1$, $\ell_2 = \lambda_4 +3$, $\ell_3 = \lambda_3 + 5$, $\ell_4 = \lambda_2 + 5$, $\ell_5 = \lambda_1 +6$, $\ell_6 = 4$. 
See Figure~\ref{fig:class10}.
In this case, 
$$\GF (P_6 ^{\lambda_1 +4}(X_{10}))=\frac{1}{(q^{|\lambda|+25};q)_{\lambda_1 +4}}\GF (P_6 (X_{10})),$$
where 
\begin{multline*}
\GF (P_6 (X_{10}))=\frac{q^{-(\sum_{i=1}^5 i\lambda_i +48)}}{(1-q)^6}
\int_{0\le x_1\le \cdots \le x_6 \le 1} \frac{a_{\lambda +\delta_5}(x_1,\dots, x_5)}{\prod_{i=1}^5 (q;q)_{\lambda_i+5-i}}
\cdot \frac{-a_{\delta_2}(x_1,x_2)}{1-q}\\
\times \frac{-a_{(1,0)+\delta_2}(x_2,x_3)}{(q;q)_2}
\cdot \frac{-a_{\delta_2}(x_3,x_4)}{1-q}\cdot \frac{-a_{\delta_3}(x_4,x_5,x_6)}{(1-q)^2(1-q^2)}\cdot
\frac{-a_{\delta_2}(x_5,x_6)}{1-q}d_q x_1 \cdots d_q x_6.
\end{multline*}
Hence the hook length property for class $10$ is equivalent to the identity
\begin{multline}\label{eqn:hlp_class10}
\int_{0\le x_1\le \cdots \le x_6 \le 1}
(x_1-x_2)(x_2^2-x_3^2)(x_3-x_4) (x_5-x_6) a_{\delta_3}(x_4,x_5,x_6) 
a_{\lambda +\delta_5}(x_1,\dots, x_5) d_q x_1 \cdots d_q x_6\\
= (-1)\frac{q^{\sum_{i=1}^5 i\lambda_i +48}(1-q)^{13}(1+q)^2}{(q^{|\lambda|+22};q)_3}
 \prod_{i=1}^5\frac{1-q^{|\lambda|+\lambda_i-i+29}}{(q^{\lambda_i+6-i};q)_2 (1-q^{|\lambda|-\lambda_i+14+i})}\\
\times \prod_{1\le i<j\le 5}\frac{1-q^{\lambda_i -\lambda_j +j-i}}{1-q^{\lambda_i +\lambda_j +15-i-j}}.\qquad\qquad\qquad\qquad\qquad
\end{multline}

We provide a proof of the hook length property for class $10$ in Theorem \ref{thm:class10}
by decomposing the $P$-partition generating function for $P_6 ^{\lambda_1 +4}(X_{10})$ using Lemma \ref{lem:dmkp}.
 
%----------------------------------------------------------------------------------

\subsection{Class 11: Swivel shifteds}
\label{sec:class-11}

A semi-irreducible $d$-complete poset of class $11$ is $P_{n+1} ^{\lambda_1 +n-1}(X_{11})$ with $n\ge3$, $k\ge1$, $\epsilon\in\{0,1\}$ and 
$$X_{11}=\{ (\lambda,n,1),((k-1),3,n-1),(\emptyset, 2, n),((\epsilon),1,1)\}\bigcup \bigcup_{i=1}^{n-2}\{ (\emptyset, 2, i)\}.$$
In this case we have $\ell_1 =\lambda_{n}+1$, $\ell_j =\lambda_{n+1-j}+j+1$, for $2\le j\le n$, $\ell_{n+1}=k+3$. 
See Figure~\ref{fig:class11}. Then
$$\GF (P_{n+1} ^{\lambda_1 +n-1}(X_{11}))=\frac{1}{(q^{|\lambda|+\frac{1}{2}n(n+3)+k+3};q)_{\lambda_1 +n-1}}\GF (P_{n+1}(X_{11})),$$
where 
\begin{multline*}
\GF (P_{n+1}(X_{11})) =\frac{q^{-(\sum_{i=1}^{n}i\lambda_i +\frac{1}{6}n(n^2+6n-1))}}{(1-q)^{n+1}}
 \int_{0\le x_1\le \cdots \le x_{n+1}\le 1}\frac{(-1)^{\binom{n}{2}}a_{\lambda +\delta_{n}}(x_1,\dots, x_{n})}{\prod_{i=1}^{n}(q;q)_{\lambda_i+n-i}}\\
\times \frac{-a_{(k-1)+\delta_3}(x_{n-1},x_{n},x_{n+1})}{(1-q)(q;q)_{k+1}}
\cdot \frac{-a_{\delta_2}(x_{n},x_{n+1})}{1-q}
\cdot \prod_{i=1}^{n-2}\frac{-a_{\delta_2 }(x_i,x_{i+1})}{1-q}d_q x_1 \cdots d_q x_{n+1}.
\end{multline*}
Hence, the hook length property for the poset of class $11$ is equivalent to the following identity
\begin{multline}\label{eqn:hlp_class11}
 \int_{0\le x_1\le \cdots \le x_{n+1}\le 1} a_{\lambda +\delta_{n}}(x_1,\dots, x_{n})\cdot a_{(k-1)+\delta_3}(x_{n-1},x_{n},x_{n+1})
\cdot (x_{n}-x_{n+1})\\
\times \prod_{i=1}^{n-2}(x_i-x_{i+1})d_q x_1 \cdots d_q x_{n+1}\\
= \frac{(-1)^{\binom{n+1}{2}}q^{\sum_{i=1}^{n}i\lambda_i +\frac{1}{6}n(n^2+6n-1)}(1-q)^{n+4}(q^k ;q)_2}{(q^{|\lambda|+k+\frac{1}{2}n(n+3)};q)_3}\\  \times 
 \prod_{j=1}^n \frac{1-q^{|\lambda|+\lambda_j+k+\frac{1}{2}n(n+5)+2-j}}{(q^{\lambda_j +n+1-j};q)_2 (1-q^{|\lambda|-\lambda_j+k+\frac{1}{2}n(n+1)+j-2})}
\prod_{1\le i<j \le n}\frac{1-q^{\lambda_i-\lambda_j+j-i}}{1-q^{2n+4-i-j+\lambda_i+\lambda_j}}.
\end{multline}

We provide a proof of the hook length property for class $11$ in Theorem \ref{thm:class11}
by decomposing the $P$-partition generating function for $P_{n+1} ^{\lambda_1 +n-1}(X_{11})$ using Lemma \ref{lem:dmkp}.

%----------------------------------------------------------------------------------

\subsection{Class 12: Pumps}
\label{sec:class-12:-pumps}

A semi-irreducible $d$-complete poset of class $12$ is $P_6 ^{\lambda_1 +3}(X_{12})$, where
$$ X_{12}=\{(\emptyset, 3,1),(\emptyset,2,1),(\lambda,4,2),(\emptyset,2,3),(\emptyset,3,4),(\emptyset,2,5) \},$$
with $\lambda\in\Par_3$. See Figure~\ref{fig:class12}. In this case, $\ell_1=1$, $\ell_2=3$, $\ell_3=\lambda_3 +4$, $\ell_4=\lambda_2 +4$, $\ell_5 =\lambda_1 +5$, $\ell_6 =4$. 
Then, 
$$ \GF (P_6 ^{\lambda_1 +3}(X_{12}))=\frac{1}{(q^{|\lambda|+22};q)_{\lambda_1 +3}}\GF (P_6 (X_{12})),$$
where 
\begin{multline*}
\GF (P_6 (X_{12})) =\frac{q^{-(\sum_{i=1}^3 i\lambda_i +42)}}{(1-q)^6}\int_{0\le x_1 \le \cdots \le x_6\le 1} \frac{-a_{\delta_3 }(x_1,x_2,x_3)}{(1-q)^2 (1-q^2)}
\cdot \frac{-a_{\delta_2 }(x_1,x_2)}{1-q}\\
\times  \frac{a_{\lambda+\delta_4}(x_2,x_3,x_4,x_5)}{\prod_{i=1}^3 (q;q)_{\lambda_i+4-i}}\frac{-a_{\delta_2 }(x_3,x_4)}{1-q}\cdot \frac{-a_{\delta_3}(x_4,x_5,x_6)}{(1-q)^2(1-q^2)}
\cdot \frac{-a_{\delta_2}(x_5,x_6)}{1-q}d_q x_1\cdots d_q x_6.
\end{multline*}
The hook length property for class $12$ is equivalent to the following identity 
\begin{multline}\label{eq:class12}
 \int_{0\le x_1 \le \cdots \le x_6\le 1}(x_1-x_2)(x_3-x_4)(x_5-x_6)
a_{\delta_3 }(x_1,x_2,x_3) a_{\delta_3}(x_4,x_5,x_6) \\
\times a_{\lambda+\delta_4}(x_2,x_3,x_4,x_5) d_q x_1\cdots d_q x_6\\
=\frac{(-1)q^{\sum_{i=1}^3 i\lambda_i +42}(1-q)^{12}}{(1-q^3)^2(1-q^4)(q^{|\lambda|+15};q)_3}
 \prod_{i=1}^3 \frac{(1-q^{\lambda_i +4-i})(1-q^{|\lambda|+\lambda_i+25-i})}{(q^{\lambda_i+8-i};q)_3}
 \prod_{1\le i<j \le 3}\frac{1-q^{\lambda_i -\lambda_j +j-i}}{(q^{\lambda_i +\lambda_j +15-i-j};q)_2}.
\end{multline}

It takes too much time to verify \eqref{eq:class12} by a brute-force computation. In Section~\ref{sec:class-12} we modify \eqref{eq:class12} so that the resulting integral can be verified by a brute-force computation.

%----------------------------------------------------------------------------------

\subsection{Class 13: Tailed pumps}

A semi-irreducible $d$-complete poset of class $13$ is $P_6 ^{\lambda_1 +3}(X_{13})$, where
$$ X_{13}=\{((1),1,1),(\emptyset, 2,1),(\emptyset,3,1),(\lambda,4,2),(\emptyset,2,3),(\emptyset,3,4),(\emptyset,2,5) \},$$
with $\lambda\in\Par_2$. See Figure~\ref{fig:class13}. We have $\ell_1=2$, $\ell_2=3$, $\ell_3=4$, $\ell_4=\lambda_2 +4$, $\ell_5 = \lambda_1 +5$, $\ell_6 =4$. In this case, 
$$\GF (P_6 ^{\lambda_1 +3}(X_{13}))=\frac{1}{(q^{|\lambda|+23};q)_{\lambda_1 +3}}\GF (P_6 (X_{13})),$$
where 
\begin{multline*}
\GF (P_6 (X_{13}))=\frac{q^{-(\sum_{i=1}^2 i\lambda_i +47)}}{(1-q)^6}
\int_{0\le x_1\le \cdots \le x_6\le 1}\frac{a_{(1)+\delta_1}(x_1)}{1-q}\cdot \frac{-a_{\delta_2}(x_1,x_2)}{1-q}
\cdot \frac{-a_{\delta_3}(x_1,x_2,x_3)}{(1-q)^2(1-q^2)}\\
\times \frac{a_{(\lambda_1,\lambda_2,0,0)+\delta_4}(x_2,x_3,x_4,x_5)}{(1-q) \prod_{i=1}^2 (q;q)_{\lambda_i +4-i} }
\cdot \frac{-a_{\delta_2 }(x_3,x_4)}{1-q}\cdot \frac{-a_{\delta_3}(x_4,x_5,x_6)}{(1-q)^2(1-q^2)}
\cdot \frac{-a_{\delta_2}(x_5,x_6)}{1-q}d_q x_1\cdots d_q x_6.
\end{multline*}
The hook length property for class $13$ is equivalent to the identity
\begin{multline}\label{eq:class13}
\int_{0\le x_1\le \cdots \le x_6\le 1} x_1(x_1-x_2) (x_3-x_4) (x_5-x_6) a_{\delta_3}(x_1,x_2,x_3) a_{\delta_3}(x_4,x_5,x_6)\\
\times a_{(\lambda_1,\lambda_2,0,0)+\delta_4}(x_2,x_3,x_4,x_5) d_q x_1\cdots d_q x_6\\
=\frac{(-1)q^{\sum_{i=1}^2 i\lambda_i +47}(1-q)^{17}(1+q)^2}{(q;q)_7 (q^3;q)_3}
\cdot \frac{(1-q^{\lambda_1-\lambda_2+1})(1-q^{|\lambda|+23})}{(q^{|\lambda|+13};q)_5}\\
\times \prod_{i=1}^2 \frac{(q^{\lambda_i +3-i};q)_2 (1-q^{|\lambda|+\lambda_i+26-i})}{(q^{\lambda_i +9-i};q)_5}.\qquad\qquad\qquad\qquad
\end{multline}
This formula has been verified by computer.

%----------------------------------------------------------------------------------

\subsection{Class 14: Near bats}

A semi-irreducible $d$-complete poset of class $14$ is $P_6 ^{m+3}(X_{14})$, with
$$ X_{14}=\{((2),1,1),(\emptyset,2,1),(\emptyset,3,1),((m),4,2),(\emptyset,2,3),(\emptyset,3,4),(\emptyset,2,5) \}, $$
and $m\ge0$. See Figure~\ref{fig:class14}. We have $\ell_1=3$, $\ell_2 =3$, $\ell_3=4$, $\ell_4=4$, $\ell_5=m+5$, $\ell_6=4$. Then,
$$\GF (P_6 ^{m+3}(X_{14})) =\frac{1}{(q^{m+24};q)_{m+3}}\GF (P_6 (X_{14})),$$
where 
\begin{multline*}
\GF (P_6 (X_{14})) =\frac{q^{-(m+52)}}{(1-q)^6}
\int_{0\le x_1\le \cdots \le x_6\le 1}\frac{a_{(2)+\delta_1}(x_1)}{(q;q)_2}
\cdot \frac{-a_{\delta_2}(x_1,x_2)}{1-q}\cdot \frac{-a_{\delta_3}(x_1,x_2,x_3)}{(1-q)^2(1-q^2)}\\
\times \frac{a_{(m,0,0,0)+\delta_4}(x_2,x_3,x_4,x_5)}{(q;q)_{m+3}(1-q)^2(1-q^2)}
\cdot \frac{-a_{\delta_2}(x_3,x_4)}{1-q}
\cdot \frac{-a_{\delta_3}(x_4,x_5,x_6)}{(1-q)^2(1-q^2)}
\cdot \frac{-a_{\delta_2}(x_5,x_6)}{1-q} d_q x_1 \cdots d_q x_6.
\end{multline*}
Thus the hook length property for the posets of class $14$ is equivalent to the identity
\begin{multline}\label{eq:class14}
\int_{0\le x_1\le \cdots \le x_6\le 1} x_1^2 (x_1-x_2) (x_3-x_4)(x_5-x_6) 
a_{\delta_3}(x_1,x_2,x_3) a_{\delta_3}(x_4,x_5,x_6) \\
\times a_{(m,0,0,0)+\delta_4}(x_2,x_3,x_4,x_5) d_q x_1 \cdots d_q x_6\\
=\frac{(-1)q^{m+52}(1-q)^{20}(1+q)^4(q^{m+1};q)_3 (q^{m+24};q)_2 (1+q^{m+13})}{(q;q)_{11}(q^4;q)_5 (q^{m+9};q)_9}.
\qquad\qquad\qquad
\end{multline}
This formula has been verified by computer.

%----------------------------------------------------------------------------------

\subsection{Class 15: Bat}

A semi-irreducible $d$-complete poset of class $15$ is $P_6 ^3(X_{15})$, with
$$ X_{15}=\{((3),1,1),(\emptyset,2,1),(\emptyset,3,1),(\emptyset,4,2),(\emptyset,2,3),(\emptyset,3,4),(\emptyset,2,5) \},$$
and $\ell_1=4$, $\ell_2 =3$, $\ell_3=4$, $\ell_4=4$, $\ell_5=5$, $\ell_6=4$. 
See Figure~\ref{fig:class15}. Then,
$$\GF (P_6 ^3(X_{15}) )=\frac{1}{(q^{25};q)_{3}}\GF (P_6 (X_{15})),$$
where 
\begin{multline*}
\GF (P_6 (X_{15})) =\frac{q^{-57}}{(1-q)^6}
\int_{0\le x_1\le \cdots \le x_6\le 1}\frac{a_{(3)+\delta_1}(x_1)}{(q;q)_3}
\cdot \frac{-a_{\delta_2}(x_1,x_2)}{1-q}\cdot \frac{-a_{\delta_3}(x_1,x_2,x_3)}{(1-q)^2(1-q^2)}\\
\times \frac{a_{\delta_4}(x_2,x_3,x_4,x_5)}{(1-q)(q;q)_2(q;q)_{3}}
\cdot \frac{-a_{\delta_2}(x_3,x_4)}{1-q}
\cdot \frac{-a_{\delta_3}(x_4,x_5,x_6)}{(1-q)^2(1-q^2)}
\cdot \frac{-a_{\delta_2}(x_5,x_6)}{1-q} d_q x_1 \cdots d_q x_6.
\end{multline*}
Hence the hook length property for the posets of class $15$ is equivalent to the identity
\begin{multline}\label{eq:class15}
\int_{0\le x_1\le \cdots \le x_6\le 1} x_1^3(x_1-x_2) (x_3-x_4)(x_5-x_6) 
a_{\delta_3}(x_1,x_2,x_3) a_{\delta_3}(x_4,x_5,x_6) \\
\times a_{\delta_4}(x_2,x_3,x_4,x_5) d_q x_1 \cdots d_q x_6\\
=\frac{(-1) q^{57}(1-q)^{18}(1+q)^3(q;q)_3^2(q^{25};q)_3}{(q;q)_{17} (q^5;q)_9(1-q^9)}.
\qquad\qquad\qquad
\end{multline}
This formula has been verified by computer.

%======================================================

\section{Evaluation of the $q$-integrals: fixed diagonal type}
\label{sec:eval-q-integr}

%======================================================

In this and next sections  we complete the proof of the hook length property of $d$-complete posets
by showing the corresponding $q$-integral formulas obtained in the previous section. 

Observe that the $q$-integral formula for class $i$, where
\[
i\in \{3,5,6,7,\mbox{8-(1),8-(2),8-(3),8-(4),9-(1),9-(2)},10,12,13,14,15\},
\]
has a $q$-integral of a polynomial with a fixed number of variables. Hence, it can be verified by computer. We say that these $q$-integral formulas are \emph{of fixed diagonal type} and the $q$-integral formulas for classes $1,2,4,11$ are \emph{of arbitrary diagonal type}.

In this section we prove the hook length property for $d$-complete posets of fixed diagonal type. The $q$-integral formulas of fixed diagonal type except classes 8-(4), 10 and 12 can be verified by a brute-force computation within a reasonable amount of time.

\begin{thm}\label{thm:fixed_type}
The hook length property holds for classes 3, 5, 6, 7, 8-(1), 8-(2), 8-(3), 9-(1), 9-(2), 13, 14, 15.
Equivalently, the equations \eqref{eqn:class3id}, \eqref{eq:class5}
\eqref{eq:class6}, \eqref{eq:class7}, \eqref{eq:class8-1}, \eqref{eq:class8-2}, \eqref{eq:class8-3},
\eqref{eq:class9-1}, \eqref{eq:class9-2}, \eqref{eq:class13}, \eqref{eq:class14}, \eqref{eq:class15} are true. 
\end{thm}

We provide testing code which computes the following integral using Lemma~\ref{lem:qint} when $n$ is given as a specific integer and $f$ is a polynomial
\[
\int_{0\le x_1\le \dots\le x_n\le 1} f(x_1,\dots,x_n)d_q x_1\dots d_q x_n.
\]
We have verified the equations in Theorem~\ref{thm:fixed_type} using our testing code written in SageMath \cite{sagemath} and a personal computer with 3.5 GHz Intel Core i7 and 32 GB memory, which took about 11 hours in total.
We note that it took less than a second for each of classes 3,5,6,
less than a minute for each of classes 7, 8-(1), 8-(4), 10, 15,
less than an hour for each of classes 12, 14, about 2 hours for each of classes 8-(2), 8-(3), 9-(1), 9-(2)
and about 3 hours for class 13. It is possible to prove for classes 3,5,6 by hand since the polynomial in the integral
has 4, 24, 24 terms, respectively. However, even for class 7, there are more than 144 terms in the polynomial to integrate. 
The polynomial for class 13 has 3552 terms. The polynomials for classes 8-(4),10,12 without modification have 14400, 9600, 4800 terms, respectively. 

For the remainder of this section we modify the $q$-integral formulas \eqref{eq:class8-4}, \eqref{eqn:hlp_class10}, \eqref{eq:class12} for classes 8-(4), 10, 12. After necessary modification, the resulting formulas are again verified by computer.

% \begin{verbatim}
% sage: load("check_d_complete.sage")
% sage: test_all()
% Testing Class3()
% True
% It took 0 hours 0 minutes 0.804966926575 seconds.
% Testing Class5()
% True
% It took 0 hours 0 minutes 0.267915010452 seconds.
% Testing Class6()
% True
% It took 0 hours 0 minutes 0.391582965851 seconds.
% Testing Class7()
% True
% It took 0 hours 0 minutes 17.5302219391 seconds.
% Testing Class8_1()
% True
% It took 0 hours 0 minutes 7.67528510094 seconds.
% Testing Class8_2()
% True
% It took 2 hours 6 minutes 50.6743628979 seconds.
% Testing Class8_3()
% True
% It took 1 hours 40 minutes 1.43686294556 seconds.
% Testing Class8_4()
% True
% It took 0 hours 0 minutes 10.5542430878 seconds.
% Testing Class9_1()
% True
% It took 1 hours 57 minutes 22.2945978642 seconds.
% Testing Class9_2()
% True
% It took 1 hours 56 minutes 15.6300830841 seconds.
% Testing Class10()
% True
% It took 0 hours 0 minutes 3.3731648922 seconds.
% Testing Class12()
% True
% It took 0 hours 24 minutes 39.6105351448 seconds.
% Testing Class13()
% True
% It took 2 hours 41 minutes 19.9908559322 seconds.
% Testing Class14()
% True
% It took 0 hours 15 minutes 59.6924049854 seconds.
% Testing Class15()
% True
% It took 0 hours 0 minutes 7.89000296593 seconds.
% All checked!
% It took 11 hours 3 minutes 17.8191730976 seconds.
% \end{verbatim}

Throughout this and next sections we use the following notation. For a partition $\mu\in\Par_n$ and an integer $1\le \ell\le n$, we define
$\widehat{\mu}_i ^{(\ell)}$ to be the partition whose parts are given by
\begin{equation}\label{eqn:la_j}
\widehat{\mu}_i ^{(\ell)}=
\begin{cases}
\mu_i +1 & \text{ if } i<\ell,\\ 
\mu_{i+1} & \text{ if } i\ge \ell.
\end{cases}
\end{equation}

The following lemma is useful in this and next sections.

\begin{lem}\label{lem:expand}
Let $f(x_1,\dots,x_{n-1})$ be a homogeneous function of degree $d$ in variables $x_1,\dots,x_{n-1}$, i.e.,
$f(tx_1,\dots,tx_{n-1})=t^d f(x_1,\dots,x_{n-1})$. Then for a partition $\mu\in\Par_n$,  we have
\begin{multline*}
\int_{0\le x_1\le \dots\le x_n\le 1} x_n^k f(x_1,\dots,x_{n-1}) a_{\mu+\delta_n}(x_1,\dots,x_n) d_qx_1\dots d_qx_n\\
=\frac{1-q}{1-q^{|\mu|+\binom{n+1}2+k+d}}  \sum_{\ell=1}^n (-1)^{n-\ell}
\int_{0\le x_1\le \dots\le x_{n-1}\le 1} f(x_1,\dots,x_{n-1}) \\
\times a_{\widehat{\mu}^{(\ell)}+\delta_{n-1}}(x_1,\dots,x_{n-1})  d_qx_1\dots d_qx_{n-1}.
\end{multline*}
\end{lem}
\begin{proof}
Since
\begin{equation}
  \label{eq:exp_alt}
a_{\mu +\delta_n}(x_1,\dots, x_n)=\sum_{\ell=1}^n (-1)^{n-\ell}x_n ^{\mu_\ell +n-\ell}a_{\widehat{\mu}^{(\ell)}+\delta_{n-1}}(x_1,\dots, x_{n-1}),
\end{equation}
the left hand side is equal to
\begin{multline*}
\sum_{\ell=1}^n (-1)^{n-\ell}
\int_{0\le x_1\le \dots\le x_{n-1}\le 1} x_n^{\mu_\ell+n-\ell+k} f(x_1,\dots,x_{n-1}) a_{\widehat{\mu}^{(\ell)}+\delta_{n-1}}(x_1,\dots,x_{n-1})
 d_qx_1\dots d_qx_{n-1}.  
\end{multline*}
Since $f(x_1,\dots,x_{n-1}) a_{\widehat{\mu}^{(\ell)}+\delta_{n-1}}(x_1,\dots,x_{n-1})$ is a homogeneous polynomial
in $x_1,\dots,x_{n-1}$ whose degree is
\[
|\widehat{\mu}^{(\ell)}|+\binom{n-1}2 +d= |\mu|-\mu_\ell+\ell-1+\binom{n-1}2 +d,
\]
by Lemma~\ref{lem:change_of_variables}, 
\begin{multline*}
\int_{0\le x_1\le \dots\le x_{n-1}\le 1} x_n^{\mu_\ell+n-\ell+k} f(x_1,\dots,x_{n-1}) a_{\widehat{\mu}^{(\ell)}+\delta_{n-1}}(x_1,\dots,x_{n-1})
 d_qx_1\dots d_qx_{n-1}\\
=\frac{1-q}{1-q^{|\mu|+\binom{n+1}2+k+d}}
\int_{0\le x_1\le \dots\le x_{n-1}\le 1} f(x_1,\dots,x_{n-1})
 a_{\widehat{\mu}^{(\ell)}+\delta_{n-1}}(x_1,\dots,x_{n-1})  d_qx_1\dots d_qx_{n-1},
\end{multline*}
which completes the proof.
\end{proof}

%----  class 8 -----------

\subsection{Class 8-(4) : Swivels}

\begin{thm}\label{thm:class8}
The hook length property holds for class 8-(4). 
\end{thm}

\begin{proof}
In Section \ref{subsec:class8}, 
we have expressed the hook length property in \eqref{eq:class8-4}. However, we will compute the $q$-integral for the poset $P_6 ^{\lambda_1 +4}(X_8 ^{(4)})$
in a different way by utilizing Lemma \ref{lem:dmkp}.

In terms of the notation defined in Definition \ref{def:dmkp}, the poset $P_6 ^{\lambda_1 +4}(X_8 ^{(4)})$ can be expressed as $D_{\mu_1 +4 ,1}(Q)$, where $Q=P_5(X)^-$ for
\begin{equation}
  \label{eq:14}
X=\{(\mu,5,1), (\emptyset,2,1)\}
\end{equation}
and $\mu=\lambda +(1^5)$, i.e., $\mu_i = \lambda_i +1$.  See Figure~\ref{fig:class8-4:1:2} for the posets $Q ^+$ and $D_1 (Q)$.  Note that the discrepancy between 
$\lambda$ in Figure~\ref{fig:class8-4} and $\mu$ in Figure~\ref{fig:class8-4:1:2} comes from moving the square dots (diagonal entries) up by one.

\begin{figure}[h]
\centering
\includegraphics{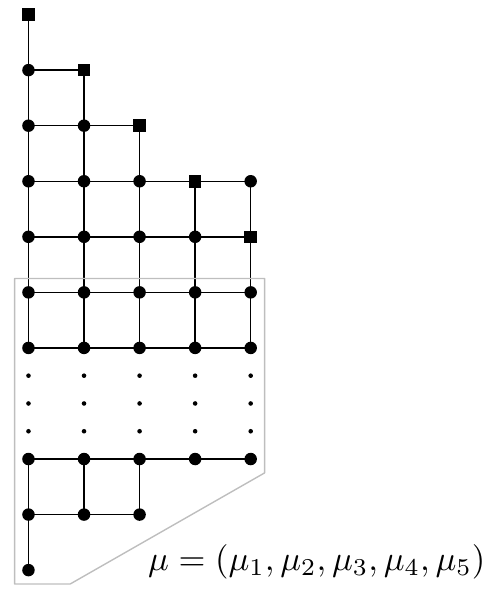} \qquad\qquad
\includegraphics{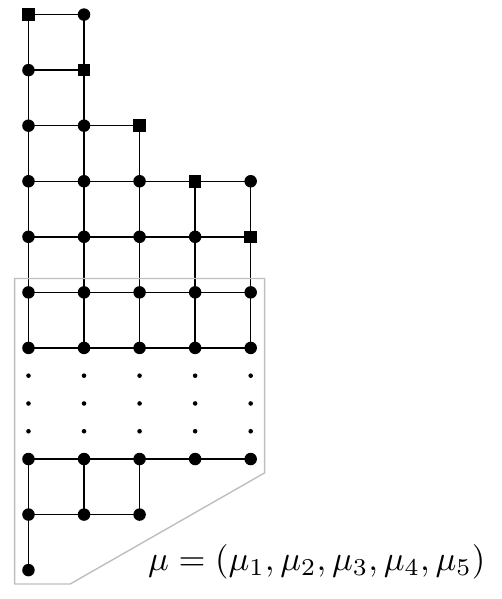}
   \caption{The posets $Q ^+$ on the left and $D_1 (Q)$ on the right.}\label{fig:class8-4:1:2}
   \end{figure}
Then by Lemma \ref{lem:dmkp}, 
\begin{equation}\label{eqn:pf8_1}
\GF (D_{\mu_1 +4 ,1}(Q))=\frac{1}{(q^{|\mu|+17};q)_{\mu_1 +6}}\big(q^{|\mu|+16}\GF (Q ^+)  + (1-q^{2|\mu|+34})\GF (D_1 (Q))\big).
\end{equation}

Since $Q^+=P_5(X)$ and $D_1 (Q)=P_5(X')$ 
for $X$ in \eqref{eq:14} and $X'=\{(\mu,5,1), (\emptyset,2,1),(\emptyset,2,4)\}$,
by Theorem~\ref{thm:attach}, 
\[
\GF (Q ^+) = \frac{(-1)q^{-\sum_{i=1}^5 (i-1)\mu_{i}-23}}{(1-q)^6\prod_{i=1}^5 (q;q)_{\mu_i +5-i}}
\int_{0\le x_1 \le \cdots \le x_5 \le 1} (x_1 -x_2) a_{\mu+\delta_5 }(x_1,\dots, x_5) d_q x_1 \cdots d_q x_5
\]
and
\begin{multline*}
 \GF (D_1 (Q))
= \frac{q^{-\sum_{i=1}^5 (i-1)\mu_{i}-23}}{(1-q)^7\prod_{i=1}^5 (q;q)_{\mu_i +5-i}}\\
\times \int_{0\le x_1 \le \cdots \le x_5 \le 1}(x_1 -x_2) (x_4 -x_5)a_{\mu+\delta_5 }(x_1,\dots, x_5) d_q x_1 \cdots d_q x_5.
\end{multline*}

By Lemma~\ref{lem:expand},
\begin{multline}
  \label{eq:16}
\int_{0\le x_1 \le \cdots \le x_5 \le 1} (x_1 -x_2) a_{\mu+\delta_5 }(x_1,\dots, x_5) d_q x_1 \cdots d_q x_5\\
=\sum_{\ell=1}^5 \frac{(-1)^{5-\ell}(1-q)}{1-q^{|\mu|+16}}
\int_{0\le x_1\le \cdots \le x_4 \le 1} (x_1 -x_2)a_{\widehat{\mu}^{(\ell)}+\delta_4 }(x_1,x_2,x_3,x_4)d_q x_1\cdots d_q x_4,
\end{multline}
\begin{multline}
  \label{eq:17}
\int_{0\le x_1 \le \cdots \le x_5 \le 1} x_4 (x_1 -x_2) a_{\mu+\delta_5 }(x_1,\dots, x_5) d_q x_1 \cdots d_q x_5\\
=\sum_{\ell=1}^5 \frac{(-1)^{5-\ell}(1-q)}{1-q^{|\mu|+17}}
\int_{0\le x_1\le \cdots \le x_4 \le 1} x_4 (x_1 -x_2) a_{\widehat{\mu}^{(\ell)}+\delta_4 }(x_1,x_2,x_3,x_4)d_q x_1\cdots d_q x_4
\end{multline}
and
\begin{multline}
  \label{eq:18}
\int_{0\le x_1 \le \cdots \le x_5 \le 1} x_5 (x_1 -x_2) a_{\mu+\delta_5 }(x_1,\dots, x_5) d_q x_1 \cdots d_q x_5\\
=\sum_{\ell=1}^5 \frac{(-1)^{5-\ell}(1-q)}{1-q^{|\mu|+17}}
\int_{0\le x_1\le \cdots \le x_4 \le 1} (x_1 -x_2) a_{\widehat{\mu}^{(\ell)}+\delta_4 }(x_1,x_2,x_3,x_4)d_q x_1\cdots d_q x_4.
\end{multline}

The $q$-integrals with $4$ variables in \eqref{eq:16}, \eqref{eq:17} and \eqref{eq:18} can be explicitly computed by computer as follows. 
For any partition $\nu\in\Par_4$ and $\epsilon\in\{0,1\}$, we have
\begin{multline}\label{eq:19}
f(\nu, \epsilon):= \int_{0\le x_1\le \cdots \le x_4 \le 1} x_4^\epsilon(x_1 -x_2)a_{\nu +\delta_4 }(x_1,x_2,x_3,x_4)d_q x_1\cdots d_q x_4\\
 =  \frac{(-1)q^{12+\sum_{i=1}^4 (i-1)\nu_{i}}(1-q)^5 (1-q^{|\nu| +11+\epsilon})\prod_{1\le i<j\le 4}(1-q^{\nu_i -\nu_j+j-i})}{\prod_{i=1}^4 (1-q^{|\nu|-\nu_i+6+i})(1-q^{\nu_i+5-i})(1-q^{\nu_i +6-i})}.
\end{multline}
We note that $f(\nu,\epsilon)$ in \eqref{eq:19} does not seem to have a nice factorization formula when $\epsilon\ge2$. 

On the other hand, by explicitly computing the hook lengths of the elements in $P_6 ^{\lambda_1 +4}(X_8 ^{(4)})=D_{\mu_1+4,1}(Q)$, we get 
\[
\prod_{u \in P_6 ^{\lambda_1 +4}(X_8 ^{(4)})}\frac{1}{1-q^{h(u)}}=
\frac{\prod_{i=1}^5 (1-q^{|\mu|+\mu_j -j+23})\prod_{1\le i<j \le 5 }(1-q^{\mu_i -\mu_j +j-i})}
{(q^{|\mu| +16};q)_{\mu_1 +7}\prod_{i=1}^5 (q;q)_{\mu_i +7-i}\prod_{1\le i<j \le 5}(1-q^{|\mu| -\mu_i -\mu_j +i+j+4})}.
\]
Using the above observations we obtain that the hook length property for class 8-(4) is equivalent to
\begin{multline}
  \label{eq:class8-4(2)}
\prod_{j=1}^5\frac{1-q^{|\mu|+\mu_j -j+23}}{(q^{\mu_j+6-j};q)_2}
\prod_{1\le i<j \le 5 }\frac{1-q^{\mu_i -\mu_j +j-i}}{1-q^{|\mu| -\mu_i -\mu_j +i+j+4}}\\
=\sum_{\ell=1}^5\frac{(-1)^{5-\ell}q^{-\sum_{i=1}^5 (i-1)\mu_{i}-23}}{(1-q)^6}
 \left( (1-q^{|\mu|+16}) (1+q^{|\mu|+17}) f(\widehat{\mu}^{(\ell)},1)
-(1-q^{2|\mu|+23}) f(\widehat{\mu}^{(\ell)},0) \right).
\end{multline}
We have verified \eqref{eq:class8-4(2)} by computer.
\end{proof}

%----  class 10 -----------

\subsection{Class 10 : Tagged Swivels}

\begin{thm}\label{thm:class10}
The hook length property holds for class $10$. 
\end{thm}

\begin{proof}

Note that we derived the $q$-integral identity which implies the hook length property for the semi-irreducible 
$d$-complete posets of class $10$ in \eqref{eqn:hlp_class10}. As we did in the proof of Theorem \ref{thm:class8}, 
rather than proving \eqref{eqn:hlp_class10} by computing the $q$-integral directly, for the sake of the simplicity 
of the computation, we express the poset in a different way using Lemma \ref{lem:dmkp}.

Let $Q=P_5 (X)^-$ for $X=\{ (\mu,5,1), ((1),1,2)\}$ and $\mu=\lambda+(1^5)$.
The poset $P_6 ^{\lambda_1 +4}(X_{10})$ (see Figure~\ref{fig:class10}) can be also expressed as $D_{\mu_1 +4 ,1}(Q)$
and, by Lemma~\ref{lem:dmkp}, the $P$-partition generating function satisfies the relation  
$$\GF (D_{\mu_1 +4 ,1}(Q))=
\frac{1}{(q^{|\mu|+17};q)_{\mu_1 +6}}(q^{|\mu|+16}\GF (Q ^+) + (1-q^{2|\mu|+34})\GF (D_1 (Q))).$$
Note that $Q^+=P_5 (X)$ and $D_1 (Q)=P_5 (X')$, where
\[
X=\{ (\mu,5,1), ((1),1,2)\},\qquad X'=\{ (\mu,5,1), ((1),1,2), (\emptyset, 2,4)\}.
\]
By Theorem~\ref{thm:attach}, 
\[
\GF (Q ^+) = \frac{q^{-\sum_{i=1}^5 (i-1)\mu_{i}-23}}{(1-q)^6\prod_{i=1}^5 (q;q)_{\mu_i +5-i}}
\int_{0\le x_1 \le \cdots \le x_5 \le 1} x_2 a_{\mu+\delta_5 }(x_1,\dots, x_5) d_q x_1 \cdots d_q x_5
\]
and
\[
 \GF (D_1 (Q)) = \frac{(-1)q^{-\sum_{i=1}^5 (i-1)\mu_{i}-23}}{(1-q)^7\prod_{i=1}^5 (q;q)_{\mu_i +5-i}}
\int_{0\le x_1 \le \cdots \le x_5 \le 1} x_2 (x_4 -x_5)a_{\mu+\delta_5 }(x_1,\dots, x_5) d_q x_1 \cdots d_q x_5.
\]
By Lemma~\ref{lem:expand},
\begin{multline*}
\int_{0\le x_1 \le \cdots \le x_5 \le 1} x_2 a_{\mu+\delta_5 }(x_1,\dots, x_5) d_q x_1 \cdots d_q x_5  \\
=\sum_{\ell=1}^5 \frac{(-1)^{5-\ell}(1-q)}{1-q^{|\mu|+16}}
\int_{0\le x_1\le \cdots \le x_4 \le 1} x_2a_{\widehat{\mu}^{(\ell)}+\delta_4 }(x_1,x_2,x_3,x_4)d_q x_1\cdots d_q x_4,
\end{multline*}
\begin{multline*}
\int_{0\le x_1 \le \cdots \le x_5 \le 1} x_2x_4 a_{\mu+\delta_5 }(x_1,\dots, x_5) d_q x_1 \cdots d_q x_5  \\
=\sum_{\ell=1}^5 \frac{(-1)^{5-\ell}(1-q)}{1-q^{|\mu|+17}}
\int_{0\le x_1\le \cdots \le x_4 \le 1} x_2x_4a_{\widehat{\mu}^{(\ell)}+\delta_4 }(x_1,x_2,x_3,x_4)d_q x_1\cdots d_q x_4
\end{multline*}
and
\begin{multline*}
\int_{0\le x_1 \le \cdots \le x_5 \le 1} x_2x_5 a_{\mu+\delta_5 }(x_1,\dots, x_5) d_q x_1 \cdots d_q x_5  \\
=\sum_{\ell=1}^5 \frac{(-1)^{5-\ell}(1-q)}{1-q^{|\mu|+17}}
\int_{0\le x_1\le \cdots \le x_4 \le 1} x_2a_{\widehat{\mu}^{(\ell)}+\delta_4 }(x_1,x_2,x_3,x_4)d_q x_1\cdots d_q x_4.
\end{multline*}

The $q$-integrals with $4$ variables in the above 3 equations can be explicitly computed by computer:
For $\nu\in\Par_5$ and an integer $m\ge0$, we have
\begin{multline*}
g(\nu,m):=\int_{0\le x_1\le \cdots \le x_4 \le 1} x_2 x_4^m a_{\mu +\delta_4 }(x_1,x_2,x_3,x_4) d_q x_1\cdots d_q x_4\\
 =  \frac{q^{12+\sum_{i=1}^4 i\mu_{i+1}}(1-q)^4 (1-q^{|\mu| +12})\prod_{1\le i<j\le 4}(1-q^{\mu_i -\mu_j+j-i})}{(1-q^{|\mu|+11+m})\prod_{1\le i<j\le 4}(1-q^{\mu_i +\mu_j +11-i-j})\prod_{i=1}^4 (1-q^{\mu_i+5-i})}.
\end{multline*}
On the other hand, by explicitly computing the hook lengths of the elements in $P_6 ^{\lambda_1 +4}(X_{10})=D_{\mu_1+4,1}(Q)$, we get 
\begin{multline*}
\prod_{u \in P_6 ^{\lambda_1 +4}(X_{10})}\frac{1}{1-q^{h(u)}}= 
\frac{1}{(1-q)(q^{|\mu|+17};q)_{\mu_1 +6}}
\prod_{i=1}^5 \frac{1-q^{|\mu|+\mu_i-i+23}}{(1-q^{|\mu|-\mu_i +10+i})  (q;q)_{\mu_i +6-i}}\\
\times\prod_{1\le i<j\le 5}\frac{1-q^{\mu_i-\mu_j+j-i}}{1-q^{\mu_i+\mu_j-i-j+13}}.
\end{multline*}

Summarizing the above observations, we obtain that the hook length property for class 10 is equivalent to
\begin{multline}\label{eq:class10(2)}
\prod_{i=1}^5 \frac{1-q^{|\mu|+\mu_i-i+23}}{  (1-q^{\mu_i +6-i})(1-q^{|\mu|-\mu_i +10+i})}
\prod_{1\le i<j\le 5}\frac{1-q^{\mu_i-\mu_j+j-i}}{1-q^{\mu_i+\mu_j-i-j+13}}\\
=\sum_{\ell=1}^5\frac{(-1)^{5-\ell}q^{-\sum_{i=1}^5 (i-1)\mu_{i}-23}}{(1-q)^5 (1-q^{|\mu|+16})}
 \left( (1-q^{2|\mu|+23}) g(\widehat{\mu}^{(\ell)},0) 
-(1-q^{|\mu|+16}) (1+q^{|\mu|+17}) g(\widehat{\mu}^{(\ell)},1) \right).
\end{multline}
We have verified \eqref{eq:class10(2)} by computer.

\end{proof}

\subsection{Class 12: Pumps}
\label{sec:class-12}

\begin{thm}\label{thm:class12}
The hook length property holds for class 12.
\end{thm}
 
\begin{proof}
In Section~\ref{sec:class-12:-pumps}, we have shown that the hook length property for class 12 is equivalent to \eqref{eq:class12}.
By \cite[Lemma~5.7]{KimStanton17} we have
\[
a_{\lambda+\delta_4}(x_2,x_3,x_4,x_5) = \frac{-\prod_{i=1}^3(1-q^{\lambda_i+4-i})}{(1-q)^3}
\int_{x_2\le y_1\le x_3\le y_2\le x_4\le y_3 \le x_5} a_{\lambda+\delta_3}(y_1,y_2,y_3) d_q y_1 d_q y_2 d_q y_3.
\]
Thus, 
\eqref{eq:class12} is equivalent to
\begin{multline}\label{eq:class12-1}
 \int_{D}
f(x_1,\dots,x_6) a_{\lambda+\delta_3}(y_1,y_2,y_3) d_q y_1 d_q y_2 d_q y_3 d_q x_1\cdots d_q x_6\\
=\frac{q^{\sum_{i=1}^3 i\lambda_i +42}(1-q)^{15}}{(1-q^3)^2(1-q^4)(q^{|\lambda|+15};q)_3}
 \prod_{i=1}^3 \frac{1-q^{|\lambda|+\lambda_i+25-i}}{(q^{\lambda_i+8-i};q)_3}
 \prod_{1\le i<j \le 3}\frac{1-q^{\lambda_i -\lambda_j +j-i}}{(q^{\lambda_i +\lambda_j +15-i-j};q)_2},
\end{multline}
where $D$ is the list of inequalities $0\le x_1\le x_2\le y_1\le x_3\le y_2\le x_4\le y_3 \le x_5\le x_6\le 1$ and
\[
f(x_1,\dots,x_6) = (x_1-x_2)(x_3-x_4)(x_5-x_6) a_{\delta_3 }(x_1,x_2,x_3) a_{\delta_3}(x_4,x_5,x_6),
\]
Using \cite[Lemma~5.7]{KimStanton17}, we can change the order of integration to obtain
that the left hand side of  \eqref{eq:class12-1} is equal to
\begin{multline*}
 \int_{D} f(x_1,\dots,x_6) a_{\lambda+\delta_3}\left(\frac{y_1}q, \frac{y_2}q, \frac{y_3}q\right) q^{-3} 
d_q x_1\cdots d_q x_6 d_q y_1 d_q y_2 d_q y_3 \\
= q^{-|\lambda|-6} \int_{D} f(x_1,\dots,x_6) a_{\lambda+\delta_3}(y_1,y_2,y_3)
d_q x_1\cdots d_q x_6 d_q y_1 d_q y_2 d_q y_3.
\end{multline*}
Hence, \eqref{eq:class12-1} is equivalent to
\begin{multline}\label{eq:class12-2}
 \int_{0\le y_1\le y_2\le y_3\le 1}
g(y_1,y_2,y_3) a_{\lambda+\delta_3}(y_1,y_2,y_3) d_q y_1 d_q y_2 d_q y_3\\
=\frac{q^{\sum_{i=1}^3 (i+1)\lambda_i +48}(1-q)^{15}}{(1-q^3)^2(1-q^4)(q^{|\lambda|+15};q)_3}
 \prod_{i=1}^3 \frac{1-q^{|\lambda|+\lambda_i+25-i}}{(q^{\lambda_i+8-i};q)_3}
 \prod_{1\le i<j \le 3}\frac{1-q^{\lambda_i -\lambda_j +j-i}}{(q^{\lambda_i +\lambda_j +15-i-j};q)_2},
\end{multline}
where
\begin{multline*}
g(y_1,y_2,y_3)
= \int_{0\le x_1\le x_2\le y_1\le x_3\le y_2\le x_4\le y_3 \le x_5\le x_6\le 1} 
f(x_1,\dots,x_6) d_qx_1\dots d_qx_6\\
= \int_{y_3}^1 \int_{y_3}^{x_6} \int_{y_2}^{y_3} \int_{y_1}^{y_2}\int_{0}^{y_1}  \int_{0}^{x_2} 
(x_1-x_2)(x_3-x_4)(x_5-x_6) a_{\delta_3 }(x_1,x_2,x_3) a_{\delta_3}(x_4,x_5,x_6) d_qx_1\dots d_qx_6.
\end{multline*}
We can evaluate $g(y_1,y_2,y_3)$ by our testing code. Interestingly, it has many factors:
\begin{multline}\label{eq:gy}
g(y_1,y_2,y_3) =  q^2 y_1^4 y_2(y_1-y_2) (y_2-1) (y_2-y_3)(qy_3-1)(y_3-q)(y_3-1)^2\\
\times\frac{(qy_1y_2^2+q^2y_2y_3-qy_2^2y_3-qy_2^2-qy_1y_3+y_2y_3)}
{(q^2+q+1)^4(q^2+1)^2(q+1)^5}.
\end{multline}
By using \eqref{eq:gy}, we have verified \eqref{eq:class12-2} by our testing code.

\end{proof}

%======================================================

\section{Evaluation of the $q$-integrals: arbitrary diagonal type}
\label{sec:eval-q-integr2}

%======================================================

%-----------------------------------------------------------------------

In order to complete the proof of the hook length property for $d$-complete posets, it remains to prove the $q$-integral formulas \eqref{eq:class1}, \eqref{eq:class2}, \eqref{eq:class4} and \eqref{eqn:hlp_class11} for classes 1, 2, 4 and 11. Since they are of arbitrary diagonal type, they have to be proved by hand. Although the hook length property for class 1 (shapes) and class 2 (shifted shapes) are well known, for the sake of completeness, we state them together here. 
We also give another proof of the hook length formula for class 2. To prove the hook length property for classes 2, 4 and 11, we utilize some partial fraction expansion identities.

\subsection{Class 1: Shapes}

\begin{thm}\label{thm:class1}
The hook length property holds for class $1$. 
\end{thm}
\begin{proof}
In Section~\ref{subsec:class1}, we have shown that the hook length property for class 1 is equivalent to \eqref{eq:class1}. This is the special case $k = 1$ of Warnaar's integral \cite[Theorem 1.4]{Warnaar}.
\end{proof}

\subsection{Class 2: Shifted shapes}

\begin{thm}\label{thm:class2}
The hook length property holds for class $2$. 
\end{thm}
\begin{proof}
In Section~\ref{subsec:class2}, we have shown that the hook length property for class 2 is equivalent to \eqref{eq:class2}.
We give a proof of the following lemma which is slightly more general than  \eqref{eq:class2}.
\end{proof}

\begin{lem}\label{lem:alternant}
For nonnegative integers $n,k$ and $\mu\in\Par_n$, we have
\begin{multline}\label{eqn:alternant}
\int_{0\le x_1\le \cdots \le x_n\le 1} (-1)^{\binom{n}{2}}a_{\mu +\delta_n}(x_1,\dots, x_n) x_n ^kd_q x_1 \cdots d_q x_n\\
=(1-q)^n q^{\binom{n+1}{3}+\sum_{i=1}^n (i-1)\mu_i}
\frac{1-q^{|\mu|+\binom{n+1}{2}}}{1-q^{|\mu|+\binom{n+1}{2}+k}}\cdot 
\frac{\prod_{1\le i<j\le n} (1-q^{\mu_i -\mu_j +j-i})}{\prod_{1\le i\le j \le n}(1-q^{2n+1-i-j+\mu_i +\mu_{j+1}})}.
\end{multline}
\end{lem}

\begin{proof}
By Lemma~\ref{lem:expand}, the left hand side of \eqref{eqn:alternant} can be written as 
\[
\frac{1-q}{1-q^{|\mu|+\binom{n+1}{2}+k}}
\sum_{\ell=1}^n (-1)^{1-\ell}\int_{0\le x_1\le \cdots \le x_{n-1}\le 1}(-1)^{\binom{n-1}{2}} a_{\widehat{\mu} ^{(\ell)}+\delta_{n-1}}(x_1,\dots, x_{n-1}) d_q x_1 \cdots d_q x_{n-1}.
\]
Then, by induction, proving \eqref{eqn:alternant} is equivalent to showing the following identity
\begin{multline*}
(1-q)^n q^{\binom{n+1}{3}+\sum_{i=1}^n (i-1)\mu_i}
\frac{1-q^{|\mu|+\binom{n+1}{2}}}{1-q^{|\mu|+\binom{n+1}{2}+k}}\cdot 
\frac{\prod_{1\le i<j\le n}(1-q^{\mu_i -\mu_j +j-i})}{\prod_{1\le i\le j \le n}(1-q^{2n+1-i-j+\mu_i +\mu_{j+1}})}\\
=\frac{1-q}{1-q^{|\mu|+\binom{n+1}{2}+k}}
\sum_{\ell=1}^n (-1)^{1-\ell}(1-q)^{n-1} q^{\binom{n}{3}+\sum_{i=1}^{n-1} (i-1)\widehat{\mu}_i ^{(\ell)}}\\
\times \frac{\prod_{1\le i<j\le n-1}(1-q^{\widehat{\mu}_i ^{(\ell)}-\widehat{\mu}_j ^{(\ell)}+j-i})}{\prod_{1\le i\le j \le n-1}(1-q^{2n-1-i-j+\widehat{\mu}_i ^{(\ell)}+\widehat{\mu}_{j+1} ^{(\ell)}})},
\end{multline*}
or, 
\begin{multline}\label{eq:21}
q^{\binom{n}{2}}(1-q^{|\mu|+\binom{n+1}{2}})
=\sum_{\ell=1}^n (-1)^{1-\ell}q^{-\sum_{i=1}^n (i-1)\mu_i +\sum_{i=1}^{n-1}(i-1)\widehat{\mu}_i ^{(\ell)}}\\
\times \frac{\prod_{1\le i<j\le n-1}(1-q^{\widehat{\mu}_i ^{(\ell)}-\widehat{\mu}_j ^{(\ell)}+j-i})}{\prod_{1\le i<j\le n}(1-q^{\mu_i -\mu_j +j-i})} \cdot
\frac{\prod_{1\le i\le j \le n}(1-q^{2n+1-i-j+\mu_i +\mu_{j+1}})}{\prod_{1\le i\le j \le n-1}(1-q^{2n-1-i-j+\widehat{\mu}_i ^{(\ell)}+\widehat{\mu}_{j+1} ^{(\ell)}})}.
\end{multline}
By considering the parts of $\widehat{\mu} ^{(\ell)}$ specified in \eqref{eqn:la_j}, it is not difficult to check 
\begin{equation}
  \label{eq:1}
-\sum_{i=1}^n (i-1)\mu_i +\sum_{i=1}^{n-1}(i-1)\widehat{\mu}_i ^{(\ell)} =\binom{\ell-1}{2}-(\ell-1)\mu_\ell -\sum_{j=\ell+1}^n \mu_j,  
\end{equation}
\begin{equation}
  \label{eq:2}
\frac{\prod_{1\le i<j\le n-1}(1-q^{\widehat{\mu}_i ^{(\ell)}-\widehat{\mu}_j ^{(\ell)}+j-i})}{\prod_{1\le i<j\le n}(1-q^{\mu_i -\mu_j +j-i})}  
= \frac{(-1)^{\ell-1}q^{-\sum_{i=1}^{\ell-1}\mu_i +(\ell-1)\mu_\ell -\binom{\ell}{2}}}{\prod_{i=1, i\ne \ell}^n (1-q^{\mu_\ell -\mu_i +i-\ell}) }
\end{equation}
and 
\begin{equation}
  \label{eq:3}
\frac{\prod_{1\le i\le j \le n}(1-q^{2n+1-i-j+\mu_i +\mu_{j+1}})}{\prod_{1\le i\le j \le n-1}(1-q^{2n-1-i-j+\widehat{\mu}_i ^{(\ell)}+\widehat{\mu}_{j+1} ^{(\ell)}})}
= (1-q^{n+1-\ell+\mu_\ell})\prod_{i=1, i\ne \ell}^n (1-q^{2n+2-\ell-i+\mu_i+\mu_\ell}).  
\end{equation}
By \eqref{eq:1}, \eqref{eq:2} and \eqref{eq:3}, we can rewrite \eqref{eq:21} as
\begin{equation}\label{eqn:alt_identity}
q^{\binom{n+1}{2}+|\mu|}(1-q^{\binom{n+1}{2}+|\mu|})
=\sum_{\ell=1}^n q^{n+1-\ell+\mu_\ell}(1-q^{n+1-\ell+\mu_\ell}) \prod_{i=1, i\ne \ell}^n\frac{1-q^{2n+2-\ell-i+\mu_i+\mu_\ell}}{1-q^{\mu_\ell -\mu_i +i-\ell}}.
\end{equation}
If we let $a_j=q^{n+1-j+\mu_j}$, then this identity can also be written as 
\begin{equation}\label{eqn:a_id}
a_1 \cdots a_n (1-a_1\cdots a_n )=\sum_{\ell=1}^n a_\ell (1-a_\ell) \prod_{i=1, i\ne \ell}^n\frac{1-a_i a_\ell}{1-a_\ell/a_i}.
\end{equation}
To prove \eqref{eqn:a_id}, we note the following partial fraction expansion \cite[p.81-83]{Rainville71} (cf. \cite[(7.13) in Lemma 7.9]{Milne1988})

\begin{equation}\label{eqn:pfexp1}
\prod_{i=1}^n \frac{1-t x_i y_i}{1-t x_i}
= y_1\cdots y_n +\sum_{\ell=1}^n \frac{1-y_\ell}{1-t x_\ell}\prod_{i=1, i\ne \ell}^n \frac{1-x_i y_i/x_\ell}{1-x_i/x_\ell}.
\end{equation}
Take $x_i \mapsto a_i$ and $y_i \mapsto 1/a_i ^2$ in \eqref{eqn:pfexp1} and get 
\begin{equation}\label{eqn:pfexp2}
1-\prod_{i=1}^n \frac{a_i (a_i -t)}{1-t a_i}=\sum_{\ell=1}^n \frac{1-a_\ell ^2}{1-t a_\ell}\prod_{i=1, i\ne \ell}^n \frac{1-a_i a_\ell}{1-a_\ell/a_i}.
\end{equation}
Then \eqref{eqn:a_id} is the result obtained by subtracting $t=-1$ case of \eqref{eqn:pfexp2} from the $t=0$ case of \eqref{eqn:pfexp2}.
\end{proof}

We remark that Lemma \ref{lem:alternant} proves the shifted version of Warnaar's $q$-integral (cf. \cite[Theorem 8.16]{KimStanton17})
via explicit computation of the $q$-integral.

%-----------------------------------------------------------------------

\subsection{Class 4 : Insets}

%----  class 4 -----------

\begin{thm}\label{thm:class4}
The hook length property holds for class $4$. 
\end{thm}

\begin{proof}
In Section \ref{subsec:class4}, we figured that the hook length property for class $4$ is equivalent to \eqref{eq:class4}.
By Lemma~\ref{lem:expand}, the left hand side of \eqref{eq:class4} is 
\begin{align}\label{eq:4lsh}
&\int_{0\le x_1 \le \cdots \le x_{n+1}\le 1} x_n^k a_{\lambda +\delta_{n-1}}(x_1,\dots, x_{n-1})a_{\mu +\delta_{n+1}}(x_1,\dots, x_{n+1})d_q x_1 \cdots d_q x_{n+1}\\
&=\frac{1-q}{1-q^{|\lambda|+|\mu|+k+n^2+2}}\sum_{\ell=1}^{n+1}(-1)^{n+1-\ell} Y_\ell, \nonumber
\end{align}
where
\[
Y_\ell = \int_{0\le x_1\le \cdots \le x_n \le 1} x_n^k a_{\lambda+\delta_{n-1}}(x_1,\dots, x_{n-1})
 a_{\widehat{\mu}^{(\ell)}+\delta_n}(x_1,\dots, x_n) d_q x_1\cdots d_q x_n.
\]

Now we compute $Y_\ell$ for a fixed integer $\ell$. Let $Q=P_n(X)$, where
\[
X=\{(\lambda,n-1,1),(\widehat{\mu}^{(\ell)},n,1),((k),1,n)\}.
\]
Then the number $\ell_i$ of elements of level $i$ in $Q$ is $\ell_i=\lambda_i+\widehat{\mu}^{(\ell)}_{i+1}+2i-1$ for $1\le i\le n-1$ 
and $\ell_n = \widehat{\mu}^{(\ell)}_1+n+k$. By Theorem~\ref{thm:attach}, 
\begin{equation}
  \label{eq:Q1}
\GF(Q)=Y_\ell \cdot \frac{(-1)^{n-1} q^{-\sum_{i=1}^{n-1}i(\lambda_i +\widehat{\mu}_{i+1}^{(\ell)})-\frac{1}{6}n(n-1)(2n-1)}}{(1-q)^n (q;q)_k}
\prod_{i=1}^{n-1}\frac{1}{(q;q)_{\lambda_i+n-1-i}}
\prod_{i=1}^n\frac{1}{(q;q)_{\widehat{\mu}_i ^{(\ell)}+n-i}} .
\end{equation}
On the other hand, $Q^-$ is the disjoint union of $\nu$ and a chain with $k$ elements, where $\nu$ is the Young poset obtained from the rectangular shape $((n-1)^n)$ by attaching $\mu$ to the right and the transpose of $\lambda$ at the bottom. Thus, by Lemma~\ref{lem:P+} and 
Theorem~\ref{thm:class1},
\begin{multline}
  \label{eq:Q2}
\GF(Q) = \frac{1}{1-q^{|\lambda|+|\widehat{\mu}^{(\ell)}|+n(n-1)+k+1}} \GF(Q^-) \\
=  \frac{1}{1-q^{|\lambda|+|\mu|-\mu_\ell+\ell+n^2-n+k}} \cdot \frac{1}{(q;q)_k}
 \prod_{\substack{1\le i\le n\\ 1\le j \le n-1}}\frac{1}{1-q^{\widehat{\mu}_i ^{(\ell)}+\lambda_j +2n-i-j}}\\
\times \frac{\prod_{1\le i< j\le n} (1-q^{\widehat{\mu}_i ^{(\ell)}-\widehat{\mu}_j ^{(\ell)}+j-i})}{\prod_{i=1}^n (q;q)_{\widehat{\mu}_i ^{(\ell)}+n-i}}
\cdot\frac{\prod_{1\le i< j\le n-1}(1-q^{\lambda_i -\lambda_j +j-i})}{\prod_{i=1}^{n-1} (q;q)_{\lambda_i+n-1-i}}.
\end{multline}
By \eqref{eq:Q1} and \eqref{eq:Q2}, 
\begin{multline}
  \label{eq:Y_ell}
Y_\ell=   \frac{(-1)^{n-1}q^{\sum_{i=1}^{n-1}i(\lambda_i +\widehat{\mu}_{i+1}^{(\ell)})+\frac{1}{6}n(n-1)(2n-1)}(1-q)^{n}}
{1-q^{|\lambda|+|\mu|-\mu_\ell+\ell+n^2-n+k}}\\
 \times \frac{\prod_{1\le i< j\le n} (1-q^{\widehat{\mu}_i ^{(\ell)}-\widehat{\mu}_j ^{(\ell)}+j-i}) \prod_{1\le i< j\le n-1}(1-q^{\lambda_i -\lambda_j +j-i})}
{\prod_{\substack{1\le i\le n\\ 1\le j \le n-1}} (1-q^{\widehat{\mu}_i ^{(\ell)}+\lambda_j +2n-i-j})}.
\end{multline}

By \eqref{eq:4lsh} and \eqref{eq:Y_ell} and the identities \eqref{eq:1}, \eqref{eq:2} and
\[
\frac{\prod_{\substack{1\le i\le n+1\\ 1\le j \le n-1}}(1-q^{\mu_i+\lambda_j +2n-i-j+1})}
{\prod_{\substack{1\le i\le n\\ 1\le j \le n-1}} (1-q^{\widehat{\mu}_i ^{(\ell)}+\lambda_j +2n-i-j})}
=\prod_{j=1}^{n-1}(1-q^{\mu_\ell+\lambda_j+2n-\ell-j+1}),
\]
we can rewrite \eqref{eq:class4} as
\begin{multline}\label{eqn:cl4_pf1}
\frac{\prod_{j=1}^{n-1}(1-q^{|\lambda|+|\mu|+\lambda_j +n^2 +n-j+k+1})}{\prod_{i=1}^{n+1}(1-q^{|\lambda|+|\mu|-\mu_i +n^2 -n+k+i})}\\
= \sum_{\ell=1}^{n+1}\frac{q^{-|\lambda|-|\mu|+\mu_\ell-n^2+n-k-\ell}}{1-q^{|\lambda|+|\mu|-\mu_\ell+n^2-n+k+\ell}}
\cdot\frac{\prod_{j=1}^{n-1}(1-q^{\mu_\ell+\lambda_j+2n-\ell-j+1})}{\prod_{j=1,j\ne \ell}^{n+1}(1-q^{\mu_\ell-\mu_j+j-\ell})}.
\end{multline}
To prove \eqref{eqn:cl4_pf1}, we remark a partial fraction expansion \cite[p.451]{WW}
\begin{equation}\label{eqn:ww_pfe}
\frac{\prod_{j=1}^{n+1}(1-b_j/t)}{\prod_{j=1}^n (1-a_j/t)}=\sum_{\ell=1}^n 
\frac{\prod_{j=1}^{n+1}(1-a_\ell /b_j)}{(1-a_\ell/t)\prod_{j=1,j\ne \ell}^{n}(1-a_\ell/a_j)},
\quad \text{ for } b_1\cdots b_{n+1} =a_1\cdots a_n t.
\end{equation}
Let $n\mapsto n+1$ and 
\begin{align*}
a_i &= q^{i-\mu_i-n-2-k}, ~ 1\le i\le n+1,\\
b_i &= q^{\lambda_i -i+n-1-k},~ 1\le i \le n-1,\\
b_ n &= tc,~ b_{n+1}=tc, ~ b_{n+2} =1/c^2\text{ for some $c$, and }\\
t&= q^{-|\lambda|-|\mu|-n^2-2-2k}.
\end{align*} 
With this substitution, we can check that the condition $b_1\cdots b_{n+2} =a_1\cdots a_{n+1} t$ is satisfied, 
and the partial fraction expansion becomes 
$$\frac{\prod_{j=1}^{n-1}(1-b_j/t)(1-c)^2 (1-1/tc^2)}{\prod_{j=1}^{n+1} (1-a_j/t)}=
\sum_{\ell=1}^{n+1} \frac{\prod_{j=1}^{n-1}(1-a_\ell /b_j)(1-a_\ell/tc)^2(1-a_\ell c^2)}{(1-a_\ell/t)\prod_{j=1,j\ne \ell}^n(1-a_\ell/a_j)}.$$
If we divide both sides by $c^2$ and take the limit $c\rightarrow \infty$, 
then we obtain 
$$\frac{\prod_{j=1}^{n-1}(1-b_j/t)}{\prod_{j=1}^{n+1} (1-a_j/t)}=
\sum_{\ell=1}^{n+1} \frac{(-a_\ell)\prod_{j=1}^{n-1}(1-a_\ell /b_j)}{(1-a_\ell/t)\prod_{j=1,j\ne \ell}^{n+1}(1-a_\ell/a_j)},$$
or, by substituting $a_i$'s, $b_i$'s and $t$,
\begin{multline}\label{eqn:cl4_pf2}
\frac{\prod_{j=1}^{n-1}(1-q^{|\lambda|+|\mu|+\lambda_j +n^2 +n-j+k+1})}{\prod_{i=1}^{n+1}(1-q^{|\lambda|+|\mu|-\mu_i +n^2 -n+k+i})}\\
= \sum_{\ell=1}^{n+1}\frac{-q^{-n-\mu_\ell+\ell-k-2}}{1-q^{|\lambda|+|\mu|-\mu_\ell+n^2-n+k+\ell}} \cdot
\frac{\prod_{j=1}^{n-1}(1-q^{j+\ell-2n-\mu_\ell-\lambda_j-1})}{\prod_{j=1,j\ne \ell}^{n+1}(1-q^{\mu_j -\mu_\ell+\ell-j})},
\end{multline}
which is equivalent to \eqref{eqn:cl4_pf1}.

\end{proof}

\begin{remark}
The partial fraction expansion \eqref{eqn:ww_pfe} that we use is actually written in an elliptic form in \cite{WW}, and 
\eqref{eqn:ww_pfe} can be derived by setting $p=0$ in equation $(4.3)$ of Rosengren's paper \cite{Rosengren04}.
We remark that the elliptic
partial fraction identity plays an essential role in the theory of
elliptic hypergeometric series associated to the
root system $A_n$, see \cite{Rosengren04}.
\end{remark}

%-----------------------------------------------------------------------

\subsection{Class 11: Swivel shifted shapes}

%----  class 11 -----------
\begin{thm}\label{thm:class11}
The hook length property holds for class $11$. 
\end{thm}

\begin{proof}
In Section~\ref{sec:class-11}, we expressed the hook length property for class $11$ as the $q$-integral formula \eqref{eqn:hlp_class11}. However, instead of proving this formula directly, we consider another equivalent version of the hook length property. 

First of all, let us assume that the element $A$ in Figure~\ref{fig:class11} is not there (i.e., $\epsilon = 0$).
Let $\mu = \lambda +(1^n)$, i.e., $\mu_i =\lambda_i +1$, for $1\le i \le n$. 
Observe that $P_{n+1}^{\lambda_1 +n-1}(X_{11})=D_{\mu_1 +n-2,k}(Q)$, where $Q=((\mu+\delta_{n+1})^*)^-$.
Since $Q^+=(\mu+\delta_{n+1})^*$, by Theorem~\ref{thm:class2}, we have
\[
\GF(Q^+) = \prod_{u\in (\mu+\delta_{n+1})^*}\frac{1}{1-q^{h(u)}}.  
\]
By applying Lemma~\ref{lem:dmkp}, the $P$-partition generating function for the poset $P_{n+1}^{\lambda_1 +n-1}(X_{11})$ can be written as 
\begin{multline}\label{eq:15}
\GF(P_{n+1}^{\lambda_1 +n-1}(X_{11}))=\frac{1}{(q^{|\mu|+\binom{n+1}{2}+k};q)_{\mu_1 +n}}
\Bigg( \frac{q^{|\mu|+\binom{n+1}2}}{(q;q)_{k-1}} \prod_{u\in (\mu+\delta_{n+1})^*}\frac{1}{1-q^{h(u)}}  \\
+(1-q^{2|\mu|+2\binom{n+1}{2}+2k}) \GF(D_{k}(Q))\Bigg),
\end{multline}

Note that $D_k(Q)=P_n(X)$ for $X=\{(\mu,n,1),((k-1,0),2,n-1)\}$. Thus, by Theorem~\ref{thm:attach}, 
\begin{multline*}
\GF(D_{k}(Q)) =   
\frac{q^{-\sum_{i=1}^n (i-1)(\mu_i +n+1-i)}}{(1-q)^{n}(q;q)_{k} \prod_{i=1}^n (q;q)_{\mu_i +n-i}}\\
 \times \int_{0\le x_1\le \cdots \le x_n \le 1}(-1)^{\binom{n}{2}}a_{\mu +\delta_n }(x_1,\dots, x_n)(x_n ^k -x_{n-1}^k)d_q x_1\cdots d_q x_n.
\end{multline*}

By Lemmas~\ref{lem:expand}, \ref{lem:alternant} and \eqref{eq:1}, we have
\begin{equation}
  \label{eq:FXY}
\GF(D_{k}(Q)) = Y-Z,  
\end{equation}
where
\[
Y = \frac{1}{(q;q)_{k} \prod_{i=1}^n (q;q)_{\mu_i +n-i}}\cdot
\frac{1-q^{|\mu|+\binom{n+1}{2}}}{1-q^{|\mu|+\binom{n+1}{2}+k}}\cdot
\frac{\prod_{1\le i<j\le n} (1-q^{\mu_i -\mu_j +j-i})}{\prod_{1\le i\le j \le n}(1-q^{2n+1-i-j+\mu_i +\mu_{j+1}})},
\]
\begin{multline*}
Z=\frac{q^{-\binom{n}{2}}}{1-q^{|\mu|+\binom{n+1}{2}+k}} \cdot 
\frac{1}{(q;q)_{k} \prod_{i=1}^n (q;q)_{\mu_i +n-i}}\\
\times  \sum_{\ell=1}^n (-1)^{\ell-1} q^{\binom{\ell-1}2-(\ell-1)\mu_\ell -\sum_{j=\ell+1}^{n}\mu_j}
\frac{1-q^{|\widehat{\mu}^{(\ell)}|+\binom{n}{2}}}{1-q^{|\widehat{\mu}^{(\ell)}|+\binom{n}{2}+k}}\cdot
\frac{\prod_{1\le i<j\le n-1} (1-q^{\widehat{\mu}_i ^{(\ell)}-\widehat{\mu}_j ^{(\ell)}+j-i})}{\prod_{1\le i\le j \le n-1}(1-q^{2n-1-i-j+\widehat{\mu}_i ^{(\ell)}+\widehat{\mu}_{j+1} ^{(\ell)}})},
\end{multline*}
where $\widehat{\mu}^{(\ell)}$ is as defined in \eqref{eqn:la_j}.
By \eqref{eq:shifted_hook}, we have
\begin{equation}
  \label{eq:Y}
Y = \frac{1}{(q;q)_{k}}\cdot
\frac{1-q^{|\mu|+\binom{n+1}{2}}}{1-q^{|\mu|+\binom{n+1}{2}+k}}
\prod_{u\in (\mu+\delta_{n+1})^*}\frac{1}{1-q^{h(u)}}.
\end{equation}
By \eqref{eq:shifted_hook}, \eqref{eq:2} and \eqref{eq:3}, we have
\begin{multline}\label{eq:Z}
Z=\frac{q^{-|\mu|-\binom{n}{2}}}{1-q^{|\mu|+\binom{n+1}{2}+k}} \cdot 
\frac{1}{(q;q)_{k}}
\prod_{u\in (\mu+\delta_{n+1})^*} \frac{1}{1-q^{h(u)}}\\
\times  \sum_{\ell=1}^n q^{\mu_\ell-\ell+1}
\frac{1-q^{|\widehat{\mu}^{(\ell)}|+\binom{n}{2}}}{1-q^{|\widehat{\mu}^{(\ell)}|+\binom{n}{2}+k}}
(1-q^{n+1-\ell+\mu_\ell})
\prod_{i=1, i\ne \ell}^n\frac{1-q^{2n+2-\ell-i+\mu_i+\mu_\ell}}{1-q^{\mu_\ell -\mu_i +i-\ell}}.
\end{multline}

On the other hand, specific computation of the hook lengths of the elements in $P_{n+1}^{\lambda_1 +n-1}(X_{11})$ gives 
\begin{multline}\label{eq:P11_hook}
\prod_{u\in P_{n+1}^{\lambda_1 +n-1}(X_{11})}\frac{1}{1-q^{h(u)}}\\
= \frac{\prod_{j=1}^n (1-q^{|\mu|+\mu_j +k +\frac{1}{2}n(n+3)-j+1})}{(q;q)_{k-1}(q^{|\mu|+\binom{n+1}{2}+k};q)_{\mu_1 +n}\prod_{j=1}^{n} (1-q^{|\mu|-\mu_j +k +\binom{n}{2}+j-1})}
\prod_{u\in (\mu+\delta_{n+1})^*}\frac{1}{1-q^{h(u)}} .
\end{multline}

Then by \eqref{eq:15}, \eqref{eq:FXY}, \eqref{eq:Y}, \eqref{eq:Z} and \eqref{eq:P11_hook}, to prove the hook length property of the semi-irreducible $d$-complete posets of class $11$, we need to prove the identity 
\begin{multline}\label{eqn:pf1}
\prod_{j=1}^n \frac{1-q^{|\mu|+\mu_j +k +\frac{n(n+3)}{2}-j+1}}{1-q^{|\mu|-\mu_j +k +\binom{n}{2}+j-1}}
= \frac{1-q^{2|\mu|+2\binom{n+1}{2}+k}}{1-q^k}\\
 -q^{-\binom{n}{2}-|\mu|}\frac{1+q^{|\mu|+\binom{n+1}{2}+k}}{1-q^k}
\sum_{\ell=1}^n q^{\mu_\ell-\ell+1}\frac{1-q^{\binom{n}{2}+|\mu|-\mu_\ell +\ell-1}}{1-q^{\binom{n}{2}+|\mu|-\mu_\ell +\ell+k-1}}
(1-q^{n+1-\ell+\mu_\ell})\\
\times \prod_{j=1, j\ne \ell}^n \frac{1-q^{2n+2-j-\ell+\mu_j+\mu_\ell}}{1-q^{\mu_\ell -\mu_j+j-\ell}}.
\end{multline}
By replacing the fraction 
\[
\frac{1-q^{\binom{n}{2}+|\mu|-\mu_\ell +\ell-1}}{1-q^{\binom{n}{2}+|\mu|-\mu_\ell +\ell+k-1}}
=1-q^{\binom{n}{2}+|\mu|-\mu_\ell +\ell-1} \frac{1-q^k}{1-q^{\binom{n}{2}+|\mu| -\mu_\ell +\ell+k-1}},
\]
we can rewrite \eqref{eqn:pf1} as 
\begin{multline}\label{eqn:pf2}
\prod_{j=1}^n \frac{1-q^{|\mu|+\mu_j +k +\frac{n(n+3)}{2}-j+1}}{1-q^{|\mu|-\mu_j +k +\binom{n}{2}+j-1}}
= \frac{1-q^{2|\mu|+2\binom{n+1}{2}+k}}{1-q^k}\\
 -q^{-\binom{n}{2}-|\mu|}\frac{1+q^{|\mu|+\binom{n+1}{2}+k}}{1-q^k} \sum_{\ell=1}^n q^{\mu_\ell-\ell+1}(1-q^{n+1-\ell+\mu_\ell})
\prod_{j=1, j\ne \ell}^n\frac{1-q^{2n+2-j-\ell+\mu_j+\mu_\ell}}{1-q^{\mu_\ell -\mu_j+j-\ell}}\\
 + (1+q^{|\mu| +\binom{n+1}{2}+k})\sum_{\ell=1}^n 
\frac{1-q^{n+1-\ell+\mu_\ell}}{1-q^{\binom{n}{2}+|\mu| -\mu_\ell +\ell+k-1}}
\prod_{j=1, j\ne \ell}^n \frac{1-q^{2n+2-j-\ell+\mu_j+\mu_\ell}}{1-q^{\mu_\ell -\mu_j+j-\ell}}.
\end{multline}
By \eqref{eqn:alt_identity}, we know that 
\[
q^{\binom{n}{2}+|\mu|}(1-q^{\binom{n+1}{2}+|\mu|})=\sum_{\ell=1}^n q^{\mu_\ell -\ell+1} (1-q^{n+1-\ell+\mu_\ell})
\prod_{j=1, j\ne \ell}^n\frac{1-q^{2n+2-j-\ell+\mu_j+\mu_\ell}}{1-q^{\mu_\ell -\mu_j+j-\ell}},
\]
thus \eqref{eqn:pf2} can be again changed to 
\begin{multline}\label{eqn:pf3}
\prod_{j=1}^n  \frac{1-q^{|\mu|+\mu_j +k +\frac{n(n+3)}{2}-j+1}}{1-q^{|\mu|-\mu_j +k +\binom{n}{2}+j-1}}\\
= q^{|\mu|+\binom{n+1}{2}}
+ (1+q^{|\mu| +\binom{n+1}{2}+k})\sum_{\ell=1}^n \frac{1-q^{n+1-\ell+\mu_\ell}}{1-q^{\binom{n}{2}+|\mu| -\mu_\ell +\ell+k-1}}
\prod_{j=1, j\ne \ell}^n\frac{1-q^{2n+2-j-\ell+\mu_j+\mu_\ell}}{1-q^{\mu_\ell -\mu_j+j-\ell}}.
\end{multline}
To rewrite the left hand side of \eqref{eqn:pf3}, we use the identity \eqref{eqn:pfexp1} with 
$$x_i \mapsto q^{i-\mu_i-1-n}, \quad y_i\mapsto q^{2\mu_i -2i+2+2n},
\text{ and } t=q^{|\mu|+k+\binom{n+1}{2}}.$$
Then 
\begin{multline}\label{eqn:pf4}
\prod_{j=1}^n \frac{1-q^{|\mu|+\mu_j +k +\frac{n(n+3)}{2}-j+1}}{1-q^{|\mu|-\mu_j +k +\binom{n}{2}+j-1}}\\
= q^{2|\mu|+2\binom{n+1}{2}}+\sum_{\ell=1}^n \frac{1-q^{2+2n-2\ell+2\mu_\ell}}{1-q^{\binom{n}{2}+|\mu|-\mu_\ell +\ell+k-1}}
\prod_{j=1, j\ne \ell}^n \frac{1-q^{2n+2-j-\ell+\mu_j+\mu_\ell}}{1-q^{\mu_\ell -\mu_j +j-\ell}}.
\end{multline}
Using \eqref{eqn:pf4} for the left hand side of \eqref{eqn:pf3} and simplifying the terms gives 
\begin{multline*}
q^{|\mu|+\binom{n+1}{2}} -q^{2|\mu|+2\binom{n+1}{2}}
 =  \sum_{\ell=1}^n \frac{1-q^{2+2n-2\ell+2\mu_\ell}}{1-q^{\binom{n}{2}+|\mu|-\mu_\ell +\ell+k-1}}
\prod_{j=1, j\ne \ell}^n \frac{1-q^{2n+2-j-\ell+\mu_j+\mu_\ell}}{1-q^{\mu_\ell -\mu_j +j-\ell}}\\
 - (1+q^{|\mu| +\binom{n+1}{2}+k})\sum_{\ell=1}^n 
\frac{1-q^{n+1-\ell+\mu_\ell}}{1-q^{\binom{n}{2}+|\mu| -\mu_\ell +\ell+k-1}}
\prod_{j=1, j\ne \ell}^n\frac{1-q^{2n+2-j-\ell+\mu_j+\mu_\ell}}{1-q^{\mu_\ell -\mu_j+j-\ell}},
\end{multline*}
or 
\[
q^{|\mu|+\binom{n+1}{2}}(1-q^{|\mu|+\binom{n+1}{2}})\\
=\sum_{\ell=1}^n q^{n+1-\ell+\mu_\ell}(1-q^{n+1-\ell+\mu_\ell})
\prod_{j=1, j\ne \ell}^n\frac{1-q^{2n+2-j-\ell+\mu_j+\mu_\ell}}{1-q^{\mu_\ell -\mu_j+j-\ell}},
\]
which is exactly \eqref{eqn:alt_identity} from the proof for class 2.

If the element $A$ exists in the poset $P_{n+1}^{\lambda_1 +n-1}(X_{11})$ (i.e., when $\epsilon =1$), 
then we have $n+1$ many integration variables and $\mu$ is equal to $\lambda$. 
In this case, the above computation still holds, by replacing $n$ by $n+1$ and keeping $\mu=\lambda$. 
\end{proof}

%-------------------------------------------------------------------------------------------------

\appendix
\section{Figures of semi-irreducible $d$-complete posets}\label{apdx:figs}

Here, we provide the pictures of all $15$ classes of semi-irreducible $d$-complete posets. 
In each picture, when there is $\lambda$, the Young poset of the transpose of $\lambda$ is drawn. 
There are no restrictions on $\lambda$ and $\mu$ except their lengths. 

Our description is slightly different from the original description in \cite[Table 1]{Proctor1999} as follows:
\begin{itemize}
\item Class 8: In \cite[Table 1]{Proctor1999}, the circled element in Figures~\ref{fig:class8-1}, \ref{fig:class8-2}, \ref{fig:class8-3} and \ref{fig:class8-4} may or may not be in the poset. In our setting, if the circled element is missing, the poset belongs to Class 10 in Figure~\ref{fig:class10}.
\item Class 9: In \cite[Table 1]{Proctor1999}, the circled element in Figures~\ref{fig:class9-1} and \ref{fig:class9-2} may or may not be in the poset. In our setting, if the circled element is missing, the poset belongs to Class 10 in Figure~\ref{fig:class10}.
\item Class 10: In Class 10, the number of diagonal entries is at least $4$. In \cite[Table 1]{Proctor1999}, this number is at least $6$. Our classification for the number of diagonal entries being $4$ or $5$ correspond to Class 8 and class 9 in Figures~\ref{fig:class8-1}, \ref{fig:class8-2}, \ref{fig:class8-3}, \ref{fig:class8-4}, \ref{fig:class9-1} and \ref{fig:class9-2} with the circled element removed. 
\end{itemize}

\begin{figure}[H]
  \centering
\includegraphics{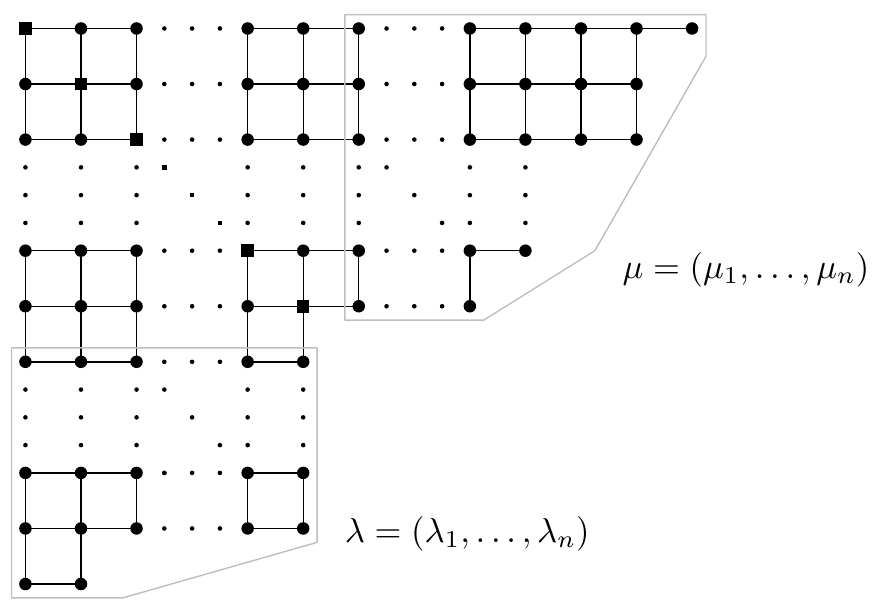}  
  \caption{(Shapes) A semi-irreducible $d$-complete poset of class $1$, $P_n(X_1)$ for $n\ge2$, $\lambda,\mu\in \Par_n$ and $X_1=\{(\lambda,n,1),(\mu,n,1)\}$. This is irreducible if and only if $\lambda_1=\lambda_2$ and $\mu_1=\mu_2$.}
  \label{fig:class1}
\end{figure}

\begin{figure}[H]
  \centering
\includegraphics{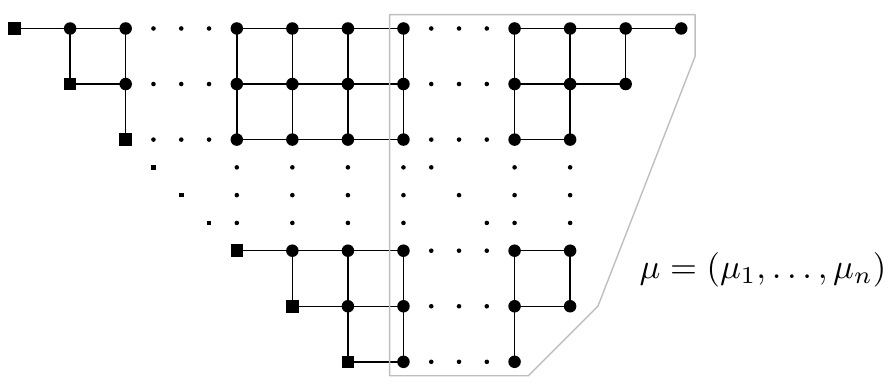}  
  \caption{(Shifted shapes) A semi-irreducible $d$-complete poset of class $2$, $P_n (X_2)$ for $n\ge4$, $\mu\in \Par_n$ and $X_2=\{(\mu,n,1)\}$. This is irreducible if and only if $\mu_1=\mu_2$.}
  \label{fig:class2}
\end{figure}

% class 3
\begin{figure}[H]
  \centering
\includegraphics{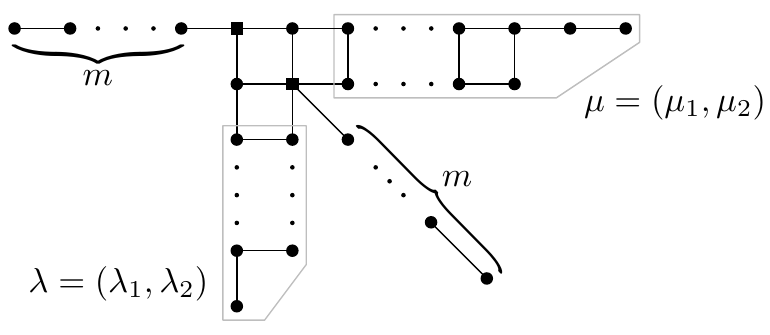}  
  \caption{(Birds) A semi-irreducible $d$-complete poset of class $3$, $P_2 ^m(X_3)$ for $m\ge0$, $\lambda,\mu\in\Par_2$ and $X_3 = \{(\lambda, 2,1), (\mu,2,1), ((m), 1, 1)\}$. This is irreducible if and only if $\lambda_1=\lambda_2$ and $\mu_1=\mu_2$.}
  \label{fig:class3}
\end{figure}

% class 4
\begin{figure}[H]
  \centering
\includegraphics{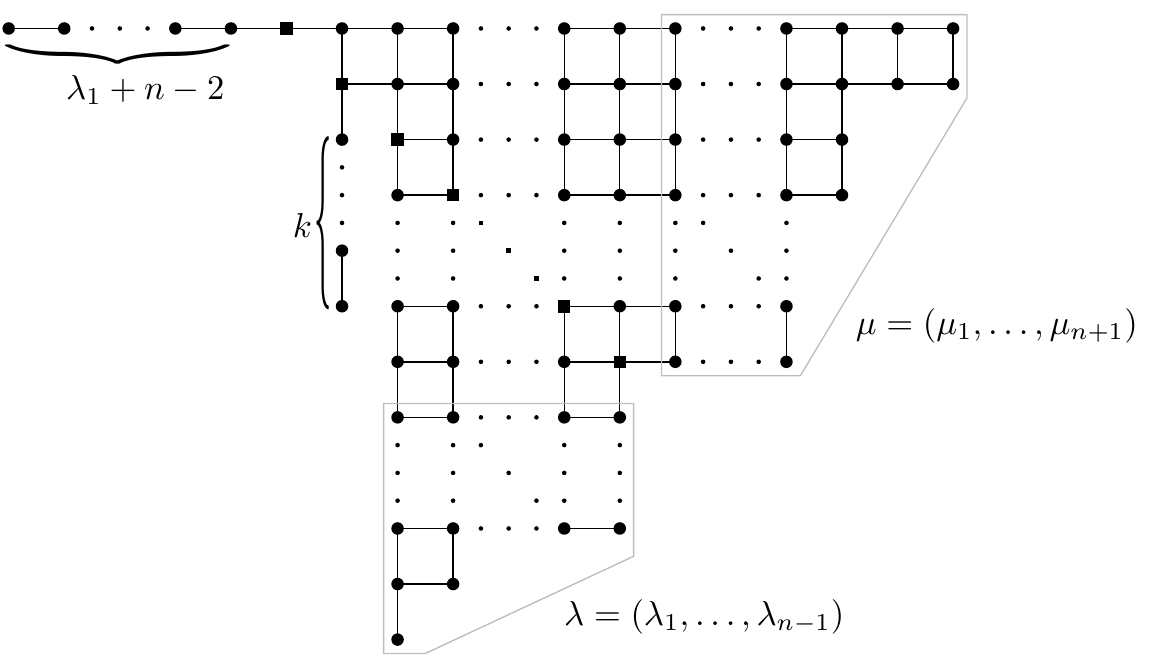}  
  \caption{(Insets) A semi-irreducible $d$-complete poset of class $4$, $P_{n+1} ^{\lambda_1 +n-2}(X_4)$  for $n\ge2$, $k\ge0$, $\lambda\in\Par_{n-1},\mu\in\Par_{n+1}$ and $X_4= \{(\lambda, n-1, 1), (\mu, n+1,1), ((k),1,n)\}$. This is irreducible if and only if $k=0$ and $\mu_1=\mu_2$.}
  \label{fig:class4}
\end{figure}

% class 5
\begin{figure}[H]
  \centering
\includegraphics{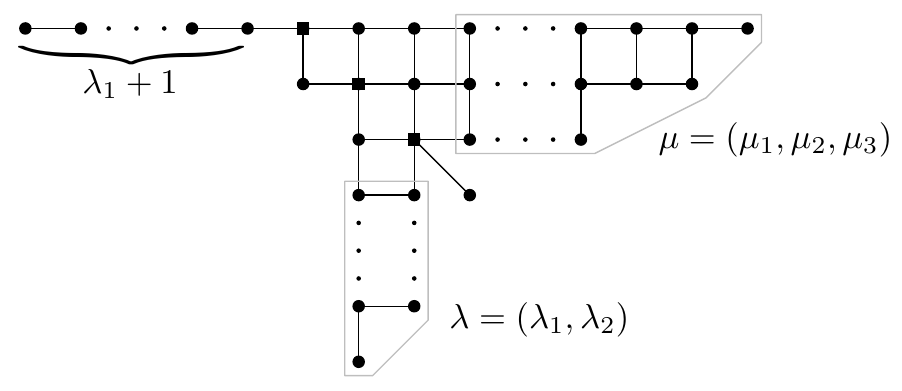}  
   \caption{(Tailed insets) A semi-irreducible $d$-complete poset of class $5$, $P_{3} ^{\lambda_1 +1}(X_5)$ for $\lambda\in\Par_2,\mu\in\Par_3$ and $X_5 =\{ (\lambda,2,1), (\mu,3,1),(\emptyset,2,2),((1),1,1)\}$. This is irreducible if and only if $\mu_1=\mu_2$.}
  \label{fig:class5}
\end{figure}

% class 6
\begin{figure}[H]
  \centering
\includegraphics{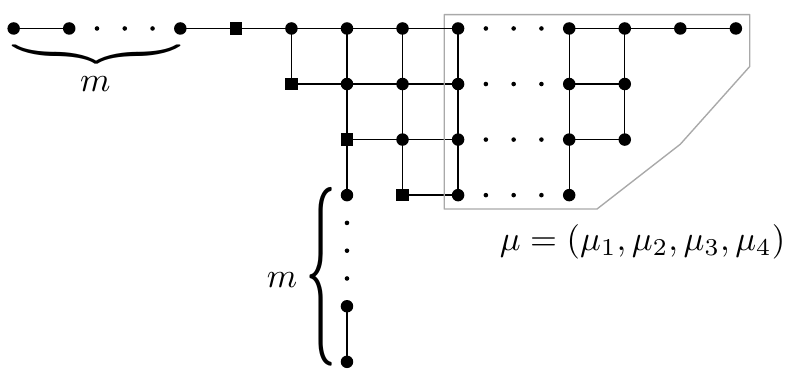}  
   \caption{(Banners) A semi-irreducible $d$-complete poset of class $6$, $P_{4} ^{m}(X_6)$ for $m\ge0$ and $\mu\in\Par_4$ with $\mu_4\ge1$ and $X_6 =\{ (\mu,4,1),((m),1,2)\}$. This is irreducible if and only if $\mu_1=\mu_2$.}
  \label{fig:class6}
\end{figure}

% class 7
\begin{figure}[H]
  \centering
\includegraphics{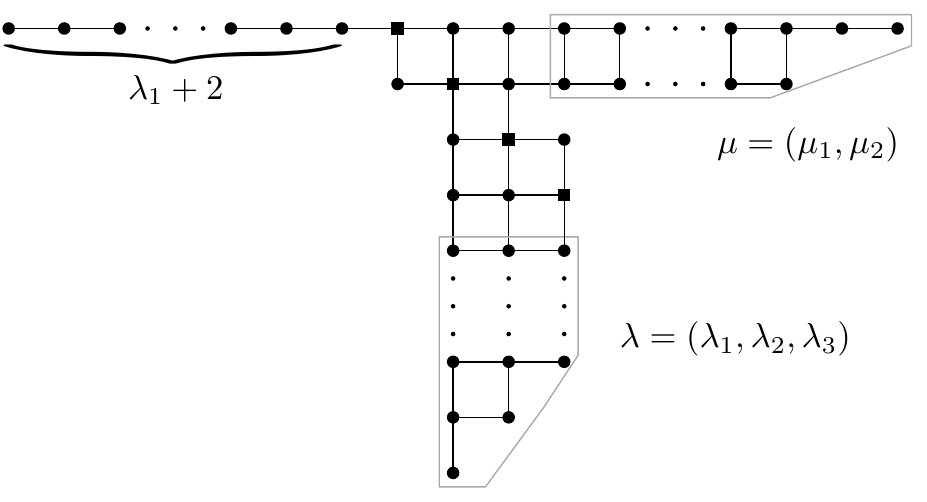}  
   \caption{(Nooks) A semi-irreducible $d$-complete poset of class $7$, $P_{4} ^{\lambda_1 +2}(X_7)$ for $\lambda\in\Par_3,\mu\in\Par_2$ and $X_7 = \{ (\lambda,3,1), (\emptyset, 2,1), (\mu,3,2), (\emptyset,2,3)\}$. This is irreducible if and only if $\mu_1=\mu_2$.}
  \label{fig:class7}
\end{figure}

% class 8-1
\begin{figure}[H]
  \centering
\includegraphics{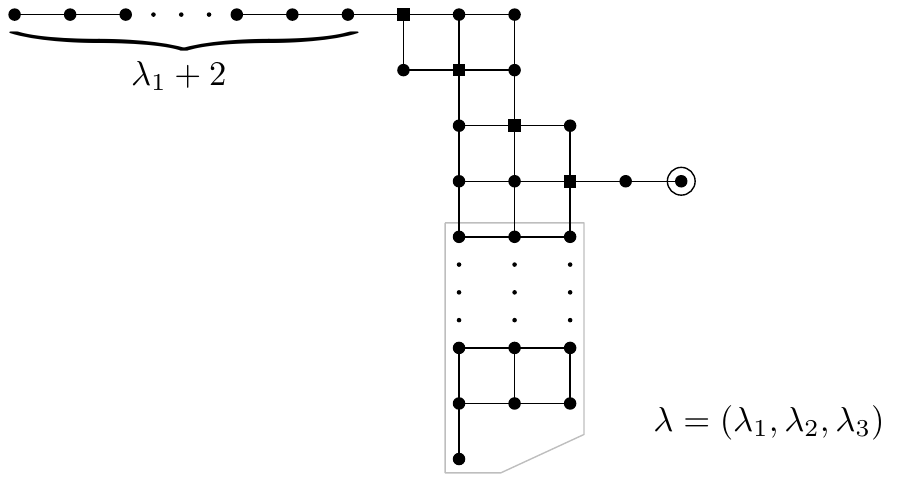}  
   \caption{(Swivels-1) A semi-irreducible $d$-complete poset of class $8$-(1), $P_{4} ^{\lambda_1 +2}(X_8^{(1)})$ for $\lambda\in\Par_3$ and $X_8^{(1)} = \{ (\lambda, 3,1), ((2), 1,1), (\emptyset, 2,1), (\emptyset, 3,2), (\emptyset, 2,3)\}$. This poset is always irreducible.}
  \label{fig:class8-1}
\end{figure}

% class 8-2
\begin{figure}[H]
  \centering
\includegraphics{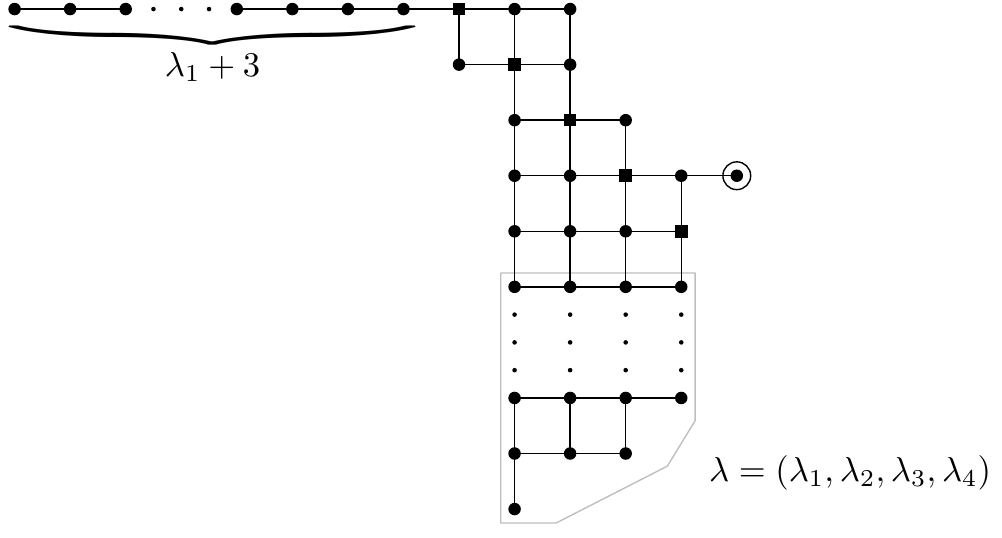}  
  \caption{(Swivels-2) A semi-irreducible $d$-complete poset of class $8$-(2), $P_{5} ^{\lambda_1 +3}(X_8^{(2)})$  for $\lambda\in\Par_4$ and $X_8^{(2)} = \{ (\lambda, 4,1), ((1,0), 2,1), (\emptyset, 2,2), (\emptyset, 3,3), (\emptyset, 2,4)\}$. This poset is always irreducible.}
  \label{fig:class8-2}
\end{figure}

% class 8-3
\begin{figure}[H]
  \centering
\includegraphics{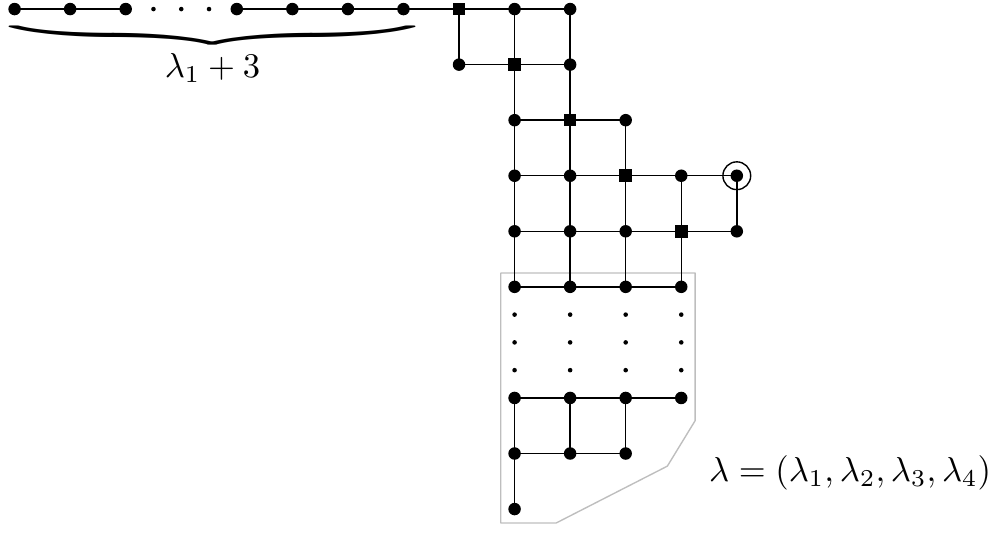}  
  \caption{(Swivels-3) A semi-irreducible $d$-complete poset of class $8$-(3), $P_{5} ^{\lambda_1 +3}(X_8^{(3)})$  for $\lambda\in\Par_4$ and $X_8^{(3)} = \{ (\lambda, 4,1), ((1,1), 2,1), (\emptyset, 2,2), (\emptyset, 3,3), (\emptyset, 2,4)\}$. This poset is always irreducible.}
  \label{fig:class8-3}
\end{figure}

% class 8-4
\begin{figure}[H]
  \centering
\includegraphics{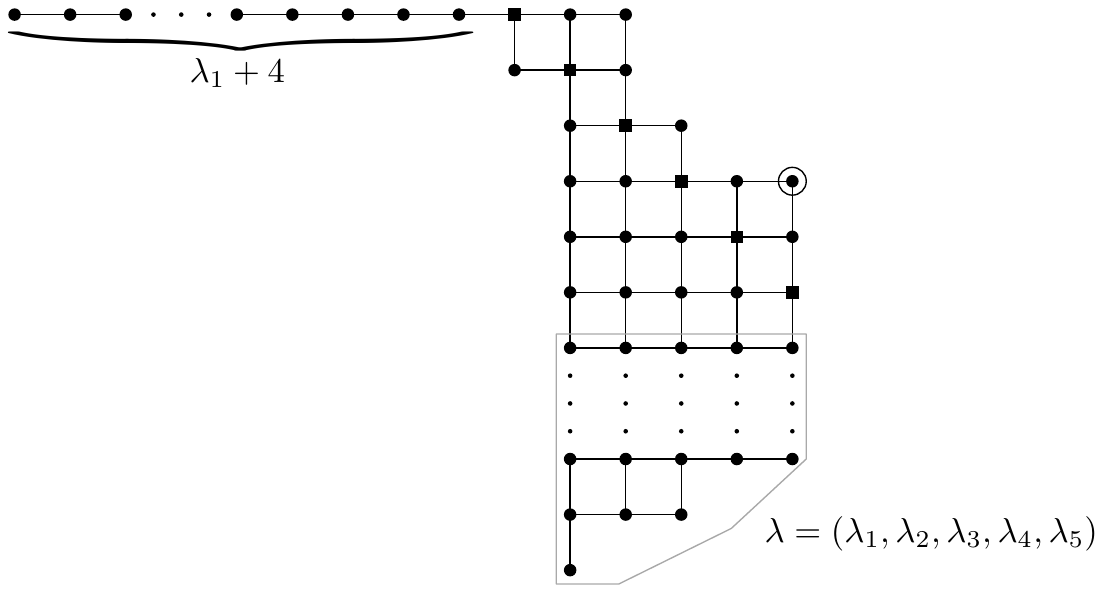}  
  \caption{(Swivels-4) A semi-irreducible $d$-complete poset of class $8$-(4), $P_{6} ^{\lambda_1 +4}(X_8^{(4)})$  for $\lambda\in\Par_5$ and $X_8^{(4)} = \{ (\lambda, 5,1), (\emptyset, 3,1), (\emptyset, 2,3), (\emptyset, 3,4), (\emptyset, 2,5)\}$. This poset is always irreducible.}
  \label{fig:class8-4}
\end{figure}

% class 9-1
\begin{figure}[H]
  \centering
\includegraphics{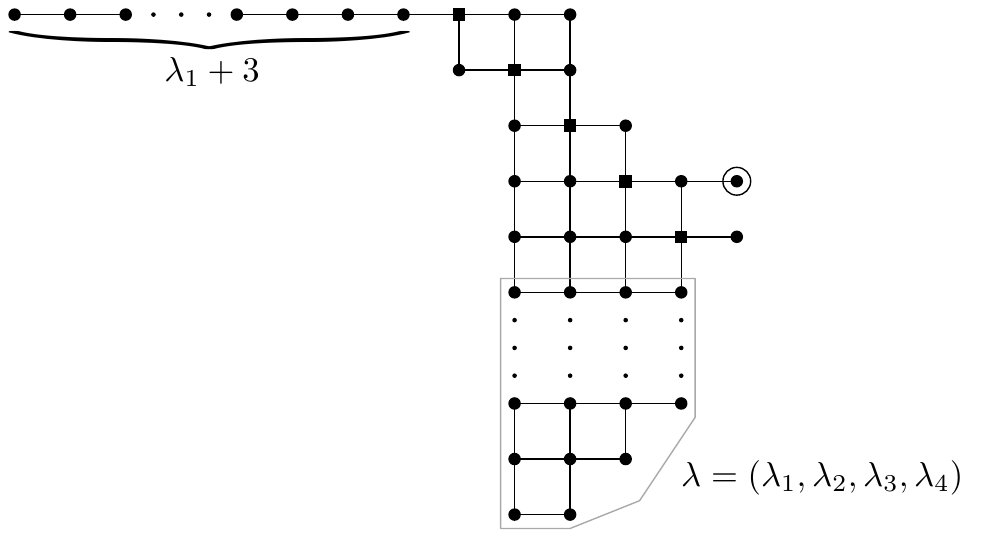}  
  \caption{(Tailed swivels-1) A semi-irreducible $d$-complete poset of class $9$-(1), $P_{5} ^{\lambda_1 +3}(X_9^{(1)})$ for $\mu\in\Par_4$ and $X_9^{(1)} = \{(\lambda, 4,1), ((1),1,1), ((1,0),2,1), (\emptyset, 2,2), (\emptyset,3,3),(\emptyset, 2,4)\}$. This poset is always irreducible.}
  \label{fig:class9-1}
\end{figure}

% class 9-2
\begin{figure}[H]
  \centering
\includegraphics{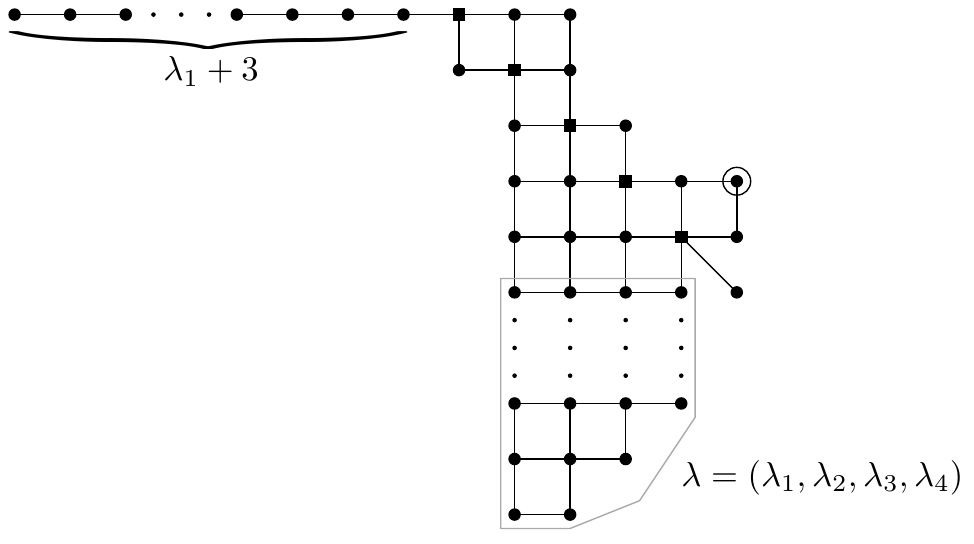}  
  \caption{(Tailed swivels-2) A semi-irreducible $d$-complete poset of class $9$-(2), $P_{5} ^{\lambda_1 +3}(X_9^{(2)})$ for $\lambda\in\Par_4$ and $X_9^{(2)} = \{(\lambda, 4,1), ((1),1,1), ((1,1),2,1), (\emptyset, 2,2), (\emptyset,3,3),(\emptyset, 2,4)\}$. This poset is always irreducible.}
  \label{fig:class9-2}
\end{figure}

% class 10
\begin{figure}[H]
  \centering
\includegraphics{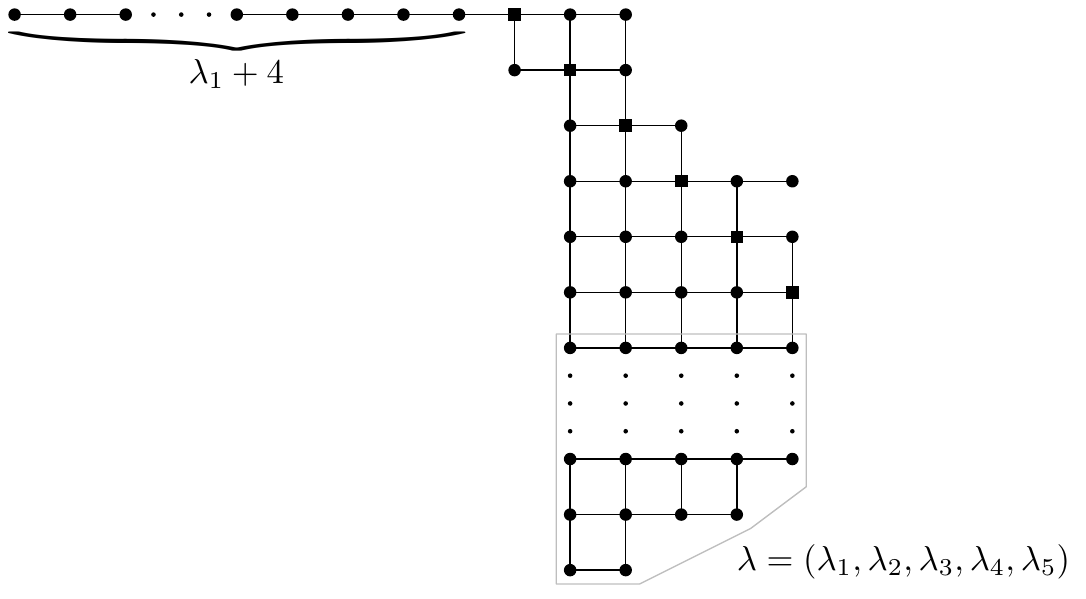}  
   \caption{(Tagged swivels) A semi-irreducible $d$-complete poset of class $10$, $P_{6} ^{\lambda_1 +4}(X_{10})$ for $\lambda\in\Par_5$ and $X_{10}=\{ (\lambda,5,1),(\emptyset ,2,1),((1),2,2),(\emptyset, 2,3),(\emptyset, 3,4),(\emptyset, 2,5) \}$. This poset is always irreducible.}
  \label{fig:class10}
\end{figure}

% class 11
\begin{figure}[H]
  \centering
\includegraphics{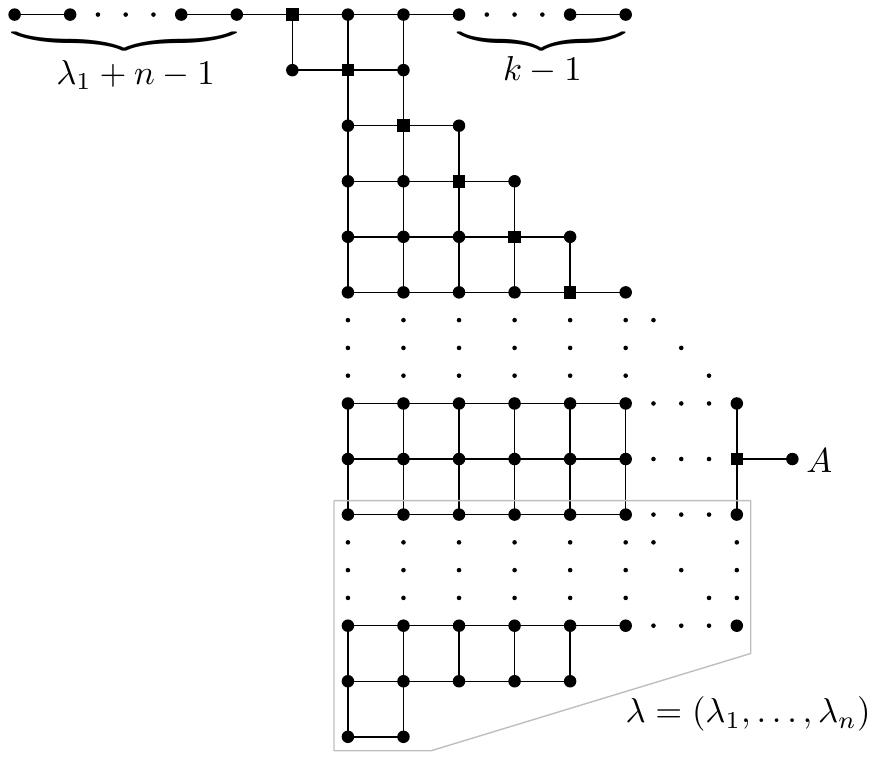}  
  \caption{(Swivel shifteds) A semi-irreducible $d$-complete poset of class $11$, $P_{n+1} ^{\lambda_1 +n-1}(X_{11})$ for $n\ge3$, $k\ge1$ and $\lambda\in\Par_n$, $\epsilon\in\{0,1\}$ and
$X_{11}=\{ (\lambda,n,1),((k-1),3,n-1),(\emptyset, 2, n),((\epsilon),1,1)\}\bigcup \bigcup_{i=1}^{n-2}\{ (\emptyset, 2, i)\}$. The element with label $A$ may or may not be in the poset depending on $\epsilon$. This poset is irreducible if and only if $k=1$.}
  \label{fig:class11}
\end{figure}

% class 12
\begin{figure}[H]
  \centering
\includegraphics{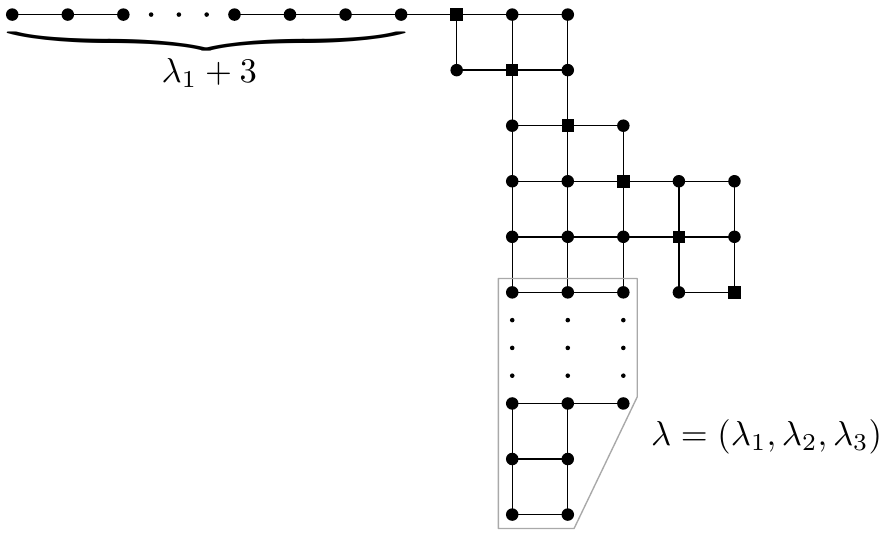}  
  \caption{(Pumps) A semi-irreducible $d$-complete poset of class $12$, $P_{6} ^{\lambda_1 +3}(X_{12})$ for $\lambda\in\Par_3$ and $X_{12}=\{(\emptyset, 3,1),(\emptyset,2,1),(\lambda,4,2),(\emptyset,2,3),(\emptyset,3,4),(\emptyset,2,5) \}$. This poset is always irreducible.}
  \label{fig:class12}
\end{figure}

% class 13
\begin{figure}[H]
  \centering
\includegraphics{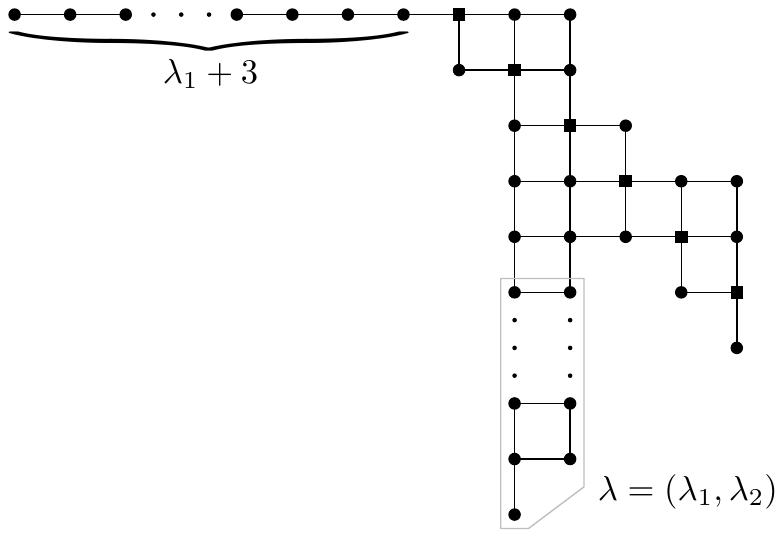}  
  \caption{(Tailed pumps) A semi-irreducible $d$-complete poset of class $13$, $P_{6} ^{\lambda_1 +3}(X_{13})$ for $\lambda\in\Par_2$ and $X_{13}=\{((1),1,1),(\emptyset, 2,1),(\emptyset,3,1),(\lambda,4,2),(\emptyset,2,3),(\emptyset,3,4),(\emptyset,2,5) \}$. This poset is always irreducible.}
  \label{fig:class13}
\end{figure}

% class 14
\begin{figure}[H]
  \centering
\includegraphics{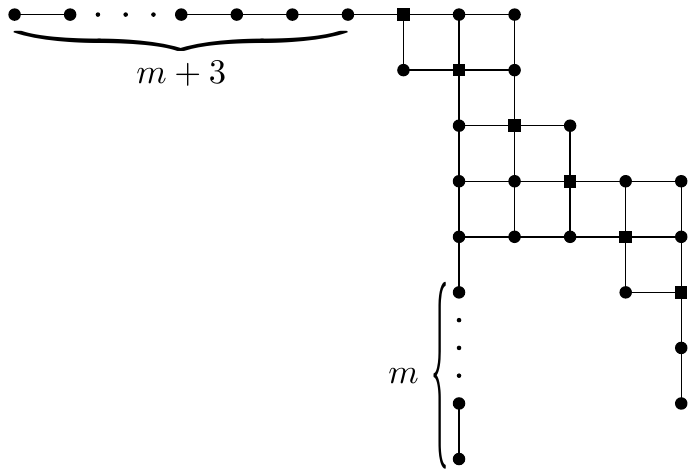}  
  \caption{(Near bats) A semi-irreducible $d$-complete poset of class $14$, $P_{6} ^{m+3}(X_{14})$ for $m\ge0$ and $X_{14}=\{((2),1,1),(\emptyset,2,1),(\emptyset,3,1),((m),4,2),(\emptyset,2,3),(\emptyset,3,4),(\emptyset,2,5) \}$. This poset is always irreducible.} 
  \label{fig:class14}
\end{figure}

% class 15
\begin{figure}[H]
  \centering
\includegraphics{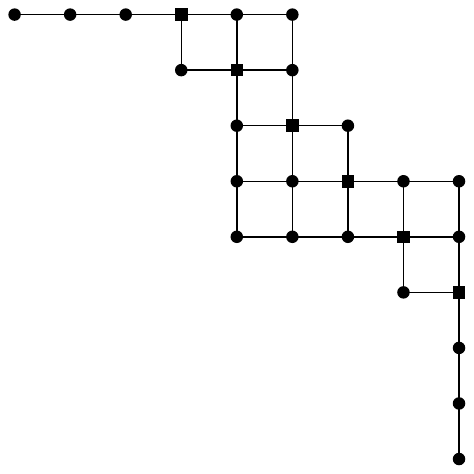}  
  \caption{(Bat) The unique semi-irreducible $d$-complete poset of class $15$, $P_{6} ^3(X_{15})$ for $X_{15}=\{((3),1,1),(\emptyset,2,1),(\emptyset,3,1),(\emptyset,4,2),(\emptyset,2,3),(\emptyset,3,4),(\emptyset,2,5) \}$. This poset is irreducible.}
  \label{fig:class15}
\end{figure}

%-----------------------------------------------------------------------

\section*{Acknowledgement}
The authors would like to thank Michael Schlosser for many useful comments on partial fraction expansion identities. They are also grateful to Robert Proctor and Ole Warnaar for careful reading of the manuscript and helpful comments. They appreciate the anonymous referees for their valuable comments that improved the presentation of this paper.


\begin{thebibliography}{10}

\bibitem{AAR}
G.~E. Andrews, R. Askey, and R.~Roy. 
\newblock Special Functions, volume 71 of Encyclopedia of Mathematics and its Applications, 1999.

\bibitem{Frame1954}
J.~S. Frame, G.~d.~B. Robinson, and R.~M. Thrall.
\newblock The hook graphs of the symmetric groups.
\newblock {\em Canadian J. Math.}, 6:316--324, 1954.

\bibitem{IshikawaTagawa07}
M.~Ishikawa and H.~Tagawa.
\newblock Schur Function Identities and Hook Length Posets.
\newblock 19th International Conference on Formal Power Series and Algebraic Combinatorics, Nankai University, Tianjin, China, 2007

\bibitem{IshikawaTagawa}
M.~Ishikawa and H.~Tagawa.
\newblock Leaf posets and multivariate hook length property.
\newblock {\em RIMS K\^oky\^uroku} 1913:67--80, 2014.

\bibitem{KimStanton17}
J.~S. Kim and D.~Stanton.
\newblock On {$q$}-integrals over order polytopes.
\newblock {\em Adv. Math.}, 308:1269--1317, 2017.

\bibitem{Milne1988}
S.~Milne.
\newblock A $q$-analog of the {G}auss summation theorem for hypergeometric
  series in ${\rm U}(n)$.
\newblock {\em Advances in Mathematics}, 72(1):59--131, 1988.

\bibitem{MPP}
A. Morales, I. Pak, and G. Panova.
\newblock Hook formulas for skew shapes I. $q$-analogues and bijections.
\newblock J. Combin. Theory Ser. A 154 (2018), 350–405.

\bibitem{Nakada2009}
K.~Nakada.
\newblock {$q$}-hook formula of {G}ansner type for a generalized {Y}oung
  diagram.
\newblock In {\em 21st {I}nternational {C}onference on {F}ormal {P}ower
  {S}eries and {A}lgebraic {C}ombinatorics ({FPSAC} 2009)}, Discrete Math.
  Theor. Comput. Sci. Proc., AK, pages 685--696. Assoc. Discrete Math. Theor.
  Comput. Sci., Nancy, 2009.

\bibitem{Nakada}
K.~Nakada.
\newblock $q$-hook formula for a generalized {Y}oung diagram.
\newblock in preparation.

\bibitem{Naruse}
H.~Naruse.
\newblock Schubert calculus and hook formula.
\newblock talk slides at 73rd S\'em. Lothar. Combin., Strobl, Austria, 2014, available at 
\url{http://www.mat.univie.ac.at/~slc/wpapers/s73vortrag/naruse.pdf}.

\bibitem{NaruseOkada}
H.~Naruse and S.~Okada.
\newblock Skew hook formula for $d$-complete posets.
\newblock \url{https://arxiv.org/abs/1802.09748}.

\bibitem{Proctor1999}
R.~A. Proctor.
\newblock Dynkin diagram classification of {$\lambda$}-minuscule {B}ruhat
  lattices and of {$d$}-complete posets.
\newblock {\em J. Algebraic Combin.}, 9(1):61--94, 1999.

\bibitem{Proctor1999a}
R.~A. Proctor.
\newblock Minuscule elements of {W}eyl groups, the numbers game, and
  {$d$}-complete posets.
\newblock {\em J. Algebra}, 213(1):272--303, 1999.

\bibitem{Proctor2014}
R.~A. Proctor.
\newblock $d$-complete posets generalize {Y}oung diagrams for the hook product
  formula: {P}artial {P}resentation of {P}roof.
\newblock {\em RIMS K\^oky\^uroku}, 1913:120--140, 2014.

\bibitem{Proctor2017}
R.~A.~Proctor and L.~M.~Scoppetta.
\newblock $d$-complete posets: local structural axioms, properties, and
  equivalent definitions.
\newblock \url{https://arxiv.org/abs/1704.05792}.


\bibitem{Rainville71}
E.~D.~Rainville.
\newblock Special functions.
\newblock Chelsea Publishing Co., Bronx, N.Y., 1971.


\bibitem{Rosengren04}
H.~Rosengren.
\newblock Elliptic hypergeometric series on root systems.
\newblock {\em Advances in Mathematics}, 181(2):417--447, 2004.



\bibitem{EC1}
R.~P. Stanley.
\newblock Enumerative Combinatorics. {V}ol. 1, second ed.
\newblock Cambridge University Press, New York/Cambridge, 2011.

\bibitem{sagemath}
The Sage Developers.
\newblock SageMath, the Sage Mathematics Software System (Version 7.5.1).
\newblock 2017, http://www.sagemath.org.

\bibitem{Warnaar}
S.~O.~Warnaar.
\newblock $q$-Selberg integrals and Macdonald polynomials.
\newblock {\em Ramanujan J.}, 10(2):237--268, 2005. 


\bibitem{WW}
E.~T.~Whittaker and G.~N.~Watson.
\newblock  A Course of Modern Analysis, 4th ed.
\newblock Cambridge University Press, Cambridge, 1927 (reprinted 1996).

\end{thebibliography}
\end{document}